\documentclass[11pt]{article}
\usepackage{mathtools, amssymb,amsfonts,amsthm,amscd,amsmath,amsgen,upref,vruler,enumerate,nicefrac,lipsum,footnote,footmisc,hyperref}

\usepackage[margin=1in]{geometry}

\theoremstyle{plain}
\newtheorem{cor}{Corollary}
\newtheorem{lem}[cor]{Lemma}
\newtheorem{prop}[cor]{Proposition}
\newtheorem*{prop*}{Proposition}

\newtheorem{thm}[cor]{Theorem}
\newtheorem*{thm*}{Theorem}
\theoremstyle{definition}
\newtheorem{definition}[cor]{Definition}
\newtheorem{remark}[cor]{Remark}

\numberwithin{cor}{section}
\numberwithin{equation}{section}

\DeclareMathOperator{\C}{C}

\DeclareMathOperator{\dd}{d}
\DeclareMathOperator{\rank}{rank}
\DeclareMathOperator{\Hess}{Hess}  
\DeclareMathOperator{\Proj}{Proj}
\DeclareMathOperator{\codim}{codim}

\newcommand{\abs}[1]{\left|#1\right|}
\newcommand{\norm}[1]{\left\|#1\right\|}
\providecommand{\ud}[1]{\, \mathrm{d} #1}
\providecommand{\dx}{\ud{x}}

\providecommand{\ds}{\ud{s}}
\providecommand{\dt}{\ud{t}}

\providecommand{\dd}{\ud}

\def\XXint#1#2#3{{\setbox0=\hbox{$#1{#2#3}{\int}$ }
\vcenter{\hbox{$#2#3$ }}\kern-.6\wd0}}

\title{Convergence rates for the stochastic gradient descent \\ method  for non-convex objective functions}
\author{Benjamin Fehrman$^1$, Benjamin Gess$^2$, and Arnulf Jentzen$^3$
	\bigskip
	\\
	\small{$^1$Mathematical Institute, University of Oxford,}\\
	\small{Oxford, United Kingdom,}\\
	\small{ e-mail:  
	benjamin.fehrman@maths.ox.ac.uk}\\
	\smallskip
	\\
	\small{$^2$ Max Planck Institute for Mathematics in the Sciences,}\\
	\small{Leipzig, Germany,}\\
	\small{Fakult\"at f\"ur Mathematik, Universit\"at Bielefeld,}\\
	\small{Bielefeld, Germany,}\\
	\small{e-mail:  benjamin.gess@mis.mpg.de}\\
	\smallskip
	\\
	\small{$^3$Seminar for Applied Mathematics, Department of Mathematics,}\\
	\small{ETH Zurich, Zurich, Switzerland,}\\
	\small{e-mail:   arnulf.jentzen@sam.math.ethz.ch}
	}

\date{\today}

\begin{document}

\maketitle

 \begin{abstract}
     We prove the local convergence to minima and estimates on the rate of convergence for the stochastic gradient descent method in the case of not necessarily globally convex nor contracting objective functions. In particular, the results are applicable to simple objective functions arising in machine learning.
 \end{abstract}

 \tableofcontents
 

 \section{Introduction}\label{introduction}
 
 Stochastic gradient descent algorithms (SGD), going back to \cite{RM85}, are the most common way to train neural networks. Despite their relevance to machine learning and much recent interest, estimates on their rate of convergence have so far only been shown under global contraction or convexity assumptions on the objective function that are often not satisfied by examples arising in machine learning. Indeed, citing from \cite{VBRS17}, ``While SGD has been rigorously analyzed only for convex loss functions [...], in deep learning the loss
 is a non-convex function of the network parameters, hence there are no guarantees that SGD finds the global minimizer.'' In the present work, we prove the \textit{local} convergence of SGD to the set of global minima of the objective function while avoiding such a global convexity or contractivity assumption. The relevance of the obtained results is demonstrated by the application to the training of (simple) neural networks. 
 
 Stochastic gradient descent methods are used to numerically minimize functions $f\colon  \mathbb{R}^{d}\to\mathbb{R}$ of the form 
 \begin{equation}\label{eq:intro-f}
  f(\theta)=\mathbb{E}\left[F(\theta,X)\right],
 \end{equation}
 for some product measurable function $F\colon  \mathbb{R}^{d}\times\mathbb{R}^{m}\rightarrow\mathbb{R}$ and some random variable $X\colon\Omega\rightarrow\mathbb{R}^m$ on some probability space $(\Omega,\mathcal{F},\mathbb{P})$. The analysis of SGD has attracted considerable attention in the literature (cf., e.g., \cite{B14,BM13,BCN18,DMG15,JKNW18,MB11,TM15} and the references therein). In \cite{DMG15,JKNW18}, the convergence of SGD with rates assuming the following contraction property for the objective function $f$, which is classical in stochastic approximation theory, was analyzed: There is an $L>0$ and a zero $\theta^*$ of $\nabla_\theta f$ such that
for every $\theta\in\mathbb{R}^d$ it holds that \begin{equation}\label{eq:intro-contraction}
   (-\nabla_\theta f(\theta),\theta-\theta^*) \le -L \|\theta-\theta^*\|^2.
 \end{equation}
In particular, this contraction property implies the uniqueness of the zero $\theta^*$ of $\nabla_\theta f$ and thus the uniqueness of local minima of $f$. This is in stark contrast to actual objective functions arising in the training of neural networks which are expected to show rich sets of local minima and saddle points/plateaus. Consequently, it is vital for the application to machine learning to avoid such global contraction assumptions. In addition, for example due to the positive homogeneity of the ReLU function, the objective functions typically satisfy certain symmetries, implying that global (and local) minima are not isolated points nor unique, but form (possibly non-compact) manifolds. Indeed, this is demonstrated for simple neural networks in Section \ref{sec_applications} below. We are therefore led to the task of analyzing the convergence properties of SGD locally at sets of minima\footnote{We emphasize that this is disjoint from the recent works \cite{BCN18,LO18,WWB18} where the global convergence of the \textit{gradient} of the objective function to zero has been shown for SGD and AdaGrad. This does not imply the local convergence to minima, since the gradient also vanishes in saddles/plateaus.}. In the present work we provide estimates on the rate of convergence for SGD under assumptions avoiding a contraction property like \eqref{eq:intro-contraction}.

 \begin{thm}\label{intro_ng_one_path}  Let $d\in\mathbb{N}$, $\mathfrak{d}\in\{ 0, 1, \ldots, d - 1 \}$, $\rho\in(\nicefrac{2}{3},1)$, let $\abs{\cdot}\colon\mathbb{R}^d\rightarrow\mathbb{R}$ be the standard norm on $\mathbb{R}^d$, let $U\subseteq\mathbb{R}^d$ be an open set, let $A\subseteq\mathbb{R}^d$ be a bounded open set, let $(\Omega,\mathcal{F},\mathbb{P})$ be a probability space, let $(S,\mathcal{S})$ be a measurable space, let $F=(F(\theta,x))_{(\theta,x)\in\mathbb{R}^d\times S}\colon  \mathbb{R}^d\times S\rightarrow\mathbb{R}$ be a measurable function, let $X_{k,n,m}\colon  \Omega\rightarrow S$, $k,n,m\in\mathbb{N}$, be i.i.d.\ random variables which satisfy for every $\theta\in\mathbb{R}^d$ that $\mathbb{E}\big[ |F(\theta,X_{1,1,1})|^2\big]<\infty$, let $f\colon\mathbb{R}^d\rightarrow\mathbb{R}$ be the function which satisfies for every $\theta\in\mathbb{R}^d$ that $f(\theta)=\mathbb{E}\big[F(\theta,X_{1,1,1})\big]$, let $\mathcal{M}\subseteq\mathbb{R}^d$ satisfy that
\begin{equation}\mathcal{M}=\big\{\theta\in\mathbb{R}^d\colon [f(\theta)=\inf\nolimits_{\vartheta\in\mathbb{R}^d} f(\vartheta)]\big\},\end{equation}
assume for every $x\in S$ that $\mathbb{R}^d\ni\theta\mapsto F(\theta,x)\in\mathbb{R}$ is a continuously differentiable function, assume that $f|_U\colon U\rightarrow\mathbb{R}$ is a three times continuously differentiable function, assume for every non-empty compact set $\mathfrak{C}\subseteq U$ that $\sup\nolimits_{\theta\in \mathfrak{C}}\mathbb{E}\big[|F(\theta,X_{1,1,1})|^2+|(\nabla_\theta F)(\theta,X_{1,1,1})|^2\big]<\infty$, assume that $\mathcal{M}\cap U$ is a $\mathfrak{d}$-dimensional $\C^1$-submanifold of $\mathbb{R}^d$, assume that $\mathcal{M}\cap U\cap A\neq\emptyset$, assume for every $\theta\in(\mathcal{M}\cap U)$ that $\rank((\Hess f)(\theta))=d-\mathfrak{d}$, for every $n,M\in\mathbb{N}$, $r\in(0,\infty)$ let $\Theta^{k,M,r}_n\colon\Omega\rightarrow\mathbb{R}^d$, $k\in\mathbb{N}$, be i.i.d.\ random variables, assume for every $M\in\mathbb{N}$, $r\in(0,\infty)$ that $\Theta^{1,M,r}_{0}$ is continuous uniformly distributed on $A$, assume for every $M\in\mathbb{N}$, $r\in(0,\infty)$ that $(\Theta^{k,M,r}_0)_{k\in\mathbb{N}}$ and $(X_{k,n,m})_{k,n,m\in\mathbb{N}}$ are independent, assume for every $k,n,M\in\mathbb{N}$, $r\in(0,\infty)$ that
\begin{equation}\label{intro_stoch_mean_10}\Theta^{k,M,r}_{n}=\Theta^{k,M,r}_{n-1}-\frac{r}{n^\rho M}\!\left[\sum_{m=1}^M(\nabla_\theta F)(\Theta^{k,M,r}_{n-1},X_{k,n,m})\right],\end{equation}
and for every $n,M,\mathfrak{M},K\in\mathbb{N}$, $r\in(0,\infty)$ let $\varTheta^{K,M,\mathfrak{M},r}_n\colon\Omega\rightarrow\mathbb{R}^d$ be a random variable which satisfies that
\begin{equation}\sum_{m=1}^\mathfrak{M} F(\varTheta^{K,M,\mathfrak{M},r}_n, X_{1,n+1,m})=\min_{k\in\{1,2\ldots,K\}}\left[\sum_{m=1}^\mathfrak{M}F(\Theta^{k,M,r}_n, X_{1,n+1,m})\right],\end{equation}
(cf.\ Lemma~\ref{lem_measurable_selection} below).  Then there exist $\mathfrak{r},c\in(0,\infty)$, $\kappa\in[0,1)$ such that for every $n,M,\mathfrak{M},K\in\mathbb{N}$, $r\in(0,\mathfrak{r}]$, $\varepsilon\in(0,1]$ it holds that
\begin{equation}\label{intro_ref} \mathbb{P}\Big( \Big[f(\varTheta^{K,M,\mathfrak{M},r}_n)-\inf\nolimits_{\theta\in\mathbb{R}^d}f(\theta)\Big]\geq\varepsilon\Big) \leq \frac{cK}{\varepsilon^2\mathfrak{M}}+\left[\kappa+c\left(\frac{1}{\varepsilon^{2}n^{\rho}}+\frac{n^{1-\rho}}{M^{\nicefrac{1}{2}}}\right)\right]^{K}.\end{equation}
\end{thm}

Theorem~\ref{intro_ng_one_path} is an immediate consequence of Theorem~\ref{thm_intro_ng_one_path} in Section~\ref{sec_nonglobal} below.  The statement of Theorem~\ref{intro_ng_one_path} should be interpreted in the following way.  We aim to minimize an objective function $f\colon\mathbb{R}^d\rightarrow\mathbb{R}$, where we assume that the set of minima
\begin{equation}\mathcal{M}=\{\theta\in\mathbb{R}^d\colon f(\theta)=\big[\inf\nolimits_{\vartheta\in\mathbb{R}^d}f(\vartheta)\big]\},\end{equation}
is somewhere locally smooth in the sense that there exists an open set $U\subseteq\mathbb{R}^d$ such that
\begin{equation} \mathcal{M}\cap U\;\textrm{is a non-empty $\mathfrak{d}$-dimensional $\C^1$-submanifold of $\mathbb{R}^d$.}\end{equation}
We furthermore assume that $f$ is locally $\C^3$ in a neighborhood of $\mathcal{M}\cap U$ and that the Hessian is maximally nondegenerate on $\mathcal{M}\cap U$ in the sense that for every $\theta\in(\mathcal{M}\cap U)$ it holds that
\begin{equation}\rank\big(\big(\Hess f\big)(\theta)\big)=d-\mathfrak{d}=\codim(\mathcal{M}\cap U).\end{equation}

Let $(\Omega,\mathcal{F},\mathbb{P})$ be a probability space, let $(S,\mathcal{S})$ be a measurable space, and let $X_{k,n,m}\colon\Omega\rightarrow S$, $k,n,m\in\mathbb{N}$, be i.i.d.\ random variables.  We assume that there exists a measurable function $F\colon S\times\mathbb{R}^d\rightarrow\mathbb{R}$ which satisfies for every $\theta\in\mathbb{R}^d$ that
\begin{equation} f(\theta)=\mathbb{E}\big[F(\theta, X_{1,1,1})\big].\end{equation}
In particular, since it is oftentimes the case in practice that the deterministic gradient $\nabla f(\theta)$ cannot be computed or cannot be efficiently computed, the random gradient $\nabla_\theta F(\theta,X_{1,1,1})$ provides an efficiently computable stochastic approximation.
 
The initial data of SGD is sampled from a bounded open set $A\subseteq\mathbb{R}^d$ which satisfies that $\mathcal{M}\cap U\cap A\neq\emptyset$.  That is, for every mini-batch size $M\in\mathbb{N}$ and $r\in(0,\infty)$, the initial data $\Theta^{k,M,r}_0\colon\Omega\rightarrow\mathbb{R}^d$, $k\in\mathbb{N}$, are uniformly distributed on $A$, independent, and independent of the driving noise $X_{k,n,m}$, $k,n,m\in\mathbb{N}$.  We then compute independent solutions to SGD in the sense that for every $k,n\in\mathbb{N}$ it holds that
\begin{equation}\Theta^{k,M,r}_{n}=\Theta^{k,M,r}_{n-1}-\frac{r}{n^\rho M}\!\left[\sum_{m=1}^M(\nabla_\theta F)(\Theta^{k,M,r}_{n-1},X_{k,n,m})\right].\end{equation}
For a fixed terminal time $n\in\mathbb{N}$, for a sampling size $K\in\mathbb{N}$, the output of the algorithm at this point is the collection of values $\Theta^{k,M,r}_n$, $k\in\{1,2,\ldots,K\}$.  It remains to identify the value $\Theta^{k,M,r}_n$, $k\in\{1,2,\ldots,K\}$, that minimizes the objective function.

Much as in the case of the gradient, since the objective function cannot be practically computed, for a terminal time $n\in\mathbb{N}$, for a mini-batch size $\mathfrak{M}\in\mathbb{N}$, we introduce the mini-batch approximation $F^{K,\mathfrak{M},n}\colon\mathbb{R}^d\times\Omega\rightarrow\mathbb{R}$ which satisfies for every $(\theta,\omega)\in\mathbb{R}^d\times\Omega$ that
\begin{equation}F^{K,\mathfrak{M},n}(\theta,\omega)=\frac{1}{\mathfrak{M}}\sum_{m=1}^\mathfrak{M}F(\theta, X_{1,n+1,m}(\omega)).\end{equation}
We then identify the value $\Theta^{k,M,r}_n$, $k\in\{1,\ldots,K\}$, that minimizes $F^{K,\mathfrak{M},n}$ in the sense that we compute a random variable $\varTheta^{K,M,\mathfrak{M},r}_n\colon\Omega\rightarrow\mathbb{R}^d$ which satisfies that
\begin{equation}\sum_{m=1}^\mathfrak{M}F(\varTheta^{K,M,\mathfrak{M},r}_n,X_{1,n+1,m})=\min_{k\in\{1,2,\ldots,K\}}\left[\sum_{m=1}^\mathfrak{M}F(\Theta^{k,M,r}_n, X_{1,n+1,m})\right].\end{equation}

The conclusion of Theorem~\ref{intro_ng_one_path} estimates the probability that $\varTheta^{K,M,\mathfrak{M},r}_n$ is an $\varepsilon\in(0,1]$ minimizer of the objective function.  Precisely, there exist $\mathfrak{r},c\in(0,\infty)$, $\kappa\in[0,1)$ such that for every $n,M,\mathfrak{M},K\in\mathbb{N}$, $r\in(0,\mathfrak{r}]$, $\varepsilon\in(0,1]$ it holds that
\begin{equation}\label{explanation}\mathbb{P}\Big( \Big[f(\varTheta^{K,M,\mathfrak{M},r}_n)-\inf\nolimits_{\theta\in\mathbb{R}^d}f(\theta)\Big]\geq\varepsilon\Big) \leq \frac{cK}{\varepsilon^2\mathfrak{M}}+\left[\kappa+c\left(\frac{1}{\varepsilon^{2}n^{\rho}}+\frac{n^{1-\rho}}{M^{\nicefrac{1}{2}}}\right)\right]^{K}.\end{equation}
The limit $\mathfrak{M}\rightarrow\infty$ corresponds to computing the minimizer of $f$ exactly.  If this can be done efficiently, then the first term on the righthand side of \eqref{explanation} vanishes.

The constant $\kappa\in[0,1)$, which we compute precisely in Theorem~\ref{thm_intro_ng_one_path} below, quantifies two sources of error:  the probability that the initial condition lies outside of a basin of attraction and a portion of the probability that SGD beginning in a basin of attraction fails to converge.  In Remark~\ref{rho_sum} below and Section~\ref{stoch_discrete}, we prove that the restriction $\rho\in(\nicefrac{2}{3},1)$ can be extended to $\rho\in(0,1)$ under the additional assumption that $\mathcal{M}\cap U$ is a compact subset of $\mathbb{R}^d$.  Finally, it is not necessary to assume that $F$ is continuously differentiable, and this assumption can be replaced with the assumption that for every $x\in S$ we have that $F(\cdot,x)$ is a locally Lipschitz continuous function of $\theta\in\mathbb{R}^d$.

We observe that the computational efficiency of the algorithm can be estimated using Theorem~\ref{intro_ng_one_path}.  In particular, it follows from Corollary~\ref{cor_computation} below that there exist constants $c_i\in(0,\infty)$, $i\in\{1,2,3,4\}$, such that for every $\varepsilon,\eta\in(0,1]$, for $n(\varepsilon)\in\mathbb{N}_0, M(\varepsilon),\mathfrak{M}(\varepsilon,\eta), K(\eta)\in\mathbb{N}$ which satisfy that
\begin{equation}\label{intro_approx_cost_00}n(\varepsilon)=c_1\varepsilon^{-\nicefrac{2}{\rho}},\;\; M(\varepsilon)=c_2\varepsilon^{-\nicefrac{4}{\rho}+4},\;\;\mathfrak{M}(\varepsilon,\eta)=c_3\varepsilon^{-2}\eta^{-1}\abs{\log(\eta)},\;\;\textrm{and}\;\;K(\eta)=c_4\abs{\log(\eta)},\end{equation}
it holds that
\begin{equation}\label{intro_approx_cost_1} \mathbb{P}\Big(\big[f(\varTheta^{K(\eta),M(\varepsilon),\mathfrak{M}(\varepsilon,\eta),r}_{n(\varepsilon)})-\inf_{\theta\in\mathbb{R}^d}f(\theta)\big]\geq\varepsilon\Big) \leq\eta.\end{equation}
For every bounded open set $A\subseteq\mathbb{R}^d$ which satisfies that $\mathcal{M}\cap U\cap A$ is non-empty, for every $\varepsilon,\eta\in(0,1]$, the computational efficiency of the algorithm $\textrm{Eff}(\varepsilon,\eta;A)\in \mathbb{N}$ satisfies that
\begin{equation}\textrm{Eff}(\varepsilon,\eta;A)=\#\;\textrm{computations sufficient to ensure \eqref{intro_approx_cost_1}}.\end{equation}
It follows from \eqref{intro_approx_cost_00} that there exists $c\in(0,\infty)$ which satisfies for every $\varepsilon,\eta\in(0,1]$ that
\begin{equation}\textrm{Eff}(\varepsilon,\eta;A)\leq c\big(\varepsilon^{-2}\eta^{-1}\abs{\log(\eta)}+\varepsilon^{-\nicefrac{6}{\rho}+4}\abs{\log(\eta)}\big),\end{equation}
where the constant $c\in(0,\infty)$ depends on the computational cost of computing $F$ and $\nabla_\theta F$ but not on the running time $n\in\mathbb{N}$, mini-batch size $M\in\mathbb{N}$, or sampling size $K\in\mathbb{N}$.  Furthermore, we prove in Corollary~\ref{cor_computation_comp} below that that computational efficiency can be improved in the case that the local manifold of minima is compact.

The estimate of Theorem~\ref{intro_ng_one_path} quantifies two sources of error.  The first term on the righthand side of \eqref{intro_ref} quantifies the error introduced by the mini-batch approximation of the objective function.  In the case that the objective function $f$ can be efficiently computed, this error can be avoided by computing $\varTheta^{K,M,\infty,r}_n\colon\Omega\rightarrow\mathbb{R}^d$ which satisfies that
\begin{equation}f(\varTheta^{K,M,\infty,r}_n)=\Big[\min_{k\in\{1,2,\ldots,K\}}f(\Theta^{k,M,r}_n)\Big],\end{equation}
for which it follows from Corollary~\ref{obj_ng_one_path} below that
\begin{equation}\mathbb{P}\Big( \big[f(\varTheta^{K,M,\infty,r}_n)-\inf_{\theta\in\mathbb{R}^d}f(\theta)\big]\geq\varepsilon\Big) \leq \left[\kappa+c\left(\frac{1}{\varepsilon^{2}n^{\rho}}+\frac{n^{1-\rho}}{M^{\nicefrac{1}{2}}}\right)\right]^{K}.\end{equation}

The second term on the rigththand side of \eqref{intro_ref} quantifies the failure of the solutions $\Theta^{k,M,r}_n$, $k\in\{1,2,\ldots,K\}$, to converge to within distance $\varepsilon\in(0,1]$ to the local manifold of minima at time $n\in\mathbb{N}$.  We quantify this error in Corollary~\ref{cor_ng_one_path} below, where we prove that
\begin{equation}\label{intro_comp_cost_000} \mathbb{P}\Big( \big[\min_{k\in\{1,2,\ldots,K\}}\inf_{\theta\in(\mathcal{M}\cap U)}\big|\Theta^{k,M,r}_{n}-\theta\big|\big]\geq\varepsilon\Big) \leq \left[\kappa+c\left(\frac{1}{\varepsilon^{2}n^{\rho}}+\frac{n^{1-\rho}}{M^{\nicefrac{1}{2}}}\right)\right]^{K}.\end{equation}
The methods of Corollary~\ref{cor_computation} below prove that there exist constants $c_i\in(0,\infty)$, $i\in\{1,2,3\}$, such that for every $\varepsilon,\eta\in(0,1]$, for $n(\varepsilon)\in\mathbb{N}, M(\varepsilon),K(\eta)\in\mathbb{N}$ which satisfy that
\begin{equation}\label{intro_comp_cost_00}n(\varepsilon)=c_1\varepsilon^{-\nicefrac{2}{\rho}},\;\; M(\varepsilon)=c_2\varepsilon^{-\nicefrac{4}{\rho}+4},\;\;\textrm{and}\;\;K(\eta)=c_3\abs{\log(\eta)},\end{equation}
it holds that
\begin{equation}\label{intro_comp_cost}\mathbb{P}\Big( \big[\min_{k\in\{1,2,\ldots,K(\eta)\}}\inf_{\vartheta\in(\mathcal{M}\cap U)}\big|\varTheta^{k,M(\varepsilon),r}_{n(\varepsilon)}-\vartheta\big|\big]\geq\varepsilon\Big) \leq \eta.\end{equation}

For every bounded open set $A\subseteq\mathbb{R}^d$ with $\mathcal{M}\cap U\cap A\neq\emptyset$, for every $\varepsilon,\eta\in(0,1]$, the computational efficiency $\textrm{Eff}_{\textrm{SGD}}(\varepsilon,\eta;A)\in \mathbb{N}$ of \eqref{intro_comp_cost_000} satisfies that
\begin{equation}\textrm{Eff}_{\textrm{SGD}}(\varepsilon,\eta;A)=\#\;\textrm{computations sufficient to ensure \eqref{intro_comp_cost}}.\end{equation}
It follows from \eqref{intro_comp_cost_00} that for every bounded open set $A\subseteq \mathbb{R}^d$ with $\mathcal{M}\cap U\cap A\neq\emptyset$ there exists $c\in(0,\infty)$ such that for every $\varepsilon,\eta\in(0,1]$ it holds that
\begin{equation}\label{intro_comp_cost_0}\textrm{Eff}_{\textrm{SGD}}(\varepsilon,\eta;A)\leq c\big(\varepsilon^{-\nicefrac{6}{\rho}+4}\abs{\log(\eta)}\big).\end{equation}

In particular, the computational efficiency $\textrm{Eff}_{\textrm{SGD}}$ yields a significant improvement when compared with a random sampling algorithm.  Precisely, suppose that $A\subseteq\mathbb{R}^d$ is a bounded open subset with $\mathcal{M}\cap U\cap A\neq\emptyset$.  Then, since $\mathcal{M}\cap U$ is a $\mathfrak{d}$-dimensional, $\C^1$-submanifold of $\mathbb{R}^d$, for the Lebesgue-Borel measure $\lambda\colon\mathcal{B}(\mathbb{R}^d)\rightarrow[0,\infty]$, there exists $c\in(0,\infty)$ which satisfies that
\begin{equation}\label{intro_comp_cost_1}\frac{\lambda\big(\{\theta \in A\colon\inf_{\vartheta\in(\mathcal{M}\cap U)}\big|x-\vartheta\big|\geq \varepsilon\})}{\lambda(A)}\geq 1-\frac{c\varepsilon^{d-\mathfrak{d}}}{\lambda(A)}.\end{equation}
If $\Theta^i\colon\Omega\rightarrow A$, $i\in\mathbb{N}$, are i.i.d.\ random variables that are continuous uniformly distributed on $A$, it follows from \eqref{intro_comp_cost_1} that for every $K\in\mathbb{N}$ it holds that
\begin{equation}\label{intro_comp_cost_2}\mathbb{P}\Big(\min_{i\in\{1,2,\ldots,K\}}\inf_{\theta\in(\mathcal{M}\cap U)}\big|\Theta^i-\theta\big|\geq\varepsilon\Big)\geq \Big(1-\frac{c\varepsilon^{d-\mathfrak{d}}}{\lambda(A)}\Big)^K.\end{equation}
For every $\varepsilon,\eta\in(0,1]$, $K\in\mathbb{N}$, in order to ensure that
\begin{equation}\mathbb{P}\Big(\min_{i\in\{1,2,\ldots,K\}}\inf_{\theta\in(\mathcal{M}\cap U)}\big|\Theta^i-\theta\big|\geq\varepsilon\Big)\leq \eta,\end{equation}
it is necessary to choose $K(\varepsilon,\eta)\in\mathbb{N}$ which satisfies that
\begin{equation}K(\varepsilon,\eta)\geq \log\Big(1-\frac{c\varepsilon^{d-\mathfrak{d}}}{\lambda(A)}\Big)^{-1}\abs{\log(\eta)}.\end{equation}
In particular, there exists $c\in(0,\infty)$ which satisfies for every $\varepsilon\in(0,(\nicefrac{\lambda(A)}{2 r})^{\nicefrac{1}{d-\mathfrak{d}}}]$ that
\begin{equation}K(\varepsilon,\eta)\geq c\varepsilon^{-(d-\mathfrak{d})}\abs{\log(\eta)}.\end{equation}
The computational efficiency of the random sampling algorithm is therefore worse than $\textrm{Eff}_{\textrm{SGD}}$ whenever the codimension $d-\mathfrak{d}$ is greater than $\nicefrac{6}{\rho}-4$.  This condition is expected to be satisfied in all practical machine learning applications, where the dimension $d\in\mathbb{N}$ is large, since for $\rho\in(\nicefrac{2}{3},1)$ we have $\nicefrac{6}{\rho}-4<5$.  In particular, this condition is satisfied for any $\rho\in(\nicefrac{2}{3},1)$ if there exists a unique minimum and $d\geq 5$.

   In a non-globally stable setting, i.e.\ when \eqref{eq:intro-contraction} is not satisfied, several obstacles in the proof of local convergence to minima and the estimation of the rate for SGD appear. In particular, even pretending a local minimum to be isolated and such that \eqref{eq:intro-contraction} holds in a neighborhood $V$ of the minimum, the global analysis put forward in \cite{JKNW18} is not immediately localizable, since deterministic bounded sets are not invariant under the dynamics of SGD. On the contrary, with probability one each realization of SGD will eventually leave the basin of attraction $V$, outside of which no control on the dynamics can be expected. Therefore, it becomes necessary to provide estimates on the probability that SGD leaves favorable neighborhoods.  Second, as pointed out above, (local) minima are not expected to appear in an isolated manner, but as (local) manifolds. This needs to be accounted for in the mathematical analysis, giving rise to a quantitative analysis inspired by the center manifold theorem, which in turn relies on estimates on the probability of SGD leaving favorable neighborhoods in normal and tangential direction separately.  In order to derive estimates on the rate of convergence, these steps are performed in a quantitative way in the proofs of this work. An intriguing observation is that the mathematical analysis of the rate of convergence relies on the use of mini-batches in order to control the loss of iterates in non-attracted regions.

In Sections \ref{section_det_cts} and \ref{section_det_dis} we provide an analysis of the deterministic gradient descent algorithm in continuous and discrete time in order to highlight the relevance of the assumptions in simplified settings. We emphasize again that, while the deterministic algorithms converge quickly, the computational costs of computing $\nabla f$ typically make the implementation of such algorithms infeasible. This is particularly the case when $f$ takes the form \eqref{intro_aan} below for a measure $\mu$ that is the empirical measure of a large training set.  An advantage of the stochastic algorithm is that, provided $M\in\mathbb{N}$ is not too large, the mini-batch gradient can be computed efficiently in the case of \eqref{intro_aan_1} below.  The disadvantage is that, inside an attracting set, the algebraic convergence of SGD in expectation is much slower than the exponential convergence of its deterministic counterpart.

\subsection{Literature}

The stochastic gradient descent algorithm has attained considerable interest in the literature, and a complete account on the existing results would go beyond the scope of this article. We will therefore restrict to works that seem most relevant to the current results and refer to the following works and the references therein for further details: See, for example, \cite{B14,BachMoulines11,BM13,Bottou12,BottouBousquet11,BottouLeCun04,DarkenChangMoody92,DieuleveutDurmusBach17,InoueParkOkada03,LeCunBottouOrrMuller98,MizutaniDreyfus10,PascanuBengio13,Pillaud-VivienRudiBach17,Qian99,RakhlinShamirSridharan12,RattraySaadAmari98,SutskeverMartensDahlHinton13,Sutton86,TM15,Xu11,Zhang04} and the references mentioned therein for numerical simulations and proofs of convergence rates for SGD type optimization algorithms, \cite{BercuFort13,BCN18,Ruder16} and the references mentioned therein for overview articles on SGD type optimization algorithms, and \cite{DeanETAL12,DengETAL13,Graves13,GravesMohamedHinton13,HintonETAL12,HintonSalakhutdinov06,KrizhevskySutskeverHinton12,LeCunBottouBengioHaffner98,SchaulZhangLeCun12} and the references mentioned therein for applications involving neural networks and SGD type optimization algorithms.

The case of a convex loss function is well-understood under mild further assumptions, for example, rates of convergence of the order $O(1/\sqrt{n})$ for SGD have been established in \cite{BCN18,Zhang04}. In the case of a strongly convex objective function these can be improved to $O(1/n)$, see \cite{HAK07,NJLS09,N13}.

The case of a non-convex objective function is considerably less well understood. In this case we have to distinguish two classes of results: The first class proves the convergence to zero (with or without rates) for the gradient of the objective function, thus implying the convergence to a critical point. The second class of results proves the convergence of the values of the loss function to their global minimum. Obviously, the second class of results are stronger and not implied by the first class, since these do not exclude convergence to saddle points or local minima. In the case of non-convex loss function rather complete results are known concerning the minimization of the gradient of the loss function. For example, the convergence of the gradient to zero with rates was shown in Lei, Hu, Li, \& Tang \cite{LHLT19} assuming a H\"older-regularity condition on the gradient of the loss function. This generalizes previous work Ghadimi, Lan, \& Zhang \cite{GLHZ13} which required a second moment boundedness condition, which in turn was generalized by previous works Ghadimi \& Lan \cite{GL13} and Reddi, Hefny, Sra, Poczos, \& Smola \cite{RHSPS16}. We note that while convergence to the global minimum with rates was obtained in \cite{GLHZ13} for the convex case, no results on the convergence of the value of the loss function have been shown in the non-convex case.

The convergence of the stochastic gradient descent method has been analysed in the literature under several additional assumptions replacing (strong) convexity, such as the error bounds condition in Luo \& Tseng \cite{LT93}, essential strong convexity \cite{LWRBS13}, weak strong convexity \cite{NNG15}, the restricted secant inequality \cite{ZY13}, and the quadratic growth condition Anitescu \cite{A00}. In these works, linear convergence rates are shown. In the notable contribution Karimi, Nutini, \& Schmidt \cite{KMS16} have shown that all of these conditions imply the Polyak-Lojasiewicz (PL) inequality, introduced in Lojasiewicz \cite{L63} and Polyak \cite{P63}, under which linear convergence of SGD is proven in \cite{KMS16}, thus generalizing these previous works. Recently, further progress was made in Lei, Hu, Li, \& Tang in \cite{LHLT19} where a boundedness assumption on the gradient of the objective function, required in \cite{KMS16}, was relaxed. We note that, while the PL condition does not require convexity, nor the uniqueness of global minimizers, it does exclude the existence of local minima, that is, assuming the PL condition each local minimum is a global minimum. Therefore, it is not implied by the assumptions made in the current work.

\subsection{Structure of the work}
The paper is organized as follows.  We will use the local smoothness of $\mathcal{M}\cap U$, the local smoothness of the objective function $f$, and the maximal nondegeneracy of the Hessian to identify a basin of attraction for SGD.  In Section~\ref{sec_geometry}, we present the geometric preliminaries that are used to identify this set.  In particular, in Proposition~\ref{nondegenerate} below we recall the existence of projections in a local neighborhoods of $\mathcal{M}\cap U$, in Proposition~\ref{cts_prop_tub} below we recall the existence of local tubular neighborhoods about $\mathcal{M}\cap U$, in Lemma~\ref{lem_main} below we prove a useful decomposition of $\nabla f$ into components normal and tangential to $\mathcal{M}\cap U$, and in Lemma~\ref{lem_projection} below we prove a contraction estimate that will be used to obtain a convergence rate for the gradient descent algorithms in discrete time.

In Section~\ref{section_det_cts}, for objective functions $f\colon\mathbb{R}^d\rightarrow\mathbb{R}$ that satisfy the conditions of Theorem~\ref{intro_ng_one_path}, we analyze the converge of the deterministic gradient descent algorithm in continuous time $\theta_t\in\mathbb{R}^d$, $t\in[0,\infty)$, that satisfies for every $t\in(0,\infty)$ that
\begin{equation} \frac{d}{dt}\theta_t=-\nabla f(\theta_t).\end{equation}
We prove in Proposition~\ref{cts_converge} below that the local smoothness of $\mathcal{M}\cap U$, the local smoothness of $f$, and the nondegeneracy of the Hessian imply the existence of a neighborhood $V\subseteq\mathbb{R}^d$ such that for every $\theta_0\in V$ the solution $\theta_t$, $t\in[0,\infty)$, converges exponentially fast to $\mathcal{M}\cap U$.  However, since in general neither $f$ nor $\nabla f$ are practically computable, and since continuous gradient descent cannot be implemented, the purpose of this section is to explain in a simplified setting the role of the assumptions and the geometric arguments from Section~\ref{sec_geometry}.

In Section~\ref{section_det_dis}, for objective functions $f\colon\mathbb{R}^d\rightarrow\mathbb{R}$ that satisfy the conditions of Theorem~\ref{intro_ng_one_path}, we analyze the converge of the deterministic gradient descent algorithm in discrete time $\theta_n\in\mathbb{R}^d$, $n\in\mathbb{N}_0$, that satisfies for $\rho\in(0,1)$, $r\in(0,\infty)$, for every $n\in\mathbb{N}$ that
\begin{equation} \theta_n=\theta_{n-1}-\frac{r}{n^\rho}\nabla f(\theta_{n-1}).\end{equation}
We prove in Proposition~\ref{dis_converge} below that there exists a neighborhood $V\subseteq\mathbb{R}^d$ such that for every $\theta_0\in V$ the solution $\theta_n$, $n\in\mathbb{N}_0$, converges exponentially quickly to $\mathcal{M}\cap U$.  However, while discrete gradient descent yields an implementable algorithm, the computational costs of $f$ and $\nabla f$ in general make it practically infeasible.  The purpose of this section is instead to explain how the geometric preliminaries of Section~\ref{sec_geometry}, and in particular Lemma~\ref{lem_main} and Lemma~\ref{lem_projection}, are applied in a simplified discrete setting.

In Section~\ref{sec_nonglobal}, we analyze the convergence of SGD to the manifold of local minima $\mathcal{M}\cap U$.  In Proposition~\ref{ng_converge} below, we prove the convergence of \eqref{intro_stoch_mean_10} to $\mathcal{M}\cap U$ in directions normal to the manifold.  Precisely, we identify a basin of attraction $V\subseteq\mathbb{R}^d$ such that, on the event that SGD remains in $V$, SGD converges to $\mathcal{M}\cap U$ in expectation with an algebraic rate.  It remains to estimate the probability that SGD remains in the basin of attraction $V$.

The first step is contained in Proposition~\ref{ng_tang} below, which estimates the maximal excursion of SGD in expectation.  Then, in Proposition~\ref{ng_good_set} below, we estimate the probability that SGD remains in a basin of attraction $V$ by separating this event into the event that SGD leaves $V$ in a direction normal to $\mathcal{M}\cap U$ and the event that SGD leaves $V$ in a direction tangential to $\mathcal{M}\cap U$.  Proposition~\ref{ng_converge} is used to estimate the first of these events, and Proposition~\ref{ng_tang} is used to estimate the second.  In Theorem~\ref{ng_one_path}, we combine Proposition~\ref{ng_converge} and Proposition~\ref{ng_good_set} to estimate the probability that SGD converges to within distance $\varepsilon\in(0,1]$ of $\mathcal{M}\cap U$.

In Corollary~\ref{cor_ng_one_path} below, we estimate the probability that $K\in\mathbb{N}$ independent copies of SGD fail to converge to within distance $\varepsilon\in(0,1]$ of $\mathcal{M}\cap U$.  In Theorem~\ref{thm_intro_ng_one_path} below we prove Theorem~\ref{intro_ng_one_path}, which relies on Lemma~\ref{lem_measurable_selection} below and estimates for the mini-batch approximation of the objective function.  Finally, in Corollary~\ref{cor_computation} below, we estimate the computational efficiency of the algorithm introduced in Theorem~\ref{intro_ng_one_path}.

In Section~\ref{stoch_discrete}, we prove that the estimates of Section~\ref{sec_nonglobal} can be improved under the additional assumption that $\mathcal{M}\cap U$ is compact.  These estimates apply, in particular, to the case when the objective function has a unique minimum.  The reason for the improved estimate of Theorem~\ref{thm_intro_ng_one_path_comp} below and the improved computational efficiency of Corollary~\ref{cor_computation_comp} below is that, in the compact case, SGD cannot escape a basin of attraction in directions tangential to the manifold.  It is therefore sufficient to take a smaller mini-batch approximation of the gradient.

In Section~\ref{sec_applications}, we prove that assumptions of Theorem~\ref{intro_ng_one_path} are satisfied by simple loss functions arising in machine learning applications.  In particular, we show that the assumptions are satisfied by objective functions $f\colon\mathbb{R}^d\rightarrow\mathbb{R}$ which satisfy that
   \begin{equation}\label{intro_aan}f(\theta)=\int_S\abs{u_\theta(x)-\varphi(x)}^p\mu(\dx),\end{equation}
   where $\theta\in\mathbb{R}^d$, $p\in[1,\infty)$, $\varphi$ a  measurable function on a measurable space $(S,\mathcal{S})$, and $(u_\theta\colon S\rightarrow\mathbb{R}\mathbb)_{\theta\in\mathbb{R}^d}$ is a jointly-measurable artificial neural network.  In this case, the function $F\colon\mathbb{R}^d\times S\rightarrow\mathbb{R}$ satisfies for every $(\theta,x)\in\mathbb{R}^d\times S$ that
   \begin{equation}\label{intro_aan_1}F(\theta,x)=\abs{u_\theta(x)-\varphi(x)}^p,\end{equation}
   and, for a probability space $(\Omega,\mathcal{F},\mathbb{P})$, the sequence of random variables $X_{k,n,m}\colon\Omega\rightarrow S$, $k,n,m\in\mathbb{N}$, are i.i.d.\ with distribution $\mu$.  For the objective functions considered in Section \ref{examples} and Section \ref{two_parameter} below, the global minima are non-unique and build locally smooth, non-compact manifolds of $\mathbb{R}^d$ on which Hessian of the objective function is maximally nondegenerate.

\section{Geometric preliminaries}\label{sec_geometry}

In this section, for an objective function $f\colon\mathbb{R}^d\rightarrow\mathbb{R}$ that satisfies the conditions of Theorem~\ref{intro_ng_one_path}, we will characterize the local geometry of the local manifold of minima $\mathcal{M}\cap U$.  The analysis will rely on on the notion of a projection to $\mathcal{M}\cap U$ which is, however, only well-defined in local neighborhoods of the local manifold.

In the following proposition, we prove that the projection map to the local manifold of minima is locally well-defined and smooth.  The proof is a consequence of Foote \cite[Lemma]{Foote} and the smoothness of $\mathcal{M}\cap U$.

\begin{prop}\label{def_projection}  Let $d\in\mathbb{N}$, $\mathfrak{d}\in\{1,\ldots,d-1\}$, let $\abs{\cdot}\colon\mathbb{R}^d\rightarrow\mathbb{R}$ be the standard norm on $\mathbb{R}^d$, and let $\mathcal{M}\cap U\subseteq\mathbb{R}^d$ be a non-empty $\mathfrak{d}$-dimensional $\C^1$-submanifold of $\mathbb{R}^d$.
Then for every $x_0\in(\mathcal{M}\cap U)$ there exists an open  neighborhood $V\subset\mathbb{R}^d$ such that
\begin{enumerate}[(i)]

\item $V$ is a neighborhood of $x_0$:  it holds that $x_0\in V$.

\item projections exist in $V$:  there exists a unique function $p\colon V\rightarrow(\mathcal{M}\cap U)$ which satisfies for every $x\in V$ that
\begin{equation}\label{VV_1}\abs{x-p(x)}=\inf \left\{\abs{x-y}\colon y\in(\mathcal{M}\cap U)\right\}.\end{equation}
\item the projection map is locally $\C^1$-smooth: the map $p\colon V\rightarrow (\mathcal{M}\cap U)$ is once continuously differentiable.
\end{enumerate}
\end{prop}

\begin{proof}[Proof of Proposition~\ref{def_projection}]   The proof is an immediate consequence of \cite[Lemma]{Foote} and the $\C^1$-regularity of $\mathcal{M}\cap U$.  \end{proof}

The family of subsets satisfying for a fixed $x_0\in(\mathcal{M}\cap U)$ the conclusion of Proposition~\ref{def_projection} will play an important role in the arguments to follow.  We therefore make a global definition, and define the projection map on a global neighborhood of $\mathcal{M}\cap U$.  The existence of the projection map is an immediate consequence of Proposition~\ref{def_projection}.

\begin{definition}\label{define_proj}  Let $d\in\mathbb{N}$, $\mathfrak{d}\in\{1,\ldots,d-1\}$, let $\mathcal{M}\cap U\subseteq\mathbb{R}^d$ be a non-empty $\mathfrak{d}$-dimensional $\C^1$-submanifold of $\mathbb{R}^d$.
\begin{enumerate}[(i)]
\item For every $x\in(\mathcal{M}\cap U)$ let $\Proj(x)\subseteq\mathcal{B}(\mathbb{R}^d)$ satisfy that
\begin{equation} \Proj(x)=\{V\subseteq\mathbb{R}^d\colon V\;\;\textrm{satisfies the conclusion of Proposition~\ref{def_projection} with $x_0=x$.}\}.\end{equation}
\item Let $p\colon\cup_{x\in(\mathcal{M}\cap U)}\left(\cup_{V\in\Proj(x)}V\right)\rightarrow(\mathcal{M}\cap U)$ be the unique function which satisfies for every $x\in \cup_{x\in(\mathcal{M}\cap U)}\left(\cup_{V\in\Proj(x)}V\right)$ that
\begin{equation}\abs{x-p(x)}=\inf \left\{\abs{x-y}\colon y\in(\mathcal{M}\cap U)\right\}.\end{equation}
\end{enumerate}
\end{definition}

The following proposition proves that for every $x\in(\mathcal{M}\cap U)$ the tangent space $T_x(\mathcal{M}\cap U)$ and normal space $\big(T_x(\mathcal{M}\cap U)\big)^\perp$ to $\mathcal{M}\cap U$ at $x$ are characterized respectively by the null space of Hessian of $f$ and the space on which the Hessian of $f$ is positive definite.
\begin{prop}\label{nondegenerate}  Let $d\in\mathbb{N}$, $\mathfrak{d}\in\{1,2,\ldots,d-1\}$, let $\abs{\cdot}\colon\mathbb{R}^d\rightarrow\mathbb{R}$ be the standard norm on $\mathbb{R}^d$, let $U\subseteq\mathbb{R}^d$ be an open set, let $f\colon U\rightarrow\mathbb{R}$ be a three times continuously differentiable function, let $\mathcal{M}\subseteq\mathbb{R}^d$ satisfy that
\begin{equation}\mathcal{M}=\big\{\theta\in\mathbb{R}^d\colon [f(\theta)=\inf\nolimits_{\vartheta\in\mathbb{R}^d} f(\vartheta)]\big\},\end{equation}
assume that $\mathcal{M}\cap U$ is a non-empty $\mathfrak{d}$-dimensional $\C^1$-submanifold of $\mathbb{R}^d$ and assume for every $\theta\in(\mathcal{M}\cap U)$ that $\rank((\Hess f)(\theta))=d-\mathfrak{d}$.  Then for every $x\in(\mathcal{M}\cap U)$ there exist a $(d-\mathfrak{d})$-dimensional subvectorspace $P_x\subseteq\mathbb{R}^d$ and a $\mathfrak{d}$-dimensional subvectorspace $N_x\subseteq\mathbb{R}^d$ such that
\begin{enumerate}[(i)]

\item it holds that
\begin{equation}\big(\Hess f\big)(x)(P_x) = P_x,\end{equation}
\item it holds for every $v\in P_x\backslash\{0\}$ that
\begin{equation}\big(\big[\big(\Hess f\big)(x)\big]v\big)\cdot v>0,\end{equation}
\item it holds that
\begin{equation}\big(\Hess f\big)(x)\vert_{N_x}=0,\end{equation}
\item it holds that
\begin{equation}N_x=T_x(\mathcal{M}\cap U),\end{equation}
\item  it holds that
\begin{equation}P_x=\big(T_x(\mathcal{M}\cap U)\big)^\perp.\end{equation}
\end{enumerate}
\end{prop}

\begin{proof}[Proof of Proposition~\ref{nondegenerate}]  Let $x\in(\mathcal{M}\cap U)$.  Since $\rank((\Hess f)(\theta))=d-\mathfrak{d}$, the symmetry of the Hessian implies that there exist subspaces $N_x,P_x\subseteq\mathbb{R}^d$ such that $\mathbb{R}^d=P_x\oplus N_x$, that $\dim(P_x)=d-\mathfrak{d}$, that
\begin{equation}\big(\Hess f\big)(x)(P_x) \subseteq P_x\;\;\textrm{with}\;\;\big(\Hess f\big)(x)\vert_{P_x}\;\;\textrm{strictly positive definite on}\;\;P_x,\end{equation}
that $\dim(N_x)=\mathfrak{d}$, and that
\begin{equation}\big(\Hess f\big)(x)\vert_{N_x}=0.\end{equation}
Let $\varepsilon\in(0,1)$ and suppose that $\gamma\colon  (-\varepsilon,\varepsilon)\rightarrow\mathcal{M}\cap U$ is a smooth curve which satisfies $\gamma(0)=x$.  Since $\nabla f|_{\mathcal{M}\cap U}=0$, it follows from the chain rule that
\begin{equation}\left.\frac{d}{dt} \nabla f(\gamma(t))\right|_{t=0}=\big(\Hess f\big)(x)\cdot\dot{\gamma}(0)=0.\end{equation}
It follows that $T_x(\mathcal{M}\cap U)\subseteq N_x$ and therefore, since $\dim(T_x(\mathcal{M}\cap U))=\mathfrak{d}$, it holds that $T_x(\mathcal{M}\cap U)= N_x$.  Since $\mathbb{R}^d=T_x(\mathcal{M}\cap U)\oplus \big(T_x(\mathcal{M}\cap U)\big)^\perp$, it holds that $P_x=\big(T_x(\mathcal{M}\cap U)\big)^\perp$, which completes the proof of Proposition~\ref{nondegenerate}.  \end{proof}

In the following lemma, for a point $x\in\mathbb{R}^d$ such that the projection $p(x)\in(\mathcal{M}\cap U)$ is well-defined, we prove that the difference $x-p(x)\in\mathbb{R}^d$ lies in the space normal to $\mathcal{M}\cap U$ at $p(x)$.  This fact will be used to obtain a rate of convergence for the discrete gradient descent algorithms.

\begin{lem}\label{normal}  Let $d\in\mathbb{N}$, $\mathfrak{d}\in\{1,2,\ldots,d-1\}$, let $\abs{\cdot}\colon\mathbb{R}^d\rightarrow\mathbb{R}$ be the standard norm on $\mathbb{R}^d$, let $U\subseteq\mathbb{R}^d$ be an open set, let $f\colon U\rightarrow\mathbb{R}$ be a three times continuously differentiable function, let $\mathcal{M}\subseteq\mathbb{R}^d$ satisfy that
\begin{equation}\mathcal{M}=\big\{\theta\in\mathbb{R}^d\colon [f(\theta)=\inf\nolimits_{\vartheta\in\mathbb{R}^d} f(\vartheta)]\big\},\end{equation}
assume that $\mathcal{M}\cap U$ is a non-empty $\mathfrak{d}$-dimensional $\C^1$-submanifold of $\mathbb{R}^d$, and assume for every $\theta\in(\mathcal{M}\cap U)$ that $\rank((\Hess f)(\theta))=d-\mathfrak{d}$.  Then for every $x_0\in(\mathcal{M}\cap U)$, for every $V\in\Proj(x_0)$ (cf. Definition~\ref{define_proj}), it holds for every $x\in V$ that
\begin{equation}x-p(x)\in T_{p(x)}\left(\mathcal{M}\cap U\right)^\perp.\end{equation}
\end{lem}

\begin{proof}[Proof of Lemma~\ref{normal}]  Let $x_0\in(\mathcal{M}\cap U)$, let $V\in\Proj(x_0)$, and let $p:V\rightarrow(\mathcal{M}\cap U)$ denote the projection map.  Let $x\in V$.  If $x\in(\mathcal{M}\cap U)$, the claim is immediate since then $x-p(x)=0$.  If $x\notin\mathcal{M}\cap U$, for some $\varepsilon\in(0,1)$ suppose that $\gamma\colon  (-\varepsilon,\varepsilon)\rightarrow\mathcal{M}\cap U$ is a smooth path which satisfies $\gamma(0)=p(x)$.  It holds that
\begin{equation}\left.\frac{d}{dt}\abs{x-\gamma(t)}^2\right|_{t=0}=-2\dot{\gamma}(0)\cdot(x-p(x))=0.\end{equation}
Therefore, since the curve $\gamma$ was arbitrary, it holds that $x-p(x)\in T_{p(x)}\left(\mathcal{M}\cap U\right)^\perp$, which completes the proof of Lemma~\ref{normal}. \end{proof}

In the following lemma, we derive a formula for the derivative of the distance function to the manifold in a neighborhood of $\mathcal{M}\cap U$.  The regularity of the distance function and the formula for its differential will be used to prove the convergence of the deterministic gradient descent algorithm in continuous time.

\begin{lem}\label{distance}  Let $d\in\mathbb{N}$, $\mathfrak{d}\in\{1,2,\ldots,d-1\}$, let $\abs{\cdot}\colon\mathbb{R}^d\rightarrow\mathbb{R}$ be the standard norm on $\mathbb{R}^d$, let $U\subseteq\mathbb{R}^d$ be an open set, let $f\colon U\rightarrow\mathbb{R}$ be a three times continuously differentiable function, let $\mathcal{M}\subseteq\mathbb{R}^d$ satisfy that
\begin{equation}\mathcal{M}=\big\{\theta\in\mathbb{R}^d\colon [f(\theta)=\inf\nolimits_{\vartheta\in\mathbb{R}^d} f(\vartheta)]\big\},\end{equation}
let $\mathbf{d}(\cdot,\mathcal{M}\cap U):\mathbb{R}^d\rightarrow\mathbb{R}$ be the function which satisfies for every $x\in\mathbb{R}^d$ that
\begin{equation} \mathbf{d}(x,\mathcal{M}\cap U)=\inf \left\{\abs{x-y}\colon y\in(\mathcal{M}\cap U)\right\},\end{equation}
assume that $\mathcal{M}\cap U$ is a non-empty $\mathfrak{d}$-dimensional $\C^1$-submanifold of $\mathbb{R}^d$, and assume for every $\theta\in(\mathcal{M}\cap U)$ that $\rank((\Hess f)(\theta))=d-\mathfrak{d}$.  Then for every $x_0\in(\mathcal{M}\cap U)$, for every $V\in\Proj(x_0)$ (cf. Definitition~\ref{define_proj}), it holds for every $x\in V\setminus\mathcal{M}\cap U$ that
\begin{equation}(\nabla \mathbf{d})(x,\mathcal{M}\cap U)=\frac{x-p(x)}{\abs{x-p(x)}}.\end{equation}
\end{lem}

\begin{proof}[Proof of Lemma~\ref{distance}]  Let $x_0\in(\mathcal{M}\cap U)$ and let $V\in\Proj(x_0)$.  It follows from Proposition~\ref{def_projection} that
\begin{equation}x\in V\mapsto \abs{x-p(x)}^2=\mathbf{d}(x,\mathcal{M}\cap U)^2\;\;\textrm{is $\C^1$}.\end{equation}
The chain rule implies for every $i\in\{1,\ldots,d\}$ that
\begin{equation}\frac{\partial}{\partial x_i}\mathbf{d}(x,\mathcal{M}\cap U)^2=\frac{\partial}{\partial x_i}\abs{x-p(x)}^2 =2(x-p(x))\cdot e_i -2(x-p(x))\cdot \frac{\partial}{\partial x_i} p(x).\end{equation}
Since $\frac{\partial}{\partial x_i} p(x)\in N_{p(x)}$ and since $x-p(x)\in P_{p(x)}$ it follows from Lemma~\ref{normal} that
\begin{equation}(x-p(x))\cdot \frac{\partial}{\partial x_i} p(x)=0.\end{equation}
Since for every $x\in V\setminus\mathcal{M}\cap U$ it holds that
\begin{equation}\nabla\mathbf{d}(x,\mathcal{M}\cap U)^2=2\mathbf{d}(x,\mathcal{M}\cap U)\nabla\mathbf{d}(x,\mathcal{M}\cap U)=2(x-p(x)),\end{equation}
it holds for every $x\in V\setminus\mathcal{M}\cap U$ that
\begin{equation}\nabla\mathbf{d}(x,\mathcal{M}\cap U)=\frac{x-p(x)}{\abs{x-p(x)}},\end{equation}
which completes the proof of Lemma~\ref{distance}.  \end{proof}

We will now quantify what are essentially local tubular neighborhoods of the local manifold $\mathcal{M}\cap U$.  The following definition will play an important role throughout the paper.

\begin{definition}\label{dtbn} Let $d\in\mathbb{N}$, $\mathfrak{d}\in\{1,\ldots,d-1\}$, let $\mathcal{M}\cap U\subseteq\mathbb{R}^d$ be a non-empty $\mathfrak{d}$-dimensional $\C^1$-submanifold of $\mathbb{R}^d$.  For every $x\in(\mathcal{M}\cap U)$, $R,\delta\in(0,\infty)$ let $V_{R,\delta}(x)\subseteq\mathbb{R}^d$ satisfy that
\begin{equation}V_{R,\delta}(x)=\{y+v\colon y\in (\overline{B}_R(x)\cap\mathcal{M}\cap U)\;\textrm{and}\;v\in \big(T_y(\mathcal{M}\cap U)\big)^\perp\;\textrm{with}\;\abs{v}<\delta\}.\end{equation}
\end{definition}

A useful feature of the sets defined in Definition~\ref{dtbn} is that the parameter $R\in(0,\infty)$ can be used to quantify distance in directions tangential to the manifold $\mathcal{M}\cap U$, and the parameter $\delta\in(0,\infty)$ can be used to quantify distance in directions normal to the manifold $\mathcal{M}\cap U$.  The following technical proposition will be used to prove Proposition~\ref{dis_converge} below and Lemma~\ref{lem_tub} below.

\begin{prop}\label{cts_prop_tub}  Let $d\in\mathbb{N}$, $\mathfrak{d}\in\{1,\ldots,d-1\}$, let $\abs{\cdot}\colon\mathbb{R}^d\rightarrow\mathbb{R}$ be the standard norm on $\mathbb{R}^d$, let $\mathcal{M}\cap U\subseteq\mathbb{R}^d$ be a non-empty $\mathfrak{d}$-dimensional $\C^1$-submanifold of $\mathbb{R}^d$, and let $\mathbf{d}(\cdot,\mathcal{M}\cap U):\mathbb{R}^d\rightarrow\mathbb{R}$ be the function which satisfies for every $x\in\mathbb{R}^d$ that
\begin{equation} \mathbf{d}(x,\mathcal{M}\cap U)=\inf \left\{\abs{x-y}\colon y\in(\mathcal{M}\cap U)\right\}.\end{equation}
Then for every $x_0\in(\mathcal{M}\cap U)$, for every $V\in\Proj(x_0)$ (cf. Definition~\ref{define_proj}), there exist $R_0,\delta_0\in(0,\infty)$ such that for every $R\in(0,R_0]$, $\delta\in(0,\delta_0]$,

\begin{enumerate}[(i)]

\item it holds that $\overline{V}_{R,\delta}(x_0)\subseteq V$ (cf. Definition~\ref{dtbn}),
\item it holds that
\begin{equation}V_{R,\delta}(x_0) = \{x\in\mathbb{R}^d\colon\mathbf{d}(x,\mathcal{M}\cap U)=\mathbf{d}(x,\overline{B}_R(x_0)\cap\mathcal{M}\cap U)<\delta\},\end{equation}
\item it holds for every $x\in (\overline{B}_R(x_0)\cap\mathcal{M}\cap U)$ and $v\in \big(T_x(\mathcal{M}\cap U)\big)^\perp$ with $\abs{v}<\delta$ that
\begin{equation}p(x+v)=x.\end{equation}
\end{enumerate}
\end{prop}

\begin{proof}[Proof of Proposition~\ref{cts_prop_tub}] Let $x_0\in(\mathcal{M}\cap U)$.  For every $R,\delta\in(0,\infty)$ let $\tilde{V}_{R,\delta}(x_0)\subseteq\mathbb{R}^d$ satisfy that
\begin{equation}\tilde{V}_{R,\delta}(x_0)=\{x\in\mathbb{R}^d\colon\mathbf{d}(x,\mathcal{M}\cap U)=\mathbf{d}(x,\overline{B}_R(x_0)\cap\mathcal{M}\cap U)<\delta\}.\end{equation}
Let $V\in\Proj(x_0)$.  Since $U,V\subseteq\mathbb{R}^d$ are open, there exist $R_0,\delta_0\in(0,\infty)$ such that for every $R\in(0,R_0]$ it holds that
\begin{equation}\overline{B}_R(x_0)\cap\mathcal{M}\subseteq\mathcal{M}\cap U,\end{equation}
and for every $R\in(0,R_0]$, $\delta\in(0,\delta_0]$ that
\begin{equation}\label{tub_1}V_{R,\delta}(x_0)\subseteq V\;\;\textrm{and}\;\;\tilde{V}_{R,\delta}(x_0)\subseteq V.\end{equation}
Following \cite[Lemma]{Foote}, the normal bundle $T\left(\mathcal{M}\cap U\right)^\perp\subseteq\mathbb{R}^{2d}$ satisfies that
\begin{equation}T\left(\mathcal{M}\cap U\right)^\perp =\left\{(x,v)\in\mathbb{R}^d\times\mathbb{R}^d\colon x\in(\mathcal{M}\cap U)\;\textrm{and}\;v\in T_x\left(\mathcal{M}\cap U\right)^\perp\right\}.\end{equation}
Since $\mathcal{M}\cap U$ is a $\mathfrak{d}$-dimensional $\C^1$-submanifold, it follows that $T(\mathcal{M}\cap U)^\perp\subseteq\mathbb{R}^{2d}$ is a $d$-dimensional $\C^1$-submanifold.  Furthermore, the map $\Psi\colon  T(\mathcal{M}\cap U)^\perp\rightarrow\mathbb{R}^d$ which satisfies for every $(x,v)\in T\left(\mathcal{M}\cap U\right)^\perp$ that  $\Psi(x,v)=x+v$ satisfies for every $x\in(\mathcal{M}\cap U)$ that
\begin{equation}D_{(x,0)}\Psi\colon  T_{(x,0)}\big(T\left(\mathcal{M}\cap U\right)^\perp\big)\rightarrow T_x\mathbb{R}^d\;\;\textrm{is nonsingular.}\end{equation}
It follows from the inverse function theorem that there exists $\delta_1\in(0,(\delta_0\wedge\nicefrac{R_0}{4}))$ such that for every $R\in(0,\nicefrac{R_0}{2}]$, $\delta\in(0,\delta_1]$ it holds that
\begin{equation}\label{tub_2}\Psi\colon  \{(x,v)\in\left(TM\right)^\perp\colon x\in \overline{B}_{R+2\delta_1}(x_0)\;\textrm{and}\;\abs{v}<\delta\}\rightarrow V_{R+2\delta_1,\delta}(x_0)\;\;\textrm{is injective.}\end{equation}
Let $R\in(0,\nicefrac{R_0}{2}]$, $\delta\in(0,\delta_1]$.  We will first prove that $\tilde{V}_{R,\delta}(x_0)\subseteq V_{R,\delta}(x_0)$.  Let $x\in \tilde{V}_{R,\delta}(x_0)$.  If $x\in \overline{B}_R(x_0)\cap\mathcal{M}\cap U$ then it holds by definition that $x\in V_{R,\delta}(x_0)$.  If $x\notin \overline{B}_R(x_0)\cap\mathcal{M}\cap U$, since $x\in \tilde{V}_{R,\delta}(x_0)$ implies that $\mathbf{d}(x,\mathcal{M}\cap U)=\mathbf{d}(x,\overline{B}_R(x_0)\cap\mathcal{M}\cap U)$ and since the choice of $R_0\in(0,\infty)$ implies that
\begin{equation}\overline{B}_R(x_0)\cap\mathcal{M}\cap U=\overline{B}_R(x_0)\cap\mathcal{M}\;\;\textrm{is a closed subset of $\mathbb{R}^d$,}\end{equation}
it holds that $p(x)\in \overline{B}_R(x_0)\cap\mathcal{M}\cap U$.   Since $\mathbf{d}(x,\mathcal{M}\cap U)=\mathbf{d}(x,\overline{B}_R(x_0)\cap\mathcal{M}\cap U)=\abs{x-p(x)}< \delta$ and since it holds that
\begin{equation}x=p(x)+ \abs{x-p(x)}\frac{x-p(x)}{\abs{x-p(x)}},\end{equation}
for $\frac{x-p(x)}{\abs{x-p(x)}}\in  T_x\left(\mathcal{M}\cap U\right)^\perp$ by Lemma~\ref{normal}, it holds that $x\in V_{R,\delta}(x_0)$.  This completes the proof that $\tilde{V}_{R,\delta}(x_0)\subseteq V_{R,\delta}(x_0)$.
It remains to prove that $V_{R,\delta}(x_0)\subseteq \tilde{V}_{R,\delta}(x_0)$.  Let $x\in V_{R,\delta}(x_0)$.  It is necessary to show that $\mathbf{d}(x,\mathcal{M}\cap U)=\mathbf{d}(x,\overline{B}_R(x_0)\cap\mathcal{M}\cap U)< \delta$.  The definition of $V_{R,\delta}(x_0)$ implies that there exist $\tilde{x}\in (\overline{B}_R(x_0)\cap\mathcal{M}\cap U)$ and $\tilde{v}\in T_{\tilde{x}}\left(\mathcal{M}\cap U\right)^\perp$ with $\abs{\tilde{v}}<\delta$ which satisfy that $x=\tilde{x}+\tilde{v}$.  We will prove that $p(x)=\tilde{x}$.  By contradiction, suppose that $p(x)\neq \tilde{x}$.  This implies that
\begin{equation}\abs{x-p(x)}<\abs{x-\tilde{x}}=\abs{\tilde{v}}<\delta.\end{equation}
It follows from the triangle inequality that
\begin{equation}\label{tub_3}\abs{p(x)-\tilde{x}}\leq \abs{p(x)-x}+\abs{x-\tilde{x}}<2\delta\leq 2\delta_1,\end{equation}
which proves that
\begin{equation}\label{tub_4}x=\tilde{x}+\tilde{v}=p(x)+(x-p(x)),\end{equation}
for $x-p(x)\in T_{p(x)}\left(\mathcal{M}\cap U\right)^\perp$ by Lemma~\ref{normal} with $\abs{x-p(x)}< \delta$.  Since $\tilde{x}\in (\overline{B}_R(x_0)\cap\mathcal{M}\cap U)$, it follows from \eqref{tub_3} that $p(x)\in(\overline{B}_{R+2\delta_1}(x_0)\cap\mathcal{M}\cap U)$.  Since $R\in(0,\nicefrac{R_0}{2}]$ and since $\delta\in(0,\delta_1]$, equation \eqref{tub_4} contradicts \eqref{tub_2}, which states that $\Psi$ is injective on the set
\begin{equation}\{(x,v)\in\left(TM\right)^\perp\colon x\in B_{R+2\delta_1}(x_0)\;\textrm{and}\;\abs{v}<\delta\}.\end{equation}
We conclude that $p(x)=\tilde{x}$, which implies that
\begin{equation}\mathbf{d}(x,\mathcal{M}\cap U)=\mathbf{d}(x,\overline{B}_R(x_0)\cap\mathcal{M}\cap U)=\abs{x-p(x)}=\abs{\tilde{v}}<\delta.\end{equation}
Therefore, it holds that $V_{R,\delta}(x_0)\subseteq \tilde{V}_{R,\delta}(x_0)$, which completes the proof that $\tilde{V}_{R,\delta}(x_0)=V_{R,\delta}(x_0)$.  The final claim follows from a repetition of the arguments leading to \eqref{tub_3} and \eqref{tub_4}.  This completes the proof of of Proposition~\ref{cts_prop_tub}.\end{proof}

The following two lemmas contain the primary use of the nondegeneracy assumption, which states for every $\theta\in(\mathcal{M}\cap U)$ that
\begin{equation}\rank((\Hess f)(\theta))=d-\mathfrak{d}=\codim(\mathcal{M}\cap U).\end{equation}
The first of these proves that $\nabla f$ can be split into a component that is approximately normal to the local manifold of minima $\mathcal{M}\cap U$, and into a component that is approximately tangential to $\mathcal{M}\cap U$.  We will use the normal component to obtain a rate of convergence for the gradient descent algorithms.  The contribution of the tangential component will create errors that will need to be controlled.

\begin{lem}\label{lem_main}  Let $d\in\mathbb{N}$, $\mathfrak{d}\in\{1,2,\ldots,d-1\}$, let $\abs{\cdot}\colon\mathbb{R}^d\rightarrow\mathbb{R}$ be the standard norm on $\mathbb{R}^d$, let $U\subseteq\mathbb{R}^d$ be an open set, let $f\colon U\rightarrow\mathbb{R}$ be a three times continuously differentiable function, let $\mathcal{M}\subseteq\mathbb{R}^d$ satisfy that
\begin{equation}\mathcal{M}=\big\{\theta\in\mathbb{R}^d\colon [f(\theta)=\inf\nolimits_{\vartheta\in\mathbb{R}^d} f(\vartheta)]\big\},\end{equation}
let $\mathbf{d}(\cdot,\mathcal{M}\cap U):\mathbb{R}^d\rightarrow\mathbb{R}$ be the function which satisfies for every $x\in\mathbb{R}^d$ that
\begin{equation} \mathbf{d}(x,\mathcal{M}\cap U)=\inf \left\{\abs{x-y}\colon y\in(\mathcal{M}\cap U)\right\},\end{equation}
assume that $\mathcal{M}\cap U$ is a non-empty $\mathfrak{d}$-dimensional $\C^1$-submanifold of $\mathbb{R}^d$, and assume for every $\theta\in(\mathcal{M}\cap U)$ that $\rank((\Hess f)(\theta))=d-\mathfrak{d}$.  Then for every $x_0\in(\mathcal{M}\cap U)$ there exist $R_0,\delta_0,c\in(0,\infty)$ and $V\in\Proj(x_0)$ (cf. Definition~\ref{define_proj}) such that for every $R\in(0,R_0]$, $\delta\in(0,\delta_0]$ it holds that (cf. Definition~\ref{dtbn})
\begin{equation}\overline{V}_{R,\delta}(x_0)\subseteq V,\end{equation}
and for every $x\in V_{R,\delta}(x_0)$ there exists $\varepsilon_x\in\mathbb{R}^d$ which satisfies $\abs{\varepsilon_x}\leq c\mathbf{d}(x,\mathcal{M}\cap U)^2$ such that
\begin{equation}\nabla f(x)=\big(\Hess f\big)(p(x))\cdot(x-p(x))+\varepsilon_x.\end{equation}
\end{lem}

\begin{proof}[Proof of Lemma~\ref{lem_main}]  Let $x_0\in(\mathcal{M}\cap U)$ and $R\in(0,\infty)$.  Since $U\subseteq\mathbb{R}^d$ is an open set, there exists $V\in\Proj(x_0)$ which satisfies that $V\subseteq U$.  Since $V$ is open, fix $R_0,\delta_0\in(0,\infty)$ such that for every $R\in(0,R_0]$, $\delta\in(0,\delta_0]$ it holds that
\begin{equation}\overline{V}_{R,\delta}(x_0)\subseteq V.\end{equation}
Due to the compactness of $\overline{V}_{R,\delta}(x_0)$ and the regularity of $f$, there exists $c\in(0,\infty)$ which satisfies for every $R\in(0,R_0]$, $\delta\in(0,\delta_0]$ that
\begin{equation}\label{lem_main_00}\norm{f}_{\C^3(V_{R,\delta}(x_0))}=\sup_{0\leq k\leq 3}\norm{\nabla^kf}_{L^\infty(V_{R,\delta_0}(x_0);\mathbb{R}^{(d^k)})}\leq c.\end{equation}
Let $x\in V_{R,\delta}(x_0)$.  By integration, since $\left.\nabla f\right|_{\mathcal{M}\cap U}=0$, it holds that
\begin{equation}\label{lem_main_1}\begin{aligned}\nabla f(x) & =  \int_0^1\big(\Hess f\big)(p(x)+s(x-p(x)))\cdot \left(x-p(x)\right)\ds \\ & =  \big(\Hess f\big)(p(x))\cdot \left(x-p(x)\right) \\ & \quad + \int_0^1\left(\big(\Hess f\big)(p(x)+s(x-p(x)))-\big(\Hess f\big)(p(x))\right)\cdot \left(x-p(x)\right)\ds. \end{aligned}\end{equation}
It follows from \eqref{lem_main_00}, the local regularity of $f$, and the definition of the projection that there exists $c\in(0,\infty)$ which satisfies that
\begin{equation}\label{lem_main_2}\begin{aligned}  \abs{\int_0^1\left(\big(\Hess f\big)(p(x)+s(x-p(x)))-\big(\Hess f\big)(p(x))\right)\cdot \left(x-p(x)\right)\ds} &  \leq  c\mathbf{d}(x,\mathcal{M}\cap U)^2\int_0^1 s\ds \\ & \leq c \mathbf{d}(x,\mathcal{M}\cap U)^2.\end{aligned}\end{equation}
After defining $\varepsilon_x\in\mathbb{R}^d$ which satisfies that
\begin{equation}\varepsilon_x =\int_0^1\left(\big(\Hess f\big)(p(x)+s(x-p(x)))-\big(\Hess f\big)(p(x))\right)\cdot \left(x-p(x)\right)\ds,\end{equation}
equation \eqref{lem_main_1} and estimate \eqref{lem_main_2} complete the proof of Lemma~\ref{lem_main}. \end{proof}

The following lemma will play an important role in the analysis of the deterministic and stochastic gradient descent algorithms in discrete time.  In the context of Lemma~\ref{lem_main}, for every $x\in\mathbb{R}^d$ with $p(x)\in(\mathcal{M}\cap U)$ well-defined, the following lemma quantifies the convergence of gradient descent to $\mathcal{M}\cap U$.

\begin{lem}\label{lem_projection}  Let $d\in\mathbb{N}$, $\mathfrak{d}\in\{1,2,\ldots,d-1\}$, let $\abs{\cdot}\colon\mathbb{R}^d\rightarrow\mathbb{R}$ be the standard norm on $\mathbb{R}^d$, let $U\subseteq\mathbb{R}^d$ be an open set, let $f\colon U\rightarrow\mathbb{R}$ be a three times continuously differentiable function, let $\mathcal{M}\subseteq\mathbb{R}^d$ satisfy that
\begin{equation}\mathcal{M}=\big\{\theta\in\mathbb{R}^d\colon [f(\theta)=\inf\nolimits_{\vartheta\in\mathbb{R}^d} f(\vartheta)]\big\},\end{equation}
let $\mathbf{d}(\cdot,\mathcal{M}\cap U):\mathbb{R}^d\rightarrow\mathbb{R}$ be the function which satisfies for every $x\in\mathbb{R}^d$ that
\begin{equation} \mathbf{d}(x,\mathcal{M}\cap U)=\inf \left\{\abs{x-y}\colon y\in(\mathcal{M}\cap U)\right\},\end{equation}
assume that $\mathcal{M}\cap U$ is a non-empty $\mathfrak{d}$-dimensional $\C^1$-submanifold of $\mathbb{R}^d$, and assume for every $\theta\in(\mathcal{M}\cap U)$ that $\rank((\Hess f)(\theta))=d-\mathfrak{d}$.  Then for every $x_0\in(\mathcal{M}\cap U)$ there exist $R_0,\delta_0,\mathfrak{r},\in(0,\infty)$, $\lambda\in(0,\infty)$ such that
\begin{equation}  \lambda\leq \max_{x\in \mathcal{M}\cap U\cap \overline{B}_R(x_0)}\abs{\big(\Hess f\big)(x)},\end{equation}
and $V\in\Proj(x_0)$ (cf. Definition~\ref{define_proj}) such that for every $R\in(0,R_0]$, $\delta\in(0,\delta_0]$, $r\in(0,\mathfrak{r}]$, $x\in V_{R,\delta}(x_0)$ it holds that
\begin{equation}\overline{V}_{R,\delta}(x_0)\subseteq V,\end{equation}
that
\begin{equation}\begin{aligned} \mathbf{d}\left(x- r\big(\Hess f\big)(p(x))\cdot(x-p(x)),\mathcal{M}\cap U\right) & \leq \abs{(x-p(x))- r\big(\Hess f\big)(p(x))\cdot(x-p(x))}\\ & \leq \left(1-\lambda r\right)\mathbf{d}(x,\mathcal{M}\cap U),\end{aligned}\end{equation}
and that
\begin{equation}\left(\big(\Hess f\big)(p(x))\cdot(x-p(x))\right)\cdot (x-p(x))\geq \lambda \mathbf{d}(x,\mathcal{M}\cap U)^2.\end{equation}
\end{lem}

\begin{proof}[Proof of Lemma~\ref{lem_projection}]  Let $x_0\in(\mathcal{M}\cap U)$.  Since $U\subseteq\mathbb{R}^d$ is an open subset, there exists $V\in\Proj(x_0)$ which satisfies that $V\subseteq U$.  Fix $R_0,\delta_0\in(0,\infty)$ such that every $R\in(0,R_0]$, $\delta\in(0,\delta_0]$ it holds that (cf. Definition~\ref{dtbn})
\begin{equation}\overline{V}_{R,\delta}(x_0)\subseteq V.\end{equation}
Due to the compactness of $\overline{V}_{R_0,\delta_0}(x_0)$ and the regularity of $f$, there exists $c\in(0,\infty)$ which satisfies for every $R\in(0,R_0]$, $\delta\in(0,\delta_0]$ that
\begin{equation}\label{lem_projection_00}\norm{f}_{\C^3(V_{R,\delta}(x_0))}\leq c.\end{equation}
Let $x\in V_{R,\delta}(x_0)$.  For the first claim, using \eqref{lem_projection_00}, fix $ \mathfrak{r}\in(0,\infty)$ which satisfies that
\begin{equation}\mathfrak{r}\left(\max_{x\in V_{R_0,\delta_0}(x_0)} \abs{\big(\Hess f\big)(p(x))}\right)\leq 1.\end{equation}
Let $ r\in(0, \mathfrak{r}]$.  The definition of the distance to $\mathcal{M}\cap U$ implies that
\begin{equation}\label{lem_projection_1}\mathbf{d}\left(x- r\big(\Hess f\big)(p(x))\cdot(x-p(x)), \mathcal{M}\cap U\right)\leq \abs{(x-p(x))- r\big(\Hess f\big)(p(x))\cdot(x-p(x))}.\end{equation}
Since the nondegeneracy assumption states that
\begin{equation}\rank((\Hess f)(p(x)))=d-\mathfrak{d}=\codim(\mathcal{M}\cap U),\end{equation}
Lemma~\ref{normal} below and \eqref{lem_projection_00} prove that there exists for $\lambda\in(0,\infty)$ which satisfies that
\begin{equation}\label{lem_projection_01}\lambda\leq\max_{x\in \mathcal{M}\cap U\cap \overline{B}_R(x_0)}\abs{\big(\Hess f\big)(p(x))},\end{equation}
for which we have that
\begin{equation}\label{lem_projection_2}\abs{(x-p(x))- r\big(\Hess f\big)(p(x))\cdot(x-p(x))}\leq (1- r\lambda)\abs{x-p(x)}=(1- r\lambda)\mathbf{d}(x,\mathcal{M}\cap U),\end{equation}
where the choice of $ \mathfrak{r}$ and \eqref{lem_projection_01} guarantee that $(1- r\lambda)\geq0$.  In combination, estimates \eqref{lem_projection_1}, \eqref{lem_projection_01}, and \eqref{lem_projection_2} complete the proof of the first claim.  The proof of the second claim is similar.  For every $x\in V_{R,\delta}(x_0)$, the nondegeneracy assumption, Lemma~\ref{normal}, and \eqref{lem_projection_00} prove that there exists $\lambda\in(0,\infty)$ which satisfies \eqref{lem_projection_01} such that
\begin{equation}\left(\big(\Hess f\big)(p(x))\cdot(x-p(x))\right)\cdot (x-p(x))\geq \lambda \abs{x-p(x)}^2= \lambda\mathbf{d}(x,\mathcal{M}\cap U)^2,\end{equation}
which completes the proof of Lemma~\ref{lem_projection}.  \end{proof}

\section{Continuous deterministic gradient descent}\label{section_det_cts}

In this section, for an objective function $f\colon\mathbb{R}^d\rightarrow\mathbb{R}$ which satisfies the conditions of Theorem~\ref{intro_ng_one_path}, we will analyze the local convergence to the local manifold of minima $\mathcal{M}\cap U$ of the deterministic gradient descent algorithm in continuous time $\theta_t\in\mathbb{R}^d$, $t\in[0,\infty)$, which satisfies for every $t\in(0,\infty)$ that
\begin{equation}\label{cts_sgd} \frac{d}{dt}\theta_t=-\nabla f(\theta_t).\end{equation}
We will prove that the solution of \eqref{cts_sgd} converges to the local manifold of minima $\mathcal{M}\cap U$, provided the initial condition is chosen in a sufficiently small neighborhood of $\mathcal{M}\cap U$.  The proof can be outlined as follows.  Given any $x_0\in \mathcal{M}\cap U$, we first fix an open neighborhood $x_0$ which satisfies the conclusions of Lemma~\ref{lem_main} and Lemma~\ref{lem_projection}.  Then, for initial data $\theta_0$ in this neighborhood, we quantify the convergence of the solution \eqref{cts_sgd} to $\mathcal{M}\cap U$ in directions normal to the manifold, using the decomposition of $\nabla f$ from Lemma~\ref{lem_main}.  Finally, after fixing a smaller neighborhood about $x_0$, we prove that the tangential components of the gradient of $\nabla f$ do not take the trajectory from the basin of attraction.

\begin{prop}\label{cts_converge}  Let $d\in\mathbb{N}$, $\mathfrak{d}\in\{1,2,\ldots,d-1\}$, let $\abs{\cdot}\colon\mathbb{R}^d\rightarrow\mathbb{R}$ be the standard norm on $\mathbb{R}^d$, let $U\subseteq\mathbb{R}^d$ be an open set, let $f\colon U\rightarrow\mathbb{R}$ be a three times continuously differentiable function, let $\mathcal{M}\subseteq\mathbb{R}^d$ satisfy that
\begin{equation}\mathcal{M}=\big\{\theta\in\mathbb{R}^d\colon [f(\theta)=\inf\nolimits_{\vartheta\in\mathbb{R}^d} f(\vartheta)]\big\},\end{equation}
let $\mathbf{d}(\cdot,\mathcal{M}\cap U):\mathbb{R}^d\rightarrow\mathbb{R}$ be the function which satisfies for every $x\in\mathbb{R}^d$ that
\begin{equation} \mathbf{d}(x,\mathcal{M}\cap U)=\inf \left\{\abs{x-y}\colon y\in(\mathcal{M}\cap U)\right\},\end{equation}
assume that $\mathcal{M}\cap U$ is a non-empty $\mathfrak{d}$-dimensional $\C^1$-submanifold of $\mathbb{R}^d$, and assume for every $\theta\in(\mathcal{M}\cap U)$ that $\rank((\Hess f)(\theta))=d-\mathfrak{d}$.  Then for every $x_0\in(\mathcal{M}\cap U)$ there exist $R_0,\delta_0,\lambda\in(0,\infty)$ such that for every $R\in(0,R_0]$, $\delta\in(0,\delta_0]$, $\theta_0\in V_{\nicefrac{R}{2},\delta}(x_0)$ (cf. Definition~\ref{dtbn}), for $\theta_t\in\mathbb{R}^d$, $t\in[0,\infty)$, which satisfies for every $t\in(0,\infty)$ that
\begin{equation}\frac{d}{dt}\theta_t=-\nabla f(\theta_t),\end{equation}
it holds for every $t\in[0,\infty)$ that
\begin{equation}\mathbf{d}(\theta_t,\mathcal{M}\cap U)\leq \exp(-\lambda t)\mathbf{d}(\theta_0,\mathcal{M}\cap U).\end{equation}
\end{prop}

\begin{proof}[Proof of Proposition~\ref{cts_converge}]  Let $x_0\in(\mathcal{M}\cap U)$.  Since $U\subseteq\mathbb{R}^d$ is an open set, fix $V\in\Proj(x_0)$ (cf. Definition~\ref{define_proj}) which satisfies that $V\subseteq U$.  In view or Proposition~\ref{cts_prop_tub}, fix $R_0,\delta_0\in(0,\infty)$ such that for every $R\in(0,R_0]$, $\delta\in(0,\delta_0]$ the set $V_{R,\delta}(x_0)$ (cf. Definition~\ref{dtbn}) satisfies that $\overline{V}_{R,\delta}(x_0)\subseteq V$ and that
\begin{equation}V_{R,\delta}(x_0) = \{x\in\mathbb{R}^d\colon\mathbf{d}(x,\mathcal{M}\cap U)=\mathbf{d}(x,\overline{B}_R(x_0)\cap\mathcal{M}\cap U)<\delta\}.\end{equation}
In particular, the compactness of $\overline{V}_{R_0,\delta_0}(x_0)$ and the regularity of $f$ imply that there exists $c\in(0,\infty)$ which satisfies that
\begin{equation}\label{cts_00}\norm{f}_{\C^3(V_{R_0,\delta_0}(x_0))}\leq c.\end{equation}
Let $R\in(0,R_0]$, $\delta\in(0,\delta_0]$.  Let $\theta_0\in V_{\nicefrac{R}{2},\delta}(x_0)$, let $\theta_t\in\mathbb{R}^d$, $t\in[0,\infty)$, satisfy for every $t\in(0,\infty)$ that
\begin{equation}\frac{d}{dt}\theta_t=-\nabla f(\theta_t),\end{equation}
and let $\tau\in(0,\infty)$ denote the exit time
\begin{equation}\tau =\inf \{\;t\geq 0\;|\;\theta_t\notin V_{R,\delta}(x_0)\;\}.\end{equation}
Lemma~\ref{distance} and the chain rule prove that
\begin{equation}\label{cts_1}\left\{\begin{aligned}  \frac{d}{dt}\mathbf{d}(\theta_t,\mathcal{M}\cap U)  & =-\nabla f(\theta_t)\cdot\nabla \mathbf{d}(\theta_t,\mathcal{M}\cap U) =-\nabla f(\theta_t)\cdot \frac{\theta_t-p(\theta_t)}{\abs{\theta_t-p(\theta_t)}} & \textrm{in}\;\;(0,\tau), \\  \frac{d}{dt}p(\theta_t)  & = -Dp(\theta_t)\cdot \nabla f(\theta_t) & \textrm{in}\;\;(0,\tau), \end{aligned}\right.\end{equation}
where the local regularity of $f$ and the stopping time $\tau$ guarantee the well-posedness of this equation.  Let $t\in(0,\tau)$.  It follows from Lemma~\ref{lem_main} and Lemma~\ref{lem_projection} that there exist $\lambda,c_1\in(0,\infty)$ which satisfy that
\begin{equation}\label{cts_0001}\nabla f(\theta_t)\cdot \frac{\theta_t-p(\theta_t)}{\abs{\theta_t-p(\theta_t)}}\geq \lambda\mathbf{d}(\theta_t,\mathcal{M}\cap U)-c_1\mathbf{d}(\theta_t,\mathcal{M}\cap U)^2.\end{equation}
Proposition~\ref{def_projection}, \eqref{cts_00}, and $\left.\nabla f\right|_{\mathcal{M}\cap U}=0$ prove that there exists $c_2\in(0,\infty)$ which satisfies that
\begin{equation}\label{cts_001}\abs{Dp(\theta_t)\cdot \nabla f(\theta_t)}\leq c_2\mathbf{d}(\theta_t,\mathcal{M}\cap U).\end{equation}
Returning to \eqref{cts_1}, it follows from \eqref{cts_0001} and \eqref{cts_001} that
\begin{equation}\label{cts_5}\left\{\begin{aligned}  \frac{d}{dt}\mathbf{d}(\theta_t,\mathcal{M}\cap U) & \leq -\lambda \mathbf{d}(\theta_t,\mathcal{M}\cap U)+c_1\mathbf{d}(\theta_t,\mathcal{M}\cap U)^2 & \textrm{in}\;\;(0,\tau), \\  \abs{\frac{d}{dt}p(\theta_t)}  & \leq c_2 \mathbf{d}(\theta_t,\mathcal{M}\cap U) & \textrm{in}\;\;(0,\tau).  \end{aligned}\right.\end{equation}
Let $\delta_1\in(0,\delta_0]$ satisfy that
\begin{equation}\label{cts_6}c_1\delta_1\leq \lambda/2.\end{equation}
Let $\delta\in(0,\delta_1]$.  For every $t\in(0,\tau)$ it follows from \eqref{cts_5} and \eqref{cts_6} that
\begin{equation}\frac{d}{dt}\mathbf{d}(\theta_t,\mathcal{M}\cap U)\leq -\frac{\lambda}{2}\mathbf{d}(\theta_t,\mathcal{M}\cap U).\end{equation}
Therefore, for every $\delta\in(0,\delta_1]$, $t\in[0,\tau)$ it holds that
\begin{equation}\label{cts_7}\mathbf{d}(\theta_t,\mathcal{M}\cap U)\leq \mathbf{d}(\theta_0,\mathcal{M}\cap U)\exp(-\lambda t/2)\leq \delta_1\exp(-\lambda t/2).\end{equation}
For every $t\in[0,\tau)$, it follows from \eqref{cts_5} and \eqref{cts_7} that
\begin{equation}\label{cts_8}\max_{0\leq t\leq \tau}\abs{p(\theta_t)-p(\theta_0)}\leq c_2\int_0^{\tau}\delta_1\exp\left(-\frac{\lambda t}{2}\right)\dt = \frac{2c_2\delta_1}{\lambda}\left(1-\exp\left(-\frac{\lambda \tau}{2}\right)\right)\leq \frac{2c_2\delta_1}{\lambda}.\end{equation}
Fix $\delta_2\in(0,\delta_1]$ which satisfies that
\begin{equation}\frac{2c_2\delta_2}{\lambda}<\frac{R}{2}.\end{equation}
Let $\delta\in(0,\delta_2]$.  In combination \eqref{cts_7}, \eqref{cts_8}, $\theta_0\in V_{\nicefrac{R}{2},\delta}(x_0)$, and the triangle inequality prove that $\theta_t\in V_{R,\delta}(x_0)$ for every $t\in(0,\infty)$.  This is to say that $\tau=\infty$.  Since $\theta_0\in V_{\nicefrac{R}{2},\delta}(x_0)$ was arbitrary, this completes the proof of Proposition~\ref{cts_converge}.  \end{proof}

\section{Discrete deterministic gradient descent}\label{section_det_dis}

In this section, for an objective function $f\colon \mathbb{R}^d\rightarrow\mathbb{R}$ which satisfies the conditions of Theorem~\ref{intro_ng_one_path}, we will analyze the convergence of the following deterministic gradient descent algorithm $\theta_n\in\mathbb{R}^d$, $n\in\mathbb{N}_0$, in discrete time which satisfies for a learning rate $\rho\in(0,1)$ and $r\in(0,\infty)$ that
\begin{equation}\label{dis_sgd} \theta_n=\theta_{n-1}-\frac{r}{n^\rho}\nabla f(\theta_{n-1}).\end{equation}

The proof is similar to the case of the deterministic gradient descent algorithm in continuous time.  However, in the discrete setting, care must be taken to choose the learning rate $r\in(0,\infty)$ sufficiently small.  Since, if the learning rate is too large, for small values of $n$ the jump $-\frac{r}{n^\rho}\nabla f$ may be an overcorrection that causes the solution to overshoot the local manifold of minima and to leave the basin of attraction.

In the proof, we first identify a basin of attraction using Proposition~\ref{def_projection} and Proposition~\ref{cts_prop_tub}.  In the second step, we prove that the solution \eqref{dis_sgd} converges along the normal directions to the manifold of local minima provided the solution remains in the basin of attraction.  For this, we use the normal component of $\nabla f$ from Lemma~\ref{lem_main} and the quantification of the convergence from Lemma~\ref{lem_projection}.  Finally, after fixing a perhaps smaller basin of attraction, we prove that the tangential component of the gradient from Lemma~\ref{lem_main} does not cause the solution \eqref{dis_sgd} to leave the basin of attraction.

\begin{prop}\label{dis_converge}  Let $d\in\mathbb{N}$, $\mathfrak{d}\in\{1,2,\ldots,d-1\}$, $\rho\in(0,1)$, let $\abs{\cdot}\colon\mathbb{R}^d\rightarrow\mathbb{R}$ be the standard norm on $\mathbb{R}^d$, let $U\subseteq\mathbb{R}^d$ be an open set, let $f\colon U\rightarrow\mathbb{R}$ be a three times continuously differentiable function, let $\mathcal{M}\subseteq\mathbb{R}^d$ satisfy that
\begin{equation}\mathcal{M}=\big\{\theta\in\mathbb{R}^d\colon [f(\theta)=\inf\nolimits_{\vartheta\in\mathbb{R}^d} f(\vartheta)]\big\},\end{equation}
let $\mathbf{d}(\cdot,\mathcal{M}\cap U):\mathbb{R}^d\rightarrow\mathbb{R}$ be the function which satisfies for every $x\in\mathbb{R}^d$ that
\begin{equation} \mathbf{d}(x,\mathcal{M}\cap U)=\inf \left\{\abs{x-y}\colon y\in(\mathcal{M}\cap U)\right\},\end{equation}
assume that $\mathcal{M}\cap U$ is a non-empty $\mathfrak{d}$-dimensional $\C^1$-submanifold of $\mathbb{R}^d$, and assume for every $\theta\in(\mathcal{M}\cap U)$ that $\rank((\Hess f)(\theta))=d-\mathfrak{d}$.  Then for every $x_0\in(\mathcal{M}\cap U)$ there exists $R_0,\delta_0,\mathfrak{r},c\in(0,\infty)$ such that for every $R\in(0,R_0]$, $\delta\in(0,\delta_0]$, $r\in(0,\mathfrak{r}]$, $\theta_0\in V_{\nicefrac{R}{2},\delta}(x_0)$ (cf. Definition~\ref{dtbn}), for $\theta_n\in\mathbb{R}^d$, $n\in\mathbb{N}_0$, which satisfies for every $n\in\mathbb{N}$ that
\begin{equation}\theta_n=\theta_{n-1}-\frac{r}{n^\rho}\nabla f(\theta_{n-1}),\end{equation}
it holds for every $n\in\mathbb{N}_0$ that
\begin{equation}\mathbf{d}(\theta_n,\mathcal{M}\cap U)\leq \exp(-c n^{1-\rho})\mathbf{d}(x_0,\mathcal{M}\cap U).\end{equation}
\end{prop}

\begin{proof}[Proof of Proposition~\ref{dis_converge}] Let $x_0\in(\mathcal{M}\cap U)$ and $\rho\in(0,1)$.  Since $U\subseteq\mathbb{R}^d$ is open, fix $V\in\Proj(x_0)$ (cf. Definition~\ref{define_proj}) which satisfies that $V\subseteq U$.  In view or Proposition~\ref{cts_prop_tub}, fix $R_0,\delta_0\in(0,\infty)$ such that for every $R\in(0,R_0]$, $\delta\in(0,\delta_0]$ the set $V_{R,\delta}(x_0)$ (cf. Definition~\ref{dtbn}) satisfies that $\overline{V}_{R,\delta}(x_0)\subseteq V$ and that
\begin{equation}V_{R,\delta}(x_0) = \{x\in\mathbb{R}^d\colon\mathbf{d}(x,\mathcal{M}\cap U)=\mathbf{d}(x,\overline{B}_R(x_0)\cap\mathcal{M}\cap U)<\delta\}.\end{equation}
The regularity of $f$ and the compactness of $\overline{V}_{R_0,\delta_0}(x_0)$ prove that there exists $c\in(0,\infty)$ which satisfies that
\begin{equation}\label{dis_000}\norm{f}_{\C^3(V_{R_0,\delta_0}(x_0))}\leq c.\end{equation}
Fix $\mathfrak{r}\in(0,\infty)$ which satisfies the conclusion of Lemma~\ref{lem_projection} for the set $V_{R_0,\delta_0}(x_0)$.  Let $R\in(0,R_0]$, $\delta\in(0,\delta_0]$, $r\in(0,\mathfrak{r}]$.  Let $\theta_0\in V_{\nicefrac{R}{2},\delta}(x_0)$, let $\theta_n\in\mathbb{R}^d$, $n\in\mathbb{N}$, satisfy that
\begin{equation}\theta_n=\theta_{n-1}-\frac{ r}{n^\rho}\nabla f(\theta_{n-1}),\end{equation}
and let $\tau\in\mathbb{N}$ be the exit time which satisfies that
\begin{equation}\tau=\inf \{\;n\in\mathbb{N}\;|\;\theta_n\notin V_{R,\delta}(x_0)\;\}.\end{equation}
Since for every $n\in\{1,\ldots,\tau\}$ the projection of $\theta_{n-1}$ is well-defined, we have that
\begin{equation}\label{dis_1}\mathbf{d}(\theta_n,\mathcal{M}\cap U) \leq  \abs{\theta_n-p(\theta_{n-1})} = \abs{\theta_{n-1}-p(\theta_{n-1})-\frac{ r}{n^\rho}\nabla f(\theta_{n-1})}.\end{equation}
Lemma~\ref{lem_main} proves that there exists $c\in(0,\infty)$ such that for every $n\in\{1,\ldots,\tau\}$ there exists $\varepsilon_n\in\mathbb{R}^d$ which satisfies that
\begin{equation}\label{dis_01}\abs{\varepsilon_n}\leq c\mathbf{d}(\theta_{n-1},\mathcal{M}\cap U)^2,\end{equation}
such that
\begin{equation}\label{dis_2}\nabla f(\theta_{n-1})=\big(\Hess f\big)(p(\theta_{n-1}))\cdot(x-p(x))+\varepsilon_n.\end{equation}
The triangle inequality, \eqref{dis_1}, \eqref{dis_01}, and \eqref{dis_2} prove that there exists $c_1\in(0,\infty)$ such that for every $n\in\{1,\ldots,\tau\}$ it holds that
\begin{equation}\label{dis_3}\begin{aligned} \mathbf{d}(\theta_n,\mathcal{M}\cap U) & \leq\abs{\theta_{n-1}-p(\theta_{n-1})-\frac{ r}{n^\rho}\big(\Hess f\big)(p(\theta_{n-1}))\cdot(\theta_{n-1}-p(\theta_{n-1}))} \\ & \quad +\frac{c_1 r}{n^\rho}\mathbf{d}(\theta_{n-1},\mathcal{M}\cap U)^2.\end{aligned}\end{equation}
Finally, the choice of $ \mathfrak{r}\in(0,\infty)$, Lemma~\ref{lem_projection}, and \eqref{dis_3} prove that there exists $\lambda\in(0,\infty)$ such that for every $n\in\{1,\ldots,\tau\}$ it holds that
\begin{equation}\label{dis_4}\mathbf{d}(\theta_n,\mathcal{M}\cap U) \leq\left(1-\frac{ r\lambda}{n^\rho}\right)\mathbf{d}(\theta_{n-1},\mathcal{M}\cap U)+\frac{c_1 r}{n^\rho}\mathbf{d}(\theta_{n-1},\mathcal{M}\cap U)^2,\end{equation}
where the choice of $ \mathfrak{r}\in(0,\infty)$ guarantees that $(1- r\lambda)\geq 0$.  Fix $\delta_1\in(0,\delta_0]$ which satisfies that
\begin{equation}\label{dis_16}c_1\delta_1\leq \frac{\lambda}{2}.\end{equation}
Let $\delta\in(0,\delta_1]$.  It follows from \eqref{dis_4} and \eqref{dis_16} that for every $n\in\{1,\ldots,\tau\}$ it holds that
\begin{equation}\mathbf{d}(\theta_n,\mathcal{M}\cap U)\leq \left(1-\frac{ r\lambda}{2 n^\rho}\right)\mathbf{d}(\theta_{n-1},\mathcal{M}\cap U).\end{equation}
After iterating this inequality, we have for every $n\in\{1,\ldots,\tau\}$ that
\begin{equation}\label{dis_5} \mathbf{d}(\theta_n,\mathcal{M}\cap U)\leq \prod_{k=1}^n \left(1-\frac{ r\lambda}{2 k^\rho}\right)\mathbf{d}(\theta_0,\mathcal{M}\cap U).\end{equation}
Since there exists $c\in(0,\infty)$ which satisfies for every $n\in\mathbb{N}$ that
\begin{equation}\label{dis_05}\log\left(\prod_{k=1}^n\left(1-\frac{ r\lambda}{2 k^\rho}\right)\right)=\sum_{k=1}^n\log\left(1-\frac{ r\lambda}{2 k^\rho}\right)\leq -c\sum_{k=1}^n\frac{ r\lambda}{2 k^\rho}\leq -c\frac{ r\lambda}{2}n^{1-\rho},\end{equation}
it follows from \eqref{dis_5} that there exists $c_2\in(0,\infty)$ which satisfies for every $n\in\{1,\ldots,\tau\}$ that
\begin{equation}\label{dis_6} \mathbf{d}(\theta_n,\mathcal{M}\cap U)\leq \exp\left(-c_2n^{1-\rho}\right)\mathbf{d}(\theta_0,\mathcal{M}\cap U).\end{equation}
It remains only to show that, provided $\delta\in(0,\delta_1]$ is chosen sufficiently small, we have that $\tau=\infty$.  It follows from \eqref{dis_000}, \eqref{dis_6}, and $\left.\nabla f\right|_{\mathcal{M}\cap U}=0$ that there exists $c\in(0,\infty)$ which satisfies that
\begin{equation}\abs{\theta_n-\theta_{n-1}}=\frac{ r}{n^\rho}\abs{\nabla f(\theta_{n-1})}\leq \frac{c}{n^\rho}\mathbf{d}(\theta_{n-1},\mathcal{M}\cap U)\leq cn^{-\rho}\exp\left(-c_2n^{1-\rho}\right)\mathbf{d}(\theta_0,\mathcal{M}\cap U).\end{equation}
The triangle inequality therefore implies that there exists $c_3\in(0,\infty)$ such that for every $n\in\{1,\ldots,\tau\}$ it holds that
\begin{equation}\label{dis_8}\abs{\theta_n-\theta_0}\leq c\mathbf{d}(\theta_0,\mathcal{M}\cap U)\sum_{k=1}^\infty ck^{-\rho}\exp\left(-c_2k^{1-\rho}\right)= c_3\mathbf{d}(\theta_0,\mathcal{M}\cap U)<\infty.\end{equation}
Fix $\delta_2\in(0,\delta_1]$ which satisfies that
\begin{equation}c_3\delta_2< \frac{R}{2}-2\delta_2.\end{equation}
Let $\delta\in(0,\delta_2]$.  The choice of $\delta_2\in(0,\delta_1]$, \eqref{dis_8}, and the triangle inequality prove for every $n\in\{1,\ldots,\tau\}$ that
\begin{equation}\label{dis_10}\abs{\theta_n-x_0}\leq \abs{\theta_n-\theta_0}+\abs{\theta_0-x_0}<c_3\delta_2+\frac{R}{2}+\delta_2<R-\delta_2.\end{equation}
In combination \eqref{dis_6} and \eqref{dis_10} prove for every $n\in\{1,\ldots,\tau\}$ that
\begin{equation}\mathbf{d}(\theta_n,\mathcal{M}\cap U)< \delta_2\;\;\textrm{and}\;\;\abs{\theta_n-x_0}\leq R-\delta_2.\end{equation}
The triangle inequality therefore implies for every $n\in\{1,\ldots,\tau\}$ that
\begin{equation}\mathbf{d}(\theta_n,\mathcal{M}\cap U)=\mathbf{d}(\theta_n,\overline{B}_R(x_0)\cap\mathcal{M}\cap U).\end{equation}
It follows from Proposition~\ref{cts_prop_tub}, the choice of $R_0,\delta_0\in(0,\infty)$, and $\theta_0\in V_{\nicefrac{R}{2},\delta}(x_0)$ that for every $n\in\mathbb{N}$ it holds that $\theta_n\in V_{R,\delta}(x_0)$.  This is to say that $\tau=\infty$, which completes the proof of Proposition~\ref{dis_converge}.  \end{proof}

\begin{remark}  The conclusion of Proposition~\ref{dis_converge} can be extended to the case of $\rho=1$ using the same techniques.  In this case, in the setting of Proposition~\ref{dis_converge}, there exists $R_0,\delta_0,\mathfrak{r},c\in(0,\infty)$ such that for every $R\in(0,R_0]$, $\delta\in(0,\delta_0]$, $r\in(0,\mathfrak{r}]$, $\theta_0\in V_{\nicefrac{R}{2},\delta}(x_0)$ (cf. Definition~\ref{dtbn}), for $\theta_n\in\mathbb{R}^d$, $n\in\mathbb{N}_0$, which satisfies for every $n\in\mathbb{N}$ that
\begin{equation}\theta_n=\theta_{n-1}-\frac{r}{n}\nabla f(\theta_{n-1}),\end{equation}
it holds for every $n\in\mathbb{N}_0$ that
\begin{equation}\mathbf{d}(\theta_n,\mathcal{M}\cap U)\leq \exp(-c \log(n))\mathbf{d}(x_0,\mathcal{M}\cap U).\end{equation}
The logarithm appears in estimate \eqref{dis_05} in the case $\rho=1$.  The remainder of the proof is then the same, where the only additional observation is that the analogue of \eqref{dis_8} is finite in the case $\rho=1$ as well.
\end{remark}

\section{Stochastic gradient descent}\label{sec_nonglobal}

In this section, in the setting of Theorem~\ref{intro_ng_one_path}, for a learning rate $\rho\in(\nicefrac{2}{3},1)$, for $r\in(0,\infty)$, $M\in\mathbb{N}$, for a bounded open subset $A\subseteq\mathbb{R}^d$, for a probability space $(\Omega,\mathcal{F},\mathbb{P})$, for a measurable space $(S,\mathcal{S})$, for a jointly measurable function $F\colon S\times\Omega\rightarrow\mathbb{R}$, for $X_{n,m}\colon\Omega\rightarrow\mathbb{R}^d$, $n,m\in\mathbb{N}$, i.i.d.\ random variables, we will analyze the convergence of the mini-batch stochastic gradient descent algorithm $\Theta_n\colon\Omega\rightarrow\mathbb{R}^d$, $n\in\mathbb{N}_0$, which satisfies that $\Theta_0$ is continuous uniformly distributed on $A$ and for every $n\in\mathbb{N}$ that
\begin{equation}\label{ng_sgd} \Theta_n=\Theta_{n-1}-\frac{r}{Mn^\rho}\sum_{m=1}^M\nabla_\theta F(\Theta_{n-1},X_{n,m}).\end{equation}®
The role of the mini-batch size $M\in\mathbb{N}$ is to reduce the variance of the random gradient
\begin{equation}\frac{1}{M}\sum_{m=1}^M\nabla_\theta F(\Theta_{n-1},X_{n,m}). \end{equation}
The variance reduction is quantified by the following well-known lemma, where the function $G$ plays the role of $\nabla_\theta F$.

\begin{lem}\label{variance_lem}  Let $d_1,d_2\in\mathbb{N}$, let $\abs{\cdot}\colon\mathbb{R}^{d_2}\rightarrow\mathbb{R}$ be the standard norm on $\mathbb{R}^{d_2}$, let $U\subseteq\mathbb{R}^{d_1}$ be a non-empty open set, let $(\Omega,\mathcal{F},\mathbb{P})$ be a probability space, let $(S,\mathcal{S})$ be a measurable space, let $G=(G(\theta,x))_{(\theta,x)\in\mathbb{R}^{d_1}\times S}\colon  \mathbb{R}^{d_1}\times S\rightarrow\mathbb{R}^{d_2}$ be a measurable function, let $X_m\colon  \Omega\rightarrow S$, $m\in\mathbb{N}$, be i.i.d.\ random variables, and assume for every non-empty compact set $\mathfrak{C}\subseteq U$ that $\sup\nolimits_{\theta\in \mathfrak{C}}\mathbb{E}\big[|G(\theta,X_1)|^2\big]<\infty$.  Then for every non-empty compact set $\mathfrak{C}\subseteq U$ there exists $c\in(0,\infty)$ which satisfies for every $M\in\mathbb{N}$ that
\begin{equation}\sup_{\theta\in \mathfrak{C}}\left(\mathbb{E}\left[\abs{\left[\frac{1}{M}\sum_{m=1}^M G(\theta,X_m)\right]-\mathbb{E}\big[G(\theta,X_1)\big]}^2\right]\right)\leq \frac{c}{M}.\end{equation}
\end{lem}

\begin{proof}[Proof of Lemma~\ref{variance_lem}]  Let $\mathfrak{C}\subseteq U$ be a compact set.  It holds for every $\theta\in \mathfrak{C}$, $M\in\mathbb{N}$ that
\begin{equation}\begin{aligned} & \mathbb{E}\Big[\Big|\frac{1}{M}\sum_{m=1}^M G(\theta,X_m)-\mathbb{E}\big[G(\theta,X_1)\big]\Big|^2\Big] \\ & =\frac{1}{M^2}\sum_{i,j=1}^M\mathbb{E}\Big[\big(G(\theta,X_i)-\mathbb{E}\big[G(\theta,X_1)\big]\big)\big(G(\theta,X_j)-\mathbb{E}\big[G(\theta,X_1)\big]\big)\Big].\end{aligned}\end{equation}
Since the $X_m$, $m\in\mathbb{N}$, are i.i.d.\ and since $G(\theta,X_{1,1})$, $\theta\in\mathbb{R}^{d_1}$, is locally bounded in $L^2(\Omega;\mathbb{R}^{d_2})$, there exists $c\in(0,\infty)$ which satisfies for every $M\in\mathbb{N}$ that
\begin{equation}\begin{aligned} \sup_{\theta\in \mathfrak{C}}\Big(\mathbb{E}\Big[\Big|\frac{1}{M}\sum_{m=1}^M G(\theta,X_m)-\mathbb{E}\big[G(\theta,X_1)\big]\Big|^2\Big]\Big) & =\sup_{\theta\in \mathfrak{C}}\Big(\frac{1}{M^2}\sum_{m=1}^M\mathbb{E}\Big[\Big|G(\theta,X_m)-\mathbb{E}\big[G(\theta,X_1)\big]\Big|^2\Big]\Big) \\ & = \frac{1}{M}\sup_{\theta\in \mathfrak{C}}\Big(\mathbb{E}\Big[\Big|G(\theta,X_1)-\mathbb{E}\big[G(\theta,X_1)\big]\Big|^2\Big]\Big) \\ & \leq \frac{c}{M}.\end{aligned}\end{equation}
This completes the proof of Lemma~\ref{variance_lem}.  \end{proof}

In the following proposition, much like the first step of the proofs of Proposition~\ref{cts_converge} and Proposition~\ref{dis_converge}, we establish the convergence of \eqref{ng_sgd} in directions normal to the local manifold of minima.  We first identify a basin of attraction for \eqref{ng_sgd} using Proposition~\ref{def_projection} and Proposition~\ref{cts_prop_tub} and prove, using the gradient decomposition of Lemma~\ref{lem_main} and the quantification of convergence from Lemma~\ref{lem_projection}, that on the event that SGD does not escape this basin of attraction SGD converges to the manifold of minima in expectation.

\begin{remark}We emphasize that the events $A_n$, $n\in\mathbb{N}_0$, defined in Proposition~\ref{ng_converge} below depend upon the quantifiers $n,M\in\mathbb{N}$, $r,R,\delta\in(0,\infty)$, $\theta\in\mathbb{R}^d$, and $x_0\in(\mathcal{M}\cap U)$.  However, in order to simplify the presentation, we will oftentimes suppress this dependence in the notation.  For every $n\in\mathbb{N}$, we will write $\mathbf{1}_{A_n}\colon\Omega\rightarrow\{0,1\}$ for the indicator function of the set $A_n\subseteq\Omega$.
\end{remark}

\begin{prop}\label{ng_converge}  Let $d\in\mathbb{N}$, $\mathfrak{d}\in\{ 0, 1, \ldots, d - 1 \}$, $\rho\in(\nicefrac{2}{3},1)$, let $\abs{\cdot}\colon\mathbb{R}^d\rightarrow\mathbb{R}$ be the standard norm on $\mathbb{R}^d$, let $U\subseteq\mathbb{R}^d$ be an open set, let $(\Omega,\mathcal{F},\mathbb{P})$ be a probability space, let $(S,\mathcal{S})$ be a measurable space, let $F=(F(\theta,x))_{(\theta,x)\in\mathbb{R}^d\times S}\colon  \mathbb{R}^d\times S\rightarrow\mathbb{R}$ be a measurable function, let $X_{n,m}\colon  \Omega\rightarrow S$, $n,m\in\mathbb{N}$, be i.i.d.\ random variables which satisfy for every $\theta\in\mathbb{R}^d$ that $\mathbb{E}\big[ |F(\theta,X_{1,1})|^2\big]<\infty$, let $f\colon\mathbb{R}^d\rightarrow\mathbb{R}$ be the function which satisfies for every $\theta\in\mathbb{R}^d$ that $f(\theta)=\mathbb{E}\big[F(\theta,X_{1,1})\big]$, let $\mathcal{M}\subseteq\mathbb{R}^d$ satisfy that
\begin{equation}\mathcal{M}=\big\{\theta\in\mathbb{R}^d\colon [f(\theta)=\inf\nolimits_{\vartheta\in\mathbb{R}^d} f(\vartheta)]\big\},\end{equation}
let $\mathbf{d}(\cdot,\mathcal{M}\cap U):\mathbb{R}^d\rightarrow\mathbb{R}$ be the function which satisfies for every $x\in\mathbb{R}^d$ that
\begin{equation} \mathbf{d}(x,\mathcal{M}\cap U)=\inf \left\{\abs{x-y}\colon y\in(\mathcal{M}\cap U)\right\},\end{equation}
assume for every $x\in S$ that $\mathbb{R}^d\ni\theta\mapsto F(\theta,x)\in\mathbb{R}$ is a locally Lipschitz continuous function, assume that $f|_U\colon U\rightarrow\mathbb{R}$ is a three times continuously differentiable function, assume for every non-empty compact set $\mathfrak{C}\subseteq U$ that $\sup\nolimits_{\theta\in \mathfrak{C}}\mathbb{E}\big[|F(\theta,X_{1,1})|^2+|(\nabla_\theta F)(\theta,X_{1,1})|^2\big]<\infty$, assume that $\mathcal{M}\cap U$ is a non-empty $\mathfrak{d}$-dimensional $\C^1$-submanifold of $\mathbb{R}^d$, assume for every $\theta\in(\mathcal{M}\cap U)$ that $\rank((\Hess f)(\theta))=d-\mathfrak{d}$, for every $M\in\mathbb{N}$, $r\in(0,\infty)$, $\theta\in\mathbb{R}^d$ let $\Theta^{M,r}_{0,\theta}\in\mathbb{R}^d\colon\Omega\rightarrow\mathbb{R}^d$ satisfy for every $\omega\in\Omega$ that $\theta^{M,r}_{0,\theta}(\omega)=\theta$, for every $n,M\in\mathbb{N}$, $r\in(0,\infty)$, $\theta\in\mathbb{R}^d$ let $\Theta^{M,r}_{n,\theta}\colon\Omega\rightarrow\mathbb{R}^d$ satisfy that
\begin{equation}\Theta^{M,r}_{n,\theta}=\Theta^{M,r}_{n-1,\theta}-\frac{r}{n^\rho M}\!\left[\sum_{m=1}^M(\nabla_\theta F)(\Theta^{M,r}_{n-1,\theta},X_{n,m})\right],\end{equation}
and for every $n,M\in\mathbb{N}$, $r,R,\delta\in(0,\infty)$, $\theta\in\mathbb{R}^d$, $x_0\in(\mathcal{M}\cap U)$ let $A_n(M,r,R,\delta,\theta,x_0)\in\mathcal{F}$ satisfy that
\begin{equation}A_n(M,r,R,\delta,\theta,x_0)=\Big\{\forall\;m\in\{0,\ldots,n\}\;\Theta^{M,r}_{m,\theta}\in V_{R,\delta}(x_0)\Big\}.\end{equation}
Then for every $x_0\in(\mathcal{M}\cap U)$ there exist $R_0,\delta_0,\mathfrak{r},c\in(0,\infty)$ such that for every $R\in(0,R_0]$, $\delta\in(0,\delta_0]$, $r\in(0,\mathfrak{r}]$, $n,M\in\mathbb{N}$, $\theta\in V_{R,\delta}(x_0)$ (cf. Definition~\ref{dtbn}) it holds that 
\begin{equation}\left(\mathbb{E}\left[\left(\mathbf{d}(\Theta^{M,r}_{n,\theta},\mathcal{M}\cap U)\wedge 1\right)^2\mathbf{1}_{A_{n-1}}\right]\right)^\frac{1}{2}\leq cn^{-\frac{\rho}{2}}.\end{equation}
\end{prop}

\begin{proof}[Proof of Proposition~\ref{ng_converge}]  Let $x_0\in(\mathcal{M}\cap U)$.  Since $U\subseteq\mathbb{R}^d$ is open, fix $V\in\Proj(x_0)$ (cf. Definition~\ref{define_proj}) which satisfies that $V\subseteq U$.  Fix $R_0,\delta_0\in(0,\infty)$ which satisfy the conclusion of Proposition~\ref{cts_prop_tub} for this set $V$.  Finally, fix $ \mathfrak{r}\in(0,\infty)$ which satisfies the conclusion of Lemma~\ref{lem_projection}.  Let $R\in(0,R_0]$, $\delta\in(0,\delta_0]$, $r\in(0,\mathfrak{r}]$, $n,M\in\mathbb{N}$.   To simplify the notation, and by a small abuse of notation, let $\nabla_\theta F^{M,n}\colon\mathbb{R}^d\times\Omega\rightarrow\mathbb{R}^d$, $n\in\mathbb{N}$, be the functions which satisfy for every $(\theta,\omega)\in\mathbb{R}^d\times\Omega$ that
\begin{equation}\nabla_\theta F^{M,n}(\theta)=\nabla_\theta F^{M,n}(\theta,\omega)=\frac{1}{M}\sum_{m=1}^M(\nabla_\theta F)(\theta,X_{n,m}(\omega)).\end{equation}
Let $\theta\in V_{R,\delta}(x_0)$, let $\Theta^{M,r}_{0,\theta}\colon\Omega\rightarrow\mathbb{R}^d$ satisfy for every $\omega\in\Omega$ that $\Theta^{M,r}_{0,\theta}(\omega)=\theta$, and for every $n\in\mathbb{N}$ let $\Theta^{M,r}_{n,\theta}\colon\Omega\rightarrow\mathbb{R}^d$ satisfy that
\begin{equation}\label{ng_sgd_proof}\Theta^{M,r}_{n,\theta}=\Theta^{M,r}_{n-1,\theta}-\frac{r}{n^\rho}\nabla_\theta F^{M,n}(\Theta^{M,r}_{n-1,\theta}).\end{equation}
We will analyze the solution $\Theta^{M,r}_{n,\theta}$ of \eqref{ng_sgd_proof} on the event $A_{n-1}$.  We observe that
\begin{equation}\label{sgd_0}\Theta^{M,r}_{n,\theta}=\Theta^{M,r}_{n-1,\theta}-\frac{ r}{n^\rho}\nabla f(\Theta^{M,r}_{n-1,\theta})+\frac{ r}{n^\rho}\left(\nabla f(\Theta^{M,r}_{n-1,\theta})-\nabla_\theta F^{M,n}(\Theta^{M,r}_{n-1,\theta})\right).\end{equation}
Since the event $A_{n-1}$ implies that $\Theta^{M,r}_{n-1,\theta}\in V_{R,\delta}(x_0)\subseteq V$, the projection of $\Theta^{M,r}_{n-1,\theta}$ is well-defined and it holds by definition of the distance to $\mathcal{M}\cap U$ that
\begin{equation}\begin{aligned}\label{sgd_000} & \mathbf{d}(\Theta^{M,r}_{n,\theta},\mathcal{M}\cap U)^2 \\ & \leq  \abs{\Theta^{M,r}_{n,\theta}-p(\Theta^{M,r}_{n-1,\theta})}^2 \\ & \leq   \abs{\Theta^{M,r}_{n-1,\theta}-p(\Theta^{M,r}_{n-1,\theta})-\frac{ r}{n^\rho}\nabla f(\Theta^{M,r}_{n-1,\theta})}^2  \\ & \quad +2\left(\Theta^{M,r}_{n-1,\theta}-p(\Theta^{M,r}_{n-1,\theta})-\frac{ r}{n^\rho}\nabla f(\Theta^{M,r}_{n-1,\theta})\right)\cdot\frac{ r}{n^\rho}\left(\nabla f(\Theta^{M,r}_{n-1,\theta})-\nabla_\theta F^{M,n}(\Theta^{M,r}_{n-1,\theta})\right) \\ & \quad + \abs{\frac{ r}{n^\rho}\left(\nabla f(\Theta^{M,r}_{n-1,\theta})-\nabla_\theta F^{M,n}(\Theta^{M,r}_{n-1,\theta})\right)}^2.\end{aligned}\end{equation}
The three terms on the righthand side of \eqref{sgd_000} will be treated separately.  For the first term on the righthand side of \eqref{sgd_000}, the choice of $ \mathfrak{r}\in(0,\infty)$, Lemma~\ref{lem_main}, and Lemma~\ref{lem_projection} prove, following identically the proof leading from \eqref{dis_1} to \eqref{dis_4}, that there exist $\lambda, c\in(0,\infty)$ such that
\begin{equation}\abs{\Theta^{M,r}_{n-1,\theta}-p(\Theta^{M,r}_{n-1,\theta})-\frac{ r}{n^\rho}\nabla f(\Theta^{M,r}_{n-1,\theta})}\leq \left(1-\frac{ r\lambda}{n^\rho}\right)\mathbf{d}(\Theta^{M,r}_{n-1,\theta},\mathcal{M}\cap U)+c\frac{ r}{n^\rho}\mathbf{d}(\Theta^{M,r}_{n-1,\theta},\mathcal{M}\cap U)^2.\end{equation}
Therefore, there exist $\lambda,c\in(0,\infty)$ which satisfy that
\begin{equation}\label{sgd_2}\begin{aligned} \abs{\Theta^{M,r}_{n-1,\theta}-p(\Theta^{M,r}_{n-1,\theta})-\frac{ r}{n^\rho}\nabla f(\Theta^{M,r}_{n-1,\theta})}^2 & \leq \left(1-\frac{ r\lambda}{n^\rho}\right)^2\mathbf{d}(\Theta^{M,r}_{n-1,\theta},\mathcal{M}\cap U)^2 \\ & \quad +c\left(1-\frac{ r\lambda}{n^\rho}\right)\frac{ r}{n^\rho}\mathbf{d}(\Theta^{M,r}_{n-1,\theta},\mathcal{M}\cap U)^3  \\ & \quad +c\frac{ r^2}{n^{2\rho}}\mathbf{d}(\Theta^{M,r}_{n-1,\theta},\mathcal{M}\cap U)^4.\end{aligned}\end{equation}
The remaining two terms of \eqref{sgd_000} and the righthand side of \eqref{sgd_2} will be handled after taking the expectation on the event $A_{n-1}\subseteq\Omega$ which satisfies that
\begin{equation}\label{sgd_02}A_{n-1}=\big\{\omega\in\Omega\colon \Theta^{M,r}_{m,\theta}\in V_{R,\delta}(x_0)\;\forall\;m\in\{0,\ldots,n-1\}\big\}.\end{equation}
After returning to \eqref{sgd_000}, it follows from \eqref{sgd_2} that there exists $c\in(0,\infty)$ which satisfies that
\begin{equation}\label{sgd_3}\begin{aligned}  & \mathbb{E}\left[\mathbf{d}(\Theta^{M,r}_{n,\theta},\mathcal{M}\cap U)^2\mathbf{1}_{A_{n-1}}\right] \\   & \leq \left(1-\frac{ r\lambda}{n^\rho}\right)^2\mathbb{E}\left[\mathbf{d}(\Theta^{M,r}_{n-1,\theta},\mathcal{M}\cap U)^2\mathbf{1}_{A_{n-1}}\right] \\ & \quad  +c\left(1-\frac{ r\lambda}{n^\rho}\right)\frac{ r}{n^\rho}\mathbb{E}\left[\mathbf{d}(\Theta^{M,r}_{n-1,\theta},\mathcal{M}\cap U)^3\mathbf{1}_{A_{n-1}}\right] +c\frac{ r^2}{n^{2\rho}}\mathbb{E}\left[\mathbf{d}(\Theta^{M,r}_{n-1,\theta},\mathcal{M}\cap U)^4\mathbf{1}_{A_{n-1}}\right] \\ & \quad + 2\mathbb{E}\left[\left(\Theta^{M,r}_{n-1,\theta}-\frac{ r}{n^\rho}\nabla f(\Theta^{M,r}_{n-1,\theta})-p(\Theta^{M,r}_{n-1,\theta})\right)\cdot\frac{ r}{n^\rho}\left(\nabla f(\Theta^{M,r}_{n-1,\theta})-\nabla_\theta F^{M,n}(\Theta^{M,r}_{n-1,\theta})\right)\mathbf{1}_{A_{n-1}}\right] \\ & \quad + \mathbb{E}\left[\abs{\frac{ r}{n^\rho}\left(\nabla f(\Theta^{M,r}_{n-1,\theta})-\nabla_\theta F^{M,n}(\Theta^{M,r}_{n-1,\theta})\right)}^2\mathbf{1}_{A_{n-1}}\right].\end{aligned}\end{equation}
For every $m\in\mathbb{R}$ let $\mathcal{F}_m\subseteq\mathcal{F}$ be the sigma algebra which satisfies that
\begin{equation}\label{sgd_sigma}\mathcal{F}_m=\sigma\big(\{X_{1,k}\}_{k=1}^M,\ldots,\{X_{m,k}\}_{k=1}^M\big).\end{equation}
For the penultimate term of \eqref{sgd_3}, since $\mathbf{1}_{A_{n-1}}$ is $\mathcal{F}_{n-1}$-measurable, properties of the conditional expectation imply that
\begin{equation}\begin{aligned}& \mathbb{E}\Big[\Big(\Theta^{M,r}_{n-1,\theta}-\frac{ r}{n^\rho}\nabla f(\Theta^{M,r}_{n-1,\theta})-p(\Theta^{M,r}_{n-1,\theta})\Big)\cdot\frac{ r}{n^\rho}\Big(\nabla f(\Theta^{M,r}_{n-1,\theta})-\nabla_\theta F^{M,n}(\Theta^{M,r}_{n-1,\theta})\Big)\mathbf{1}_{A_{n-1}}\Big] = \\ & \mathbb{E}\Big[\mathbb{E}\Big[\Big(\Theta^{M,r}_{n-1,\theta}-\frac{ r}{n^\rho}\nabla f(\Theta^{M,r}_{n-1,\theta})-p(\Theta^{M,r}_{n-1,\theta})\Big)\cdot\frac{ r}{n^\rho}\Big(\nabla f(\Theta^{M,r}_{n-1,\theta})-\nabla_\theta F^{M,n}(\Theta^{M,r}_{n-1,\theta})\Big)\mathbf{1}_{A_{n-1}}|\mathcal{F}_{n-1}\Big]\Big]. \end{aligned}\end{equation}
Therefore, it holds that
\begin{equation}\label{sgd_4}\begin{aligned} & \mathbb{E}\Big[\Big(\Theta^{M,r}_{n-1,\theta}-\frac{ r}{n^\rho}\nabla f(\Theta^{M,r}_{n-1,\theta})-p(\Theta^{M,r}_{n-1,\theta})\Big)\cdot\frac{ r}{n^\rho}\Big(\nabla f(\Theta^{M,r}_{n-1,\theta})-\nabla_\theta F^{M,n}(\Theta^{M,r}_{n-1,\theta})\Big)\mathbf{1}_{A_{n-1}}\Big] = \\ & \mathbb{E}\Big[\Big(\Theta^{M,r}_{n-1,\theta}-\frac{ r}{n^\rho}\nabla f(\Theta^{M,r}_{n-1,\theta})-p(\Theta^{M,r}_{n-1,\theta})\Big)\mathbf{1}_{A_{n-1}}\cdot\mathbb{E}\Big[\frac{ r}{n^\rho}\Big(\nabla f(\Theta^{M,r}_{n-1,\theta})-\nabla_\theta F^{M,n}(\Theta^{M,r}_{n-1,\theta})\Big)|\mathcal{F}_{n-1}\Big]\Big] \\ & =0,\end{aligned}\end{equation}
where the final equality follows from the fact that the $X_{m,k}$, $m,k\in\mathbb{N}$, are independent and therefore satisfy for every $x\in\mathbb{R}^d$ that
\begin{equation}\label{sgd_04} \mathbb{E}\left[\frac{ r}{n^\rho}\left(\nabla f(x)-\nabla_\theta F^{M,n}(x)\right)|\mathcal{F}_{n-1}\right] = \frac{ r}{Nn^\rho}\sum_{m=1}^M\mathbb{E}\left[\nabla f(x)-\nabla_\theta F(x,X_{n,m})\right] = 0.\end{equation}
The final term of \eqref{sgd_3} is handled using Lemma~\ref{variance_lem}.  Since $\overline{V}_{R,\delta}(x_0)$ is compact, the independence of the $X_{m,k}$, $m,k\in\mathbb{N}$, and Lemma~\ref{variance_lem} prove that there exists $c\in(0,\infty)$ such that
\begin{equation}\label{sgd_5}\mathbb{E}\left[\abs{\frac{ r}{n^\rho}\left(\nabla f(\Theta^{M,r}_{n-1,\theta})-\nabla_\theta F^{M,n}(\Theta^{M,r}_{n-1,\theta})\right)\mathbf{1}_{A_{n-1}}}^2\right]\leq \frac{ cr^2}{Mn^{2\rho}}.\end{equation}
Returning to \eqref{sgd_3}, it follows from \eqref{sgd_4} and \eqref{sgd_5} that there exists $c_1\in(0,\infty)$ such that
\begin{equation}\label{sgd_6}\begin{aligned} & \mathbb{E}\left[\mathbf{d}(\Theta^{M,r}_{n,\theta},\mathcal{M}\cap U)^2\mathbf{1}_{A_{n-1}}\right] \leq \\ &\left(1-\frac{ r\lambda}{n^\rho}\right)^2\mathbb{E}\left[\mathbf{d}(\Theta^{M,r}_{n-1,\theta},\mathcal{M}\cap U)^2\mathbf{1}_{A_{n-1}}\right] +c_1\left(1-\frac{ r\lambda}{n^\rho}\right)\frac{ r}{n^\rho}\left[\mathbf{d}(\Theta^{M,r}_{n-1,\theta},\mathcal{M}\cap U)^3\mathbf{1}_{A_{n-1}}\right]  \\ & +c_1\frac{ r^2}{n^{2\rho}}\mathbb{E}\left[\mathbf{d}(\Theta^{M,r}_{n-1,\theta},\mathcal{M}\cap U)^4\mathbf{1}_{A_{n-1}}\right] + c_1\frac{ r^2}{Mn^{2\rho}}.\end{aligned}\end{equation}
Fix $\delta_1\in(0,\delta_0]$ which satisfies that
\begin{equation}\delta_1\leq \frac{\lambda}{2c_1}\;\;\textrm{and}\;\;\delta_1^2\leq\frac{\lambda}{2c_1 r}.\end{equation}
Let $\delta\in(0,\delta_1]$.  We claim that inequality \eqref{sgd_6} implies that there exists some $c\in(0,\infty)$ which satisfies for every $n\in\mathbb{N}$ that
\begin{equation}\label{sgd_05}\mathbb{E}\left[\left(\mathbf{d}(\Theta^{M,r}_{n,\theta},\mathcal{M}\cap U)\wedge 1\right)^2\mathbf{1}_{A_{n-1}}\right]^\frac{1}{2}\leq cn^{-\frac{\rho}{2}}.\end{equation}
The proof of \eqref{sgd_05} will proceed by induction.  Since $\rho\in(\nicefrac{2}{3},1)$, there exists $n_0\geq 1$ such that for every $n\geq n_0$ it holds that
\begin{equation}\label{sgd_06}\left(n^\rho-(n-1)^\rho- r\lambda+\frac{ r^2\lambda^2}{n^\rho}\right)\leq \left(\rho(n-1)^{\rho-1}- r\lambda+\frac{ r^2\lambda^2}{n^\rho}\right)\leq -\frac{ r \lambda}{2},\end{equation}
where the first inequality follows from the mean value theorem and $\rho\in(\nicefrac{2}{3},1)$ and the second inequality is obtained by choosing $n\in\mathbb{N}$ sufficiently large.  Fix $n_0\geq 1$ which satisfies \eqref{sgd_06} and define $\overline{c}\in(0,\infty)$ which satisfies that
\begin{equation}\overline{c} =\max\left\{\left(n_0-1\right)^\rho,\frac{2c_1 r}{M\lambda}\right\}.\end{equation}
For the base case, the definition of $\overline{c}$ guarantees for every $n\in\{1,\ldots,n_0-1\}$ that
\begin{equation}\label{sgd_006}\mathbb{E}\left[\left(\mathbf{d}(\Theta^{M,r}_{n,\theta},\mathcal{M}\cap U)\wedge 1\right)^2\mathbf{1}_{A_{n-1}}\right]\leq \overline{c}n^{-\rho}.\end{equation}
For the induction step, suppose that for $n\geq n_0$ we have that
\begin{equation}\mathbb{E}\left[\left(\mathbf{d}(\Theta^{M,r}_{n-1,\theta},\mathcal{M}\cap U)\wedge 1\right)^2\mathbf{1}_{A_{n-2}}\right]\leq\overline{c}(n-1)^{-\rho}.\end{equation}
Since the event $A_{n-1}$ implies that
\begin{equation}\mathbf{d}(\Theta^{M,r}_{n-1,\theta},\mathcal{M}\cap U)\leq\delta\leq 1,\end{equation}
it follows from an $L^\infty$-estimate, the inclusion $A_{n-1}\subseteq A_{n-2}$, and the induction hypothesis that for every $m\in\{2,3,4\}$ it holds that
\begin{equation}\mathbb{E}\left[\mathbf{d}(\Theta^{M,r}_{n-1,\theta},\mathcal{M}\cap U)^m\mathbf{1}_{A_{n-1}}\right]\leq \delta^{m-2}\mathbb{E}\left[\mathbf{d}(\Theta^{M,r}_{n-1,\theta},\mathcal{M}\cap U)^2\mathbf{1}_{A_{n-2}}\right]\leq\delta^{m-2}\overline{c}(n-1)^{-\rho}.\end{equation}
Returning to \eqref{sgd_6}, it holds that
\begin{equation}\begin{aligned} \mathbb{E}\left[\mathbf{d}(\Theta^{M,r}_{n,\theta},\mathcal{M}\cap U)^2\mathbf{1}_{A_{n-1}}\right] & \leq   \overline{c}\left(1-\frac{ r\lambda}{n^\rho}\right)^2(n-1)^{-\rho}+\overline{c}c_1\delta\left(1-\frac{ r\lambda}{n^\rho}\right)\frac{ r}{n^\rho}(n-1)^{-\rho} \\ &\quad + \overline{c}c_1\delta^2\frac{ r^2}{n^{2\rho}}(n-1)^{-\rho}+c_1\frac{ r^2}{Mn^{2\rho}}. \end{aligned}\end{equation}
After adding and subtracting $\overline{c}n^{-\rho}$, it holds that
\begin{equation}\label{sgd_7}\begin{aligned} \mathbb{E}& \left[\mathbf{d}(\Theta^{M,r}_{n,\theta},\mathcal{M}\cap U)^2\mathbf{1}_{A_{n-1}}\right] \leq  \overline{c}n^{-\rho}  \\ & + n^{-\rho}\left(\overline{c}(n-1)^{-\rho}\left(n^\rho-(n-1)^\rho-2 r\lambda+\frac{ r^2\lambda^2}{n^\rho}+c_1\delta r\left(1-\frac{ r\lambda}{n^\rho}\right)+c_1\delta^2\frac{ r^2}{n^\rho}\right)+c_1\frac{ r^2}{Mn^\rho}\right).\end{aligned}\end{equation}
Since $\delta\in(0,\delta_1]$, it follows from \eqref{sgd_7} that
\begin{equation}\label{sgd_8}\mathbb{E} \left[\mathbf{d}(\Theta^{M,r}_{n,\theta},\mathcal{M}\cap U)^2\mathbf{1}_{A_{n-1}}\right] \leq  \overline{c}n^{-\rho} + n^{-\rho}\left(\overline{c}(n-1)^{-\rho}\left(n^\rho-(n-1)^\rho- r\lambda+\frac{ r^2\lambda^2}{n^\rho}\right)+c_1\frac{ r^2}{Mn^\rho}\right).\end{equation}
Since $n\geq n_0$, the choice $\overline{c}\geq \frac{2c_1 r}{M\lambda}$, \eqref{sgd_06}, and \eqref{sgd_8} prove that
\begin{equation}\mathbb{E} \left[\mathbf{d}(\Theta^{M,r}_{n,\theta},\mathcal{M}\cap U)^2\mathbf{1}_{A_{n-1}}\right] \leq  \overline{c}n^{-\rho} + n^{-\rho}\left(-\frac{ r\lambda}{2}\overline{c}(n-1)^{-\rho}+c_1\frac{ r^2}{Mn^\rho}\right)\leq\overline{c}n^{-\rho}.\end{equation}
Therefore, we have that
\begin{equation}\label{sgd_9}\mathbb{E}\left[\left(\mathbf{d}(\Theta^{M,r}_{n,\theta},\mathcal{M}\cap U)\wedge 1\right)^2\mathbf{1}_{A_{n-1}}\right]\leq \left[\mathbf{d}(\Theta^{M,r}_{n,\theta},\mathcal{M}\cap U)^2\mathbf{1}_{A_{n-1}}\right]\leq \overline{c}n^{-\rho},\end{equation}
which completes the induction step.  Since the base case is \eqref{sgd_006}, this completes the proof of Proposition~\ref{ng_converge}.  \end{proof}

Proposition~\ref{ng_converge} proves the convergence of SGD to $\mathcal{M}\cap U$ on the event that SGD remains in a basin of attraction.  It remains necessary to prove that, provided the mini-batch size is chosen to be sufficiently large, SGD remains in the basin of attraction for large times.  We prove the first step toward this goal in the proposition below, which estimates the maximal excursion of SGD on the event that the dynamics do not leave a basin of attraction.

\begin{prop}\label{ng_tang}  Let $d\in\mathbb{N}$, $\mathfrak{d}\in\{ 0, 1, \ldots, d - 1 \}$, $\rho\in(\nicefrac{2}{3},1)$, let $\abs{\cdot}\colon\mathbb{R}^d\rightarrow\mathbb{R}$ be the standard norm on $\mathbb{R}^d$, let $U\subseteq\mathbb{R}^d$ be an open set, let $(\Omega,\mathcal{F},\mathbb{P})$ be a probability space, let $(S,\mathcal{S})$ be a measurable space, let $F=(F(\theta,x))_{(\theta,x)\in\mathbb{R}^d\times S}\colon  \mathbb{R}^d\times S\rightarrow\mathbb{R}$ be a measurable function, let $X_{n,m}\colon  \Omega\rightarrow S$, $n,m\in\mathbb{N}$, be i.i.d.\ random variables which satisfy for every $\theta\in\mathbb{R}^d$ that $\mathbb{E}\big[ |F(\theta,X_{1,1})|^2\big]<\infty$, let $f\colon\mathbb{R}^d\rightarrow\mathbb{R}$ be the function which satisfies for every $\theta\in\mathbb{R}^d$ that $f(\theta)=\mathbb{E}\big[F(\theta,X_{1,1})\big]$, let $\mathcal{M}\subseteq\mathbb{R}^d$ satisfy that
\begin{equation}\mathcal{M}=\big\{\theta\in\mathbb{R}^d\colon [f(\theta)=\inf\nolimits_{\vartheta\in\mathbb{R}^d} f(\vartheta)]\big\},\end{equation}
assume for every $x\in S$ that $\mathbb{R}^d\ni\theta\mapsto F(\theta,x)\in\mathbb{R}$ is a locally Lipschitz continuous function, assume that $f|_U\colon U\rightarrow\mathbb{R}$ is a three times continuously differentiable function, assume for every non-empty compact set $\mathfrak{C}\subseteq U$ that $\sup\nolimits_{\theta\in \mathfrak{C}}\mathbb{E}\big[|F(\theta,X_{1,1})|^2+|(\nabla_\theta F)(\theta,X_{1,1})|^2\big]<\infty$, assume that $\mathcal{M}\cap U$ is a non-empty $\mathfrak{d}$-dimensional $\C^1$-submanifold of $\mathbb{R}^d$, assume for every $\theta\in(\mathcal{M}\cap U)$ that $\rank((\Hess f)(\theta))=d-\mathfrak{d}$, for every $M\in\mathbb{N}$, $r\in(0,\infty)$, $\theta\in\mathbb{R}^d$ let $\Theta^{M,r}_{0,\theta}\in\mathbb{R}^d\colon\Omega\rightarrow\mathbb{R}^d$ satisfy for every $\omega\in\Omega$ that $\theta^{M,r}_{0,\theta}(\omega)=\theta$, for every $n,M\in\mathbb{N}$, $r\in(0,\infty)$, $\theta\in\mathbb{R}^d$ let $\Theta^{M,r}_{n,\theta}\colon\Omega\rightarrow\mathbb{R}^d$ satisfy that
\begin{equation}\Theta^{M,r}_{n,\theta}=\Theta^{M,r}_{n-1,\theta}-\frac{r}{n^\rho M}\!\left[\sum_{m=1}^M(\nabla_\theta F)(\Theta^{M,r}_{n-1,\theta},X_{n,m})\right],\end{equation}
and for every $n,M\in\mathbb{N}$, $r,R,\delta\in(0,\infty)$, $\theta\in\mathbb{R}^d$, $x_0\in(\mathcal{M}\cap U)$ let $A_n(M,r,R,\delta,\theta,x_0)\in \mathcal{F}$ satisfy that
\begin{equation}A_n(M,r,R,\delta,\theta,x_0)=\Big\{\forall\;m\in\{0,\ldots,n\}\;\Theta^{M,r}_{m,\theta}\in V_{R,\delta}(x_0)\Big\}.\end{equation}
Then for every $x_0\in(\mathcal{M}\cap U)$ there exist $R_0,\delta_0,\mathfrak{r},c\in(0,\infty)$ such that for every $R\in(0,R_0]$, $\delta\in(0,\delta_0]$, $r\in(0,\mathfrak{r}]$, $n,M\in\mathbb{N}$, $\theta\in V_{\nicefrac{R}{2},\delta}(x_0)$ (cf. Definition~\ref{dtbn}) it holds that
\begin{equation}\mathbb{E}\left[\max_{1\leq k\leq n}\abs{\Theta^{M,r}_{k,\theta}-\Theta^{M,r}_{0,\theta}}\mathbf{1}_{A_{k-1}}\right]\leq \sum_{k=1}^n \left(\mathbb{E}\left[\abs{\Theta^{M,r}_{k,\theta}-\Theta^{M,r}_{k-1,\theta}}^2\mathbf{1}_{A_{k-1}}\right]\right)^{\frac{1}{2}} \leq cr\left(1+M^{-\frac{1}{2}}n^{1-\rho}\right).\end{equation}
\end{prop}

\begin{proof}[Proof of Proposition~\ref{ng_tang}]  Let $\mathbf{d}(\cdot,\mathcal{M}\cap U):\mathbb{R}^d\rightarrow\mathbb{R}$ be the function which satisfies for every $x\in\mathbb{R}^d$ that
\begin{equation} \mathbf{d}(x,\mathcal{M}\cap U)=\inf \left\{\abs{x-y}\colon y\in(\mathcal{M}\cap U)\right\}.\end{equation}
Let $x_0\in(\mathcal{M}\cap U)$.  Since $U\subseteq\mathbb{R}^d$ is open, fix $V\in\Proj(x_0)$ (cf. Definition~\ref{define_proj}) which satisfies that $V\subseteq U$.  Fix $R_0,\delta_0\in(0,\infty)$ which satisfies the conclusion of Proposition~\ref{cts_prop_tub} for this set $V$.  We observe that the regularity of $f$ and the compactness of $\overline{V}_{R_0,\delta_0}(x_0)$ imply that
\begin{equation}\label{tang_0000}\norm{f}_{\C^3(V_{R_0,\delta_0}(x_0))}\leq c.\end{equation}
Finally, fix $ \mathfrak{r}\in(0,\infty)$ which satisfies the conclusion of Lemma~\ref{lem_projection}.  Let $R\in(0,R_0]$, $\delta\in(0,\delta_0]$, $r\in(0,\mathfrak{r}]$, $n,M\in\mathbb{N}$.   As in Proposition~\ref{ng_converge}, let $\nabla_\theta F^{M,n}\colon\mathbb{R}^d\times\Omega\rightarrow\mathbb{R}^d$, $n\in\mathbb{N}$, be the functions which satisfy for every $(\theta,\omega)\in\mathbb{R}^d\times\Omega$ that
\begin{equation}\nabla_\theta F^{M,n}(\theta)=\nabla_\theta F^{M,n}(\theta,\omega)=\frac{1}{M}\sum_{m=1}^M(\nabla_\theta F)(\theta,X_{n,m}(\omega)).\end{equation}
Let $\theta\in V_{\nicefrac{R}{2},\delta}(x_0)$, let $\Theta^{M,r}_{0,\theta}\colon\Omega\rightarrow\mathbb{R}^d$ satisfy for every $\omega\in\Omega$ that $\Theta^{M,r}_{0,\theta}(\omega)=\theta$, and for every $n\in\mathbb{N}$ let $\Theta^{M,r}_{n,\theta}\colon\Omega\rightarrow\mathbb{R}^d$ satisfy that
\begin{equation}\Theta^{M,r}_{n,\theta}=\Theta^{M,r}_{n-1,\theta}-\frac{r}{n^\rho}\nabla_\theta F^{M,n}(\Theta^{M,r}_{n-1,\theta}).\end{equation}
We will first prove that there exists $c\in(0,\infty)$ which satisfies that
\begin{equation}\label{tang_00}\mathbb{E}\left[\abs{\Theta^{M,r}_{n,\theta}-\Theta^{M,r}_{n-1,\theta}}^2\mathbf{1}_{A_{n-1}}\right]^\frac{1}{2}\leq c\left(\frac{ r}{n^{\frac{3}{2}\rho}}+\frac{ r}{n^\rho M^\frac{1}{2}}\right),\end{equation}
where we observe that the constant $c\in(0,\infty)$ can be absorbed by fixing $r\in(0, \mathfrak{r}]$ sufficiently small.  It holds that
\begin{equation}\Theta^{M,r}_{n,\theta}=\Theta^{M,r}_{n-1,\theta}-\frac{ r}{n^\rho}\nabla f(\Theta^{M,r}_{n-1,\theta})+\frac{ r}{n^\rho}\left(\nabla f(\Theta^{M,r}_{n-1,\theta})-\nabla F^{M,n}(\Theta^{M,r}_{n-1,\theta})\right).\end{equation}
Lemma~\ref{lem_main} proves that there exists $c_1\in(0,\infty)$ and $\varepsilon_n\colon A_{n-1}\rightarrow\mathbb{R}^d$ which satisfy that
\begin{equation}\label{tang_000}\abs{\varepsilon_n}\leq c_1\mathbf{d}(\Theta^{M,r}_{n-1},\mathcal{M}\cap U)^2,\end{equation}
such that on the event $A_{n-1}$ it holds that
\begin{equation}\nabla f(\Theta^{M,r}_{n-1,\theta})=\big(\Hess f\big)(p(\Theta^{M,r}_{n-1,\theta}))\cdot(\Theta^{M,r}_{n-1,\theta}-p(\Theta^{M,r}_{n-1,\theta}))+\varepsilon_n.\end{equation}
Therefore, on the event $A_{n-1}$ it holds that
\begin{equation}\label{tang_1}\begin{aligned}  \Theta^{M,r}_{n,\theta} & = \Theta^{M,r}_{n-1,\theta}-\frac{ r}{n^\rho}\big(\Hess f\big)(p(\Theta^{M,r}_{n-1,\theta}))\cdot \big(\Theta^{M,r}_{n-1,\theta}-p(\Theta^{M,r}_{n-1,\theta})\big)-\frac{ r}{n^\rho}\varepsilon_n \\ & \quad +\frac{ r}{n^\rho}\big(\nabla f(\Theta^{M,r}_{n-1,\theta})-\nabla F^{M,n}(\Theta^{M,r}_{n-1,\theta})\big).\end{aligned}\end{equation}
Let $\tilde{\Theta}^{M,r}_{n-1,\theta}\colon A_{n-1}\rightarrow\mathbb{R}^d$ satisfy that
\begin{equation}\tilde{\Theta}^{M,r}_{n-1,\theta}=\Theta^{M,r}_{n-1,\theta}-\frac{ r}{n^\rho}\big(\Hess f\big)(p(\Theta^{M,r}_{n-1,\theta}))\cdot\left(\Theta^{M,r}_{n-1,\theta}-p(\Theta^{M,r}_{n-1,\theta})\right).\end{equation}
After taking the norm-squared of \eqref{tang_1}, on the event $A_{n-1}$ it holds that
\begin{equation}\label{tang_2}\begin{aligned}\abs{\Theta^{M,r}_{n,\theta}-\tilde{\Theta}^{M,r}_{n-1,\theta}}^2 & =  \frac{ r^2}{n^{2\rho}}\abs{\varepsilon_n}^2-2\frac{ r^2}{n^{2\rho}}\varepsilon_n\cdot\left(\nabla f(\Theta^{M,r}_{n-1,\theta})-\nabla F^{M,n}(\Theta^{M,r}_{n-1,\theta})\right) \\ & \quad +\frac{ r^2}{n^{2\rho}}\abs{\nabla f(\Theta^{M,r}_{n-1,\theta})-\nabla F^{M,n}(\Theta^{M,r}_{n-1,\theta})}^2.\end{aligned}\end{equation}
We will estimate \eqref{tang_2} by taking the expectation on the event $A_{n-1}$.  The first term on the righthand side of \eqref{tang_2} is handled using Proposition~\ref{ng_converge} and \eqref{tang_000}.  For the second term, from \eqref{sgd_sigma} we recall the sigma algebras $\mathcal{F}_m\subseteq\mathcal{F}$, $m\in\mathbb{N}$, which satisfy that
\begin{equation}\mathcal{F}_m=\sigma\big(\{X_{1,k}\}_{k=1}^M,\ldots,\{X_{m,k}\}_{k=1}^M\big).\end{equation}
Since $\varepsilon_n\colon A_{n-1}\rightarrow\mathbb{R}^d$ is $\mathcal{F}_{n-1}$-measurable, it follows identically to \eqref{sgd_4} and \eqref{sgd_04} that
\begin{equation}\label{tang_3}\mathbb{E}\left[\varepsilon_n\cdot\left(\nabla f(\Theta^{M,r}_{n-1,\theta})-\nabla F^{M,n}(\Theta^{M,r}_{n-1,\theta})\right)\mathbf{1}_{A_{n-1}}\right]=0.\end{equation}
For the final term on the righthand side of \eqref{tang_2}, the compactness of $\overline{V}_{R_0,\delta_0}(x_0)$, the independence of the $X_{m,k}$, $m,k\in\mathbb{N}$, and  Lemma~\ref{variance_lem} prove that there exists $c\in(0,\infty)$ which satisfies that
\begin{equation}\label{tang_4}\mathbb{E}\left[\abs{\nabla f(\Theta^{M,r}_{n-1,\theta})-\nabla F^{M,n}(\Theta^{M,r}_{n-1,\theta})}^2\mathbf{1}_{A_{n-1}}\right]\leq  \frac{c}{M}.\end{equation}
In combination, Proposition~\ref{ng_converge} and estimates \eqref{tang_000}, \eqref{tang_2}, \eqref{tang_3}, and \eqref{tang_4} prove that there exists $c\in(0,\infty)$ which satisfies that
\begin{equation}\label{tang_5}\begin{aligned}\mathbb{E}\left[\abs{\Theta^{M,r}_{n,\theta}-\tilde{\Theta}^{M,r}_{n-1,\theta}}^2\mathbf{1}_{A_{n-1}}\right] & \leq c\left(\frac{ r^2\delta^2}{n^{2\rho}}\mathbb{E}\left[\mathbf{d}(\Theta^{M,r}_{n-1,\theta},\mathcal{M}\cap U)^2\right]+\frac{ r^2}{n^{2\rho}M}\right) \\ & \leq  c\left(\frac{ r^2\delta^2}{n^{3\rho}}+\frac{ r^2}{n^{2\rho}M}\right).\end{aligned}\end{equation}
It follows from the definition of $\tilde{\Theta}^{M,r}_{n-1,\theta}$, \eqref{tang_0000}, and the definition of the projection that, on the event $A_{n-1}$ there exists $c\in(0,\infty)$ which satisfies that
\begin{equation}\abs{\tilde{\Theta}^{M,r}_{n-1,\theta}-\Theta^{M,r}_{n-1,\theta}}^2=\frac{ r^2}{n^{2\rho}}\abs{\big(\Hess f\big)(p(\Theta^{M,r}_{n-1,\theta}))\cdot(\Theta^{M,r}_{n-1,\theta}-p(\Theta^{M,r}_{n-1,\theta}))}^2\leq c\frac{ r^2}{n^{2\rho}}\mathbf{d}(\Theta^{M,r}_{n-1,\theta},\mathcal{M}\cap U)^2.\end{equation}
Proposition~\ref{ng_converge} proves that there exists $c\in(0,\infty)$ such that
\begin{equation}\label{tang_6}\mathbb{E}\left[\abs{\tilde{\Theta}^{M,r}_{n-1,\theta}-\Theta^{M,r}_{n-1,\theta}}^2\mathbf{1}_{A_{n-1}}\right]\leq \frac{c r^2}{n^{3\rho}}.\end{equation}
It follows from the triangle inequality, \eqref{tang_5}, and \eqref{tang_6} that there exists $c_1\in(0,\infty)$ which satisfies that
\begin{equation}\label{tang_03}\begin{aligned}\mathbb{E}\left[\abs{\Theta^{M,r}_{n,\theta}-\Theta^{M,r}_{n-1,\theta}}^2\mathbf{1}_{A_{n-1}}\right]^\frac{1}{2} & \leq   \mathbb{E}\left[\abs{\Theta^{M,r}_{n,\theta}-\tilde{\Theta}^{M,r}_{n-1,\theta}}^2\mathbf{1}_{A_{n-1}}\right]^\frac{1}{2}+\mathbb{E}\left[\abs{\tilde{\Theta}^{M,r}_{n-1,\theta}-\Theta^{M,r}_{n-1,\theta}}^2\mathbf{1}_{A_{n-1}}\right]^\frac{1}{2} \\ & \leq c_1\left(\frac{ r}{n^{\frac{3}{2}\rho}}+\frac{ r}{n^\rho M^\frac{1}{2}}\right),\end{aligned}\end{equation}
which completes the proof of \eqref{tang_00}.  Since for every $r\leq s\in\mathbb{N}_0$ we have $\mathbf{1}_{A_s}\leq \mathbf{1}_{A_r}$, it follows from \eqref{tang_03}, the triangle inequality, and H\"older's inequality that there exists $c_2\in(0,\infty)$ which satisfies for every $ r\in(0, \mathfrak{r}]$ that
\begin{equation}\label{tang_003}\begin{aligned}\mathbb{E}\left[\max_{1\leq k\leq n}\abs{\Theta^{M,r}_{k,\theta}-\Theta^{M,r}_{0,\theta}}\mathbf{1}_{A_{k-1}}\right] & \leq  \sum_{k=1}^n\mathbb{E}\left[\abs{\Theta^{M,r}_{k,\theta}-\Theta^{M,r}_{k-1,\theta}}\mathbf{1}_{A_{k-1}}\right] \\  & \leq  \sum_{k=1}^n\mathbb{E}\left[\abs{\Theta^{M,r}_{k,\theta}-\Theta^{M,r}_{k-1,\theta}}^2\mathbf{1}_{A_{k-1}}\right]^{\frac{1}{2}} \\ & \leq  c_1 r\left(\sum_{k=1}^n k^{-\frac{3}{2}\rho}+M^{-\frac{1}{2}}\sum_{k=1}^n k^{-\rho}\right) \\ & \leq  c_2 r\left(1+M^{-\frac{1}{2}}n^{1-\rho}\right),\end{aligned}\end{equation}
where we have used that fact that, since $\rho\in(\nicefrac{2}{3},1)$, there exists a $c\in(0,\infty)$ such that
\begin{equation}\label{rho_sum}\sum_{k=1}^n k^{-\frac{3}{2}\rho}+M^{-\frac{1}{2}}\sum_{k=1}^n k^{-\rho}\leq c\left(1+M^{-\frac{1}{2}}n^{1-\rho}\right).\end{equation}
This completes the proof of Proposition~\ref{ng_tang}. \end{proof}

\begin{remark}\label{rho_remark} We emphasize that the assumption $\rho\in(\nicefrac{2}{3},1)$ is only used to ensure the boundedness in $n\in\mathbb{N}$ of the first sum appearing on the lefthand side of \eqref{rho_sum}, which cannot be countered by the mini-batch size $M\in\mathbb{N}$.  Every other argument in the paper applies without change to the case $\rho\in(0,1)$.  In particular, because the result of Proposition~\ref{ng_tang} is not needed if $\mathcal{M}\cap U$ is compact, since SGD cannot leave the basin of attraction in tangential directions, the results of Section~\ref{stoch_discrete} apply for $\rho\in(0,1)$ under this additional compactness assumption.\end{remark}

We will next obtain a lower bound in probability for the events $A_n$, $n\in\mathbb{N}_0$.  For this, we will first establish sufficient conditions for containment in the set $V_{R,\delta}(x_0)$.  Effectively, these conditions split the normal and tangential movement of SGD in the sense that, in order to be outside the set $V_{R,\delta}(x_0)$, a point must be either distance greater than $\delta$ from $\mathcal{M}\cap U$ or be of distance roughly greater than $R$ from $x_0$.

\begin{lem}\label{lem_tub}  Let $d\in\mathbb{N}$, $\mathfrak{d}\in\{ 0, 1, \ldots, d - 1 \}$, let $\abs{\cdot}\colon\mathbb{R}^d\rightarrow\mathbb{R}$ be the standard norm on $\mathbb{R}^d$, and let $\mathcal{N}\subseteq\mathbb{R}^d$  be a $\mathfrak{d}$-dimensional $\C^1$-submanifold, let $\mathbf{d}(\cdot,\mathcal{N}):\mathbb{R}^d\rightarrow\mathbb{R}$ be the function which satisfies for every $x\in\mathbb{R}^d$ that
\begin{equation} \mathbf{d}(x,\mathcal{N})=\inf \left\{\abs{x-y}\colon y\in\mathcal{N}\right\}.\end{equation}
Then for every $x_0\in\mathcal{N}$ there exists $R_0,\delta_0\in(0,\infty)$ such that for every $R\in(0,R_0]$, $\delta\in(0,\delta_0]$, for $V_{R,\delta}(x_0)\subseteq\mathbb{R}^d$ which satisfies that
\begin{equation}V_{R,\delta}(x_0)=\{x+v\in\mathbb{R}^d\colon x\in \overline{B}_R(x_0)\cap\mathcal{N}\;\textrm{and}\;v\in\big(T_x\mathcal{N})^\perp\;\textrm{with}\;\abs{v}<\delta\},\end{equation}
it holds that
\begin{equation}\{x\in\mathbb{R}^d\colon\mathbf{d}(x,\mathcal{N})<\delta\;\textrm{and}\;\abs{x-x_0}\leq R-\delta\;\}\subseteq V_{R,\delta}(x_0).\end{equation}
\end{lem}

\begin{proof}[Proof of Lemma~\ref{lem_tub}]  Let $x_0\in\mathcal{N}$, let $V\in\Proj(x_0)$ (cf. Definition~\ref{define_proj}), and let $R_0,\delta_0\in(0,\infty)$ satisfy the conclusion of Proposition~\ref{cts_prop_tub}.  That is, for every $R\in(0,R_0]$, $\delta\in(0,\delta_0]$ it holds that $\overline{V}_{R,\delta}(x_0)\subseteq V$ and that
\begin{equation}V_{R,\delta}(x_0) = \{x\in\mathbb{R}^d\colon\mathbf{d}(x,\mathcal{N})=\mathbf{d}(x,\overline{B}_R(x_0)\cap\mathcal{N})<\delta\}.\end{equation}
Suppose that $x\in\mathbb{R}^d$ satisfies that
\begin{equation}\label{lem_tub_1}\mathbf{d}(x,\mathcal{N})<\delta\;\;\textrm{with}\;\;\abs{x-x_0}\leq R-\delta.\end{equation}
The definition of the distance to $\mathcal{N}$ and $\abs{x-x_0}\leq R-\delta$ imply that there exists a possibly non-unique $\tilde{x}\in\overline{\mathcal{N}}$ which satisfies that
\begin{equation}\abs{x-\tilde{x}}=\mathbf{d}(x,\mathcal{N})<\delta.\end{equation}
The triangle inequality implies that
\begin{equation}\abs{\tilde{x}-x_0}\leq \abs{\tilde{x}-x}+\abs{x-x_0}<\delta+(R-\delta)<R.\end{equation}
It follows that $\tilde{x}\in \overline{\overline{B}_R(x_0)\cap\mathcal{N}}$, and therefore that
\begin{equation}\label{lem_tub_2}\mathbf{d}(x,\mathcal{N})=\mathbf{d}(x,\overline{B}_R(x_0)\cap\mathcal{N})<\delta.\end{equation}
It follows from \eqref{lem_tub_1} and \eqref{lem_tub_2} that $x\in V_{R,\delta}(x_0)$, which completes the proof of Lemma~\ref{lem_tub}.  \end{proof}

In the following proposition, we obtain a lower bound in probability for the sets $A_n$, $n\in\mathbb{N}_0$.  The interesting observation is that Proposition~\ref{ng_converge} and Proposition~\ref{ng_tang}, which obtain estimates for the solution of \eqref{ng_sgd} conditioned on the events $A_n$, $n\in\mathbb{N}_0$, can be used together and inductively to obtain lower bound in probability for the events $A_n$, $n\in\mathbb{N}_0$.  Namely, Proposition~\ref{ng_converge} implies that, on the event $A_{n-1}$, the process is converging to $\mathcal{M}\cap U$ in the normal directions with high probability, and Proposition~\ref{ng_tang} can be used to estimate the probability that the solution \eqref{ng_sgd} escapes the basin of attraction along the tangential directions.  We first introduce some convenient notation.

\begin{prop}\label{ng_good_set}  Let $d\in\mathbb{N}$, $\mathfrak{d}\in\{ 0, 1, \ldots, d - 1 \}$, $\rho\in(\nicefrac{2}{3},1)$, let $\abs{\cdot}\colon\mathbb{R}^d\rightarrow\mathbb{R}$ be the standard norm on $\mathbb{R}^d$, let $U\subseteq\mathbb{R}^d$ be an open set, let $(\Omega,\mathcal{F},\mathbb{P})$ be a probability space, let $(S,\mathcal{S})$ be a measurable space, let $F=(F(\theta,x))_{(\theta,x)\in\mathbb{R}^d\times S}\colon  \mathbb{R}^d\times S\rightarrow\mathbb{R}$ be a measurable function, let $X_{n,m}\colon  \Omega\rightarrow S$, $n,m\in\mathbb{N}$, be i.i.d.\ random variables which satisfy for every $\theta\in\mathbb{R}^d$ that $\mathbb{E}\big[ |F(\theta,X_{1,1})|^2\big]<\infty$, let $f\colon\mathbb{R}^d\rightarrow\mathbb{R}$ be the function which satisfies for every $\theta\in\mathbb{R}^d$ that $f(\theta)=\mathbb{E}\big[F(\theta,X_{1,1})\big]$, let $\mathcal{M}\subseteq\mathbb{R}^d$ satisfy that
\begin{equation}\mathcal{M}=\big\{\theta\in\mathbb{R}^d\colon [f(\theta)=\inf\nolimits_{\vartheta\in\mathbb{R}^d} f(\vartheta)]\big\},\end{equation}
let $(\cdot)_+\colon\mathbb{R}\rightarrow\mathbb{R}$ be the function which satisfies for every $x\in\mathbb{R}$ that
\begin{equation} (x)_+=\max(x,0),\end{equation}
assume for every $x\in S$ that $\mathbb{R}^d\ni\theta\mapsto F(\theta,x)\in\mathbb{R}$ is a locally Lipschitz continuous function, assume that $f|_U\colon U\rightarrow\mathbb{R}$ is a three times continuously differentiable function, assume for every non-empty compact set $\mathfrak{C}\subseteq U$ that $\sup\nolimits_{\theta\in \mathfrak{C}}\mathbb{E}\big[|F(\theta,X_{1,1})|^2+|(\nabla_\theta F)(\theta,X_{1,1})|^2\big]<\infty$, assume that $\mathcal{M}\cap U$ is a non-empty $\mathfrak{d}$-dimensional $\C^1$-submanifold of $\mathbb{R}^d$, assume for every $\theta\in(\mathcal{M}\cap U)$ that $\rank((\Hess f)(\theta))=d-\mathfrak{d}$, for every $M\in\mathbb{N}$, $r\in(0,\infty)$, $\theta\in\mathbb{R}^d$ let $\Theta^{M,r}_{0,\theta}\in\mathbb{R}^d\colon\Omega\rightarrow\mathbb{R}^d$ satisfy for every $\omega\in\Omega$ that $\theta^{M,r}_{0,\theta}(\omega)=\theta$, for every $n,M\in\mathbb{N}$, $r\in(0,\infty)$, $\theta\in\mathbb{R}^d$ let $\Theta^{M,r}_{n,\theta}\colon\Omega\rightarrow\mathbb{R}^d$ satisfy that
\begin{equation}\Theta^{M,r}_{n,\theta}=\Theta^{M,r}_{n-1,\theta}-\frac{r}{n^\rho M}\!\left[\sum_{m=1}^M(\nabla_\theta F)(\Theta^{M,r}_{n-1,\theta},X_{n,m})\right],\end{equation}
and for every $n,M\in\mathbb{N}$, $r,R,\delta\in(0,\infty)$, $\theta\in\mathbb{R}^d$, $x_0\in(\mathcal{M}\cap U)$ let $A_n(M,r,R,\delta,\theta,x_0)\in\mathcal{F}$ satisfy that
\begin{equation}A_n(M,r,R,\delta,\theta,x_0)=\Big\{\forall\;m\in\{0,\ldots,n\}\;\Theta^{M,r}_{m,\theta}\in V_{R,\delta}(x_0)\Big\}.\end{equation}
Then for every $x_0\in(\mathcal{M}\cap U)$ there exist $R_0,\delta_0,\mathfrak{r},c\in(0,\infty)$ such that for every $R\in(0,R_0]$, $\delta\in(0,\delta_0]$, $r\in(0,\mathfrak{r}]$, $n,M\in\mathbb{N}$, $\theta\in V_{\nicefrac{R}{2},\delta}(x_0)$ (cf. Definition~\ref{dtbn}) it holds that

\begin{equation}\mathbb{P}[A_n]\geq \prod_{k=1}^n\left(1-\frac{c}{Mk^{2\rho}}\right)_+-cM^{-1}n^{1-\rho}-\frac{cr\left(1+M^{-\frac{1}{2}}n^{1-\rho}\right)}{\left(\frac{R}{2}-2\delta\right)_+}.\end{equation}
\end{prop}

\begin{proof}[Proof of Proposition~\ref{ng_good_set}]  Let $\mathbf{d}(\cdot,\mathcal{M}\cap U):\mathbb{R}^d\rightarrow\mathbb{R}$ be the function which satisfies for every $x\in\mathbb{R}^d$ that
\begin{equation} \mathbf{d}(x,\mathcal{M}\cap U)=\inf \left\{\abs{x-y}\colon y\in(\mathcal{M}\cap U)\right\}.\end{equation}
Let $x_0\in(\mathcal{M}\cap U)$.  Since $U\subseteq\mathbb{R}^d$ is open, fix $V\in\Proj(x_0)$ (cf. Definition~\ref{define_proj}) which satisfies that $V\subseteq U$.  Fix $R_0,\delta_0\in(0,\infty)$ which satisfy the conclusion of Proposition~\ref{cts_prop_tub} for this set $V$.  Fix $ \mathfrak{r}\in(0,\infty)$ which satisfies the conclusion of Lemma~\ref{lem_projection}.  Let $R\in(0,R_0]$, $\delta\in(0,\delta_0]$, $r\in(0,\mathfrak{r}]$, $n,M\in\mathbb{N}$.  As in Proposition~\ref{ng_converge}, let $\nabla_\theta F^{M,n}\colon\mathbb{R}^d\times\Omega\rightarrow\mathbb{R}^d$, $n\in\mathbb{N}$, be the functions which satisfy for every $(\theta,\omega)\in\mathbb{R}^d\times\Omega$ that
\begin{equation}\nabla_\theta F^{M,n}(\theta)=\nabla_\theta F^{M,n}(\theta,\omega)=\frac{1}{M}\sum_{m=1}^M(\nabla_\theta F)(\theta,X_{n,m}(\omega)).\end{equation}
Let $\theta\in V_{\nicefrac{R}{2},\delta}(x_0)$, let $\Theta^{M,r}_{0,\theta}\colon\Omega\rightarrow\mathbb{R}^d$ satisfy for every $\omega\in\Omega$ that $\Theta^{M,r}_{0,\theta}(\omega)=\theta$, and for every $n\in\mathbb{N}$ let $\Theta^{M,r}_{n,\theta}\colon\Omega\rightarrow\mathbb{R}^d$ satisfy that
\begin{equation}\Theta^{M,r}_{n,\theta}=\Theta^{M,r}_{n-1,\theta}-\frac{r}{n^\rho}\nabla_\theta F^{M,n}(\Theta^{M,r}_{n-1,\theta}).\end{equation}
Since it holds that
\begin{equation} \mathbf{d}(\Theta^{M,r}_{n,\theta},\mathcal{M}\cap U)\geq \delta\;\;\textrm{implies that}\;\;\Theta^{M,r}_{n,\theta}\notin V_{R,\delta}(x_0),\end{equation}
it follows that
\begin{equation}\label{tan_000}\begin{aligned} \mathbb{P}\left[\Theta^{M,r}_{n,\theta}\notin V_{R,\delta}(x_0),A_{n-1}\right] & =   \mathbb{P}\left[\mathbf{d}(\Theta^{M,r}_{n,\theta},\mathcal{M}\cap U)\geq \delta, A_{n-1}\right] \\ & \quad + \mathbb{P}\left[\mathbf{d}(\Theta^{M,r}_{n,\theta},\mathcal{M}\cap U)<\delta, \Theta^{M,r}_{n,\theta}\notin V_{R,\delta}(x_0), A_{n-1}\right].\end{aligned}\end{equation}
The two terms on the righthand side of \eqref{tan_000} will be handled separately.  We will first prove that there exists $c\in(0,\infty)$ which satisfies that
\begin{equation}\label{tan_00}\mathbb{P}\left[\mathbf{d}(\Theta^{M,r}_{n,\theta},\mathcal{M}\cap U)\geq\delta, A_{n-1}\right]\leq \frac{c}{Mn^{2\rho}}\mathbb{P}\left[A_{n-1}\right]+\frac{c}{Mn^\rho}.\end{equation}
On the event $A_{n-1}$, it follows from Lemma~\ref{lem_main} that there exists $\varepsilon_n\colon A_{n-1}\rightarrow\mathbb{R}^d$, $c_1\in(0,\infty)$ such that
\begin{equation}\label{tan_0}\abs{\varepsilon_n}\leq c_1\mathbf{d}(\Theta^{M,r}_{n-1,\theta},\mathcal{M}\cap U)^2,\end{equation}
and such that on the event $A_{n-1}$ it holds that
\begin{equation}\nabla f(\Theta^{M,r}_{n-1,\theta})=\big(\Hess f\big)(p(\Theta^{M,r}_{n-1,\theta}))\cdot(\Theta^{M,r}_{n-1,\theta}-p(\Theta^{M,r}_{n-1,\theta}))+\varepsilon_n.\end{equation}
Therefore, on the event $A_{n-1}$, we have that
\begin{equation}\begin{aligned} \Theta^{M,r}_{n,\theta} & =\Theta^{M,r}_{n-1,\theta}-\frac{ r}{n^\rho}\big(\Hess f\big)(p(\Theta^{M,r}_{n-1,\theta}))\cdot(\Theta^{M,r}_{n-1,\theta}-p(\Theta^{M,r}_{n-1,\theta}))-\frac{ r}{n^\rho}\varepsilon_n \\ & \quad +\frac{ r}{n^\rho}\left(\nabla f(\Theta^{M,r}_{n-1,\theta})-\nabla F^{M,n}(\Theta^{M,r}_{n-1,\theta})\right).\end{aligned}\end{equation}
Lemma~\ref{lem_projection}, \eqref{tan_0}, the choice of $ \mathfrak{r}\in(0,\infty)$, the definition of the projection, and the triangle inequality prove that there exist $c_1,\lambda\in(0,\infty)$ such that on the event $A_{n-1}$ it holds that
\begin{equation}\label{tan_1}\begin{aligned} & \mathbf{d}(\Theta^{M,r}_{n,\theta},\mathcal{M}\cap U) \\ & \leq  \abs{\Theta^{M,r}_{n-1,\theta}-p(\Theta^{M,r}_{n-1,\theta})-\frac{ r}{n^\rho}\big(\Hess f\big)(p(\Theta^{M,r}_{n-1,\theta}))\cdot(\Theta^{M,r}_{n-1,\theta}-p(\Theta^{M,r}_{n-1,\theta}))} \\ & \quad +\abs{\frac{ r}{n^\rho}\varepsilon_n}+\abs{\frac{ r}{n^\rho}\left(\nabla f(\Theta^{M,r}_{n-1,\theta})-\nabla F^{M,n}(\Theta^{M,r}_{n-1,\theta})\right)} \\ & \leq  \left(1-\frac{ r\lambda}{n^\rho}\right)\mathbf{d}(\Theta^{M,r}_{n-1,\theta},\mathcal{M}\cap U)+c_1\frac{ r}{n^\rho}\mathbf{d}(\Theta^{M,r}_{n-1,\theta},\mathcal{M}\cap U)^2+\frac{ r}{n^\rho}\abs{\nabla f(\Theta^{M,r}_{n-1,\theta})-\nabla F^{M,n}(\Theta^{M,r}_{n-1,\theta})}.\end{aligned}\end{equation}
Fix $\delta_1\in(0,\delta_0]$ which satisfies that
\begin{equation}c_1\delta_1\leq \frac{\lambda}{2}.\end{equation}
Let $\delta\in(0,\delta_1]$.  On the event $A_{n-1}$, it follows from \eqref{tan_1} and the choice of $\delta_1\in(0,\delta_0]$ that
\begin{equation}\mathbf{d}(\Theta^{M,r}_{n,\theta},\mathcal{M}\cap U)\leq \left(1-\frac{ r\lambda}{2n^\rho}\right)\mathbf{d}(\Theta^{M,r}_{n-1,\theta},\mathcal{M}\cap U)+\frac{ r}{n^\rho}\abs{\nabla f(\Theta^{M,r}_{n-1,\theta})-\nabla F^{M,n}(\Theta^{M,r}_{n-1,\theta})}.\end{equation}
We therefore conclude that
\begin{equation}\label{tan_2}\begin{aligned}  & \mathbb{P}\left[\mathbf{d}(\Theta^{M,r}_{n,\theta},\mathcal{M}\cap U)\geq\delta,A_{n-1}\right] \leq \\ &   \mathbb{P}\left[\abs{\nabla f(\Theta^{M,r}_{n-1,\theta})-\nabla F^{M,n}(\Theta^{M,r}_{n-1,\theta})}\geq\frac{\delta n^\rho}{2 r}, \Theta^{M,r}_{n-1,\theta}\in V_{R,\frac{\delta}{2}}(x_0), A_{n-2}\right] \\  & +  \mathbb{P}\left[\abs{\nabla f(\Theta^{M,r}_{n-1,\theta})-\nabla F^{M,n}(\Theta^{M,r}_{n-1,\theta})}\geq\frac{\delta\lambda}{2}, \Theta^{M,r}_{n-1,\theta}\in V_{R,\delta}(x_0)\setminus V_{R,\frac{\delta}{2}}(x_0), A_{n-2}\right].\end{aligned}\end{equation}
Similarly to \eqref{sgd_4} and computation \eqref{sgd_04}, it follows from the independence of the random variables $X_{m,k}$, $m,k\in\mathbb{N}$, that
\begin{equation}\label{mean_4}\begin{aligned} & \mathbb{P}\left[\;\abs{\nabla f(\Theta^{M,r}_{n-1,\theta})-\nabla_\theta F^{M,n}(\Theta^{M,r}_{n-1,\theta})}\geq \frac{\delta n^\rho}{2 r}, \Theta^{M,r}_{n-1,\theta}\in V_{R,\frac{\delta}{2}}(x_0), A_{n-2}\right] \\ & \leq \sup_{\theta\in V_{R,\frac{\delta}{2}}(x_0)}\mathbb{P}\left[\;\abs{\nabla f(\theta)-\nabla_\theta F^{M,n}(\theta)}\geq \frac{\delta n^\rho}{2 r}\;\right]\mathbb{P}\left[\Theta^{M,r}_{n-1,\theta}\in V_{R,\frac{\delta}{2}}(x_0), A_{n-2}\right],\end{aligned}\end{equation}
and that
\begin{equation}\label{mean_5}\begin{aligned} & \mathbb{P}\left[\;\abs{\nabla f(\Theta^{M,r}_{n-1,\theta})-\nabla_\theta F^{M,n}(\Theta^{M,r}_{n-1,\theta})}\geq \frac{\delta\lambda}{2}, \Theta^{M,r}_{n-1,\theta}\in V_{R,\delta}(x_0)\setminus V_{R,\frac{\delta}{2}}(x_0),  A_{n-2}\right] \\ & \leq\sup_{\theta\in V_{R,\delta}(x_0)\setminus V_{R,\frac{\delta}{2}}(x_0)}\mathbb{P}\left[\;\abs{\nabla f(\theta)-\nabla_\theta F^{M,n}(\theta)}\geq \frac{\delta\lambda}{2}\;\right]\mathbb{P}\left[\Theta^{M,r}_{n-1,\theta}\in V_{R,\delta}(x_0)\setminus V_{R,\frac{\delta}{2}}(x_0), A_{n-2}\right].\end{aligned}\end{equation}
The definition of $A_{n-1}$, Chebyshev's inequality, Lemma~\ref{variance_lem}, and \eqref{mean_4} prove that there exists $c\in(0,\infty)$ which satisfies that
\begin{equation}\label{mean_04}\begin{aligned}  \mathbb{P}\left[\;\abs{\nabla f(\Theta^{M,r}_{n-1,\theta})-\nabla_\theta F^{M,n}(\Theta^{M,r}_{n-1,\theta})}\geq \frac{\delta n^\rho}{2 r}, \Theta^{M,r}_{n-1,\theta}\in V_{R,\frac{\delta}{2}}(x_0), A_{n-2}\right]\leq &  \frac{c}{M}\cdot \frac{4 r^2}{\delta^2n^{2\rho}}\mathbb{P}\left[A_{n-1}\right] \\ \leq & \frac{c}{Mn^{2\rho}}\mathbb{P}\left[A_{n-1}\right].\end{aligned}\end{equation}
In the case of \eqref{mean_5}, Proposition~\ref{ng_converge} and Chebyshev's inequality prove that, for the indicator function $\mathbf{1}_{A_{n-2}}$ of the event $A_{n-2}$, there exists $c\in(0,\infty)$ which satisfies that
\begin{equation}\begin{aligned}\mathbb{P}\left[\Theta^{M,r}_{n-1,\theta}\in V_{R,\delta}(x_0)\setminus V_{R,\frac{\delta}{2}}(x_0), A_{n-2}\right] & \leq  \mathbb{P}\left[\left(\mathbf{d}\left(\Theta^{M,r}_{n-1,\theta},\mathcal{M}\cap U\right)\wedge 1\right)^2\mathbf{1}_{A_{n-2}}\geq \frac{\delta^2}{4}\right] \\ & \leq \frac{4c}{\delta^2}n^{-\rho} \\ & \leq  cn^{-\rho},\end{aligned}\end{equation}
where we have used the fact that, since $\rho\in(\nicefrac{2}{3},1)$, there exists $c\in(0,\infty)$ such that for every $n\in\mathbb{N}$ it holds that $(n-1)^{-\rho}\leq c n^{-\rho}$.  Furthermore, Chebyshev's inequality and Lemma~\ref{variance_lem} prove that there exists $c\in(0,\infty)$ which satisfies that
\begin{equation}\mathbb{P}\left[\;\abs{\nabla f(\Theta^{M,r}_{n-1,\theta})-\nabla_\theta F^{M,n}(\Theta^{M,r}_{n-1,\theta})}\geq \frac{\delta\lambda}{2}\;\right]\leq \frac{c}{M}\cdot\frac{4}{\delta^2\lambda^2}\leq \frac{c}{M}.\end{equation}
Returning to \eqref{mean_5}, the previous two inequalities prove that there exists $c\in(0,\infty)$ which satisfies that
\begin{equation}\label{mean_004} \mathbb{P}\left[\;\abs{\nabla f(\Theta^{M,r}_{n-1,\theta})-\nabla_\theta F^{M,n}(\Theta^{M,r}_{n-1,\theta})}\geq \frac{\delta\lambda}{2}, \Theta^{M,r}_{n-1,\theta}\in V_\delta\setminus V_{\frac{\delta}{2}},  A_{n-2}\right]\leq \frac{c}{Mn^\rho}.\end{equation}
Combining \eqref{tan_2}, \eqref{mean_04}, and \eqref{mean_004}, there exists $c\in(0,\infty)$ such that
\begin{equation}\label{mean_0004}\mathbb{P}\left[\mathbf{d}(\Theta^{M,r}_{n,\theta},\mathcal{M}\cap U)\geq \delta,A_{n-1}\right]\leq \frac{c}{Mn^{2\rho}}\mathbb{P}\left[A_{n-1}\right]+\frac{c}{Mn^\rho},\end{equation}
which completes the proof of \eqref{tan_00}.  Returning to \eqref{tan_000}, it follows from \eqref{mean_0004} that there exists $c\in(0,\infty)$ such that
\begin{equation}\begin{aligned} & \mathbb{P}\left[\Theta^{M,r}_{n,\theta}\notin V_{R,\delta}(x_0), A_{n-1}\right] \\ &  \leq \frac{c}{Mn^{2\rho}}\mathbb{P}\left[A_{n-1}\right]+\frac{c}{Mn^\rho}+\mathbb{P}\left[\mathbf{d}(\Theta^{M,r}_{n,\theta},\mathcal{M}\cap U)<\delta, \Theta^{M,r}_{n,\theta}\notin V_{R,\delta}(x_0), A_{n-1}\right].\end{aligned}\end{equation}
Therefore, there exists $c\in(0,\infty)$ which satisfies that
\begin{equation}\begin{aligned}\label{tan_6}  \mathbb{P}[A_n]  & =   \mathbb{P}\left[\Theta^{M,r}_{n,\theta}\in V_{R,\delta}(x_0), A_{n-1}\right] \\   & \geq  \left(1-\frac{c}{Mn^{2\rho}}\right)_+\mathbb{P}\left[A_{n-1}\right]-\frac{c}{Mn^\rho} -\mathbb{P}\left[\mathbf{d}(\Theta^{M,r}_{n,\theta},\mathcal{M}\cap U)<\delta, \Theta^{M,r}_{n,\theta}\notin V_{R,\delta}(x_0), A_{n-1}\right].\end{aligned}\end{equation}
We will prove inductively that \eqref{tan_6} implies that there exists $c\in(0,\infty)$ such that for every $n\in\mathbb{N}$ it holds that
\begin{equation}\mathbb{P}[A_n]\geq \prod_{k=1}^n\left(1-\frac{c}{Mk^{2\rho}}\right)_+-\sum_{k=1}^n\frac{c}{Mk^\rho}-\sum_{k=1}^n \mathbb{P}\left[\mathbf{d}(\Theta^{M,r}_{k,\theta},\mathcal{M}\cap U)<\delta, \Theta^{M,r}_{k,\theta}\notin V_{R,\delta}(x_0), A_{k-1}\right].\end{equation}
The base case $n=0$ follows immediately from $\theta\in V_{\nicefrac{R}{2},\delta}(x_0)$.  For the inductive step, suppose that \eqref{tan_7} is satisfied for some $n\in\mathbb{N}$.  It follows from \eqref{tan_6} that
\begin{equation}\begin{aligned} \mathbb{P}\left[A_{n+1}\right] & \geq   \left(1-\frac{c}{M(n+1)^{2\rho}}\right)_+\mathbb{P}\left[A_n\right]-\frac{c}{M(n+1)^\rho} \\ & \quad -\mathbb{P}\left[\mathbf{d}(\Theta^{M,r}_{n+1,\theta},\mathcal{M}\cap U)<\delta, \Theta^{M,r}_{n+1,\theta}\notin V_{R,\delta}(x_0), A_n\right].\end{aligned}\end{equation}
It then follows from the inductive hypothesis \eqref{tan_7} that
\begin{equation}\begin{aligned} & \mathbb{P}\left[A_{n+1}\right] \\ & \geq   \prod_{k=1}^{n+1}\left(1-\frac{c}{Mk^{2\rho}}\right)_+\mathbb{P}\left[A_0\right]-\frac{c}{M(n+1)^\rho} \\ & -\mathbb{P}\left[\mathbf{d}(\Theta^{M,r}_{n+1,\theta},\mathcal{M}\cap U)<\delta, \Theta^{M,r}_{n+1,\theta}\notin V_{R,\delta}(x_0), A_n\right] \\ & -\left(1-\frac{c}{M(n+1)^{2\rho}}\right)_+\left(\sum_{k=1}^n\frac{c}{Mk^\rho}+\sum_{k=1}^n\mathbb{P}\left[\mathbf{d}(\Theta^{M,r}_{k,\theta},\mathcal{M}\cap U)<\delta, \Theta^{M,r}_{k,\theta}\notin V_{R,\delta}(x_0), A_{k-1}\right]\right),\end{aligned}\end{equation}
which proves that
\begin{equation}\begin{aligned} & \mathbb{P}\left[A_{n+1}\right] \\ & \geq \prod_{k=1}^{n+1}\left(1-\frac{c}{Mk^{2\rho}}\right)_+\mathbb{P}\left[A_0\right]-\sum_{k=1}^{n+1}\frac{c}{Mk^\rho}-\sum_{k=1}^{n+1}\mathbb{P}\left[\mathbf{d}(\Theta^{M,r}_{k,\theta},\mathcal{M}\cap U)<\delta, \Theta^{M,r}_{k,\theta}\notin V_{R,\delta}(x_0), A_{k-1}\right].\end{aligned}\end{equation}
Finally, since $\theta\in V_{\nicefrac{R}{2},\delta}(x_0)\subseteq V_{R,\delta}(x_0)$ implies that $\mathbb{P}(A_0)=1$, it holds that
\begin{equation}\label{tan_7}\mathbb{P}\left[A_{n+1}\right]\geq \prod_{k=1}^{n+1}\left(1-\frac{c}{Mk^{2\rho}}\right)_+-\sum_{k=1}^{n+1}\frac{c}{Mk^\rho}-\sum_{k=1}^{n+1}\mathbb{P}\left[\mathbf{d}(\Theta^{M,r}_{k,\theta},\mathcal{M}\cap U)<\delta, \Theta^{M,r}_{k,\theta}\notin V_{R,\delta}(x_0), A_{k-1}\right],\end{equation}
which completes the induction step, and the proof of \eqref{tan_7}.  It remains only to estimate the final term on the righthand side of inequality \eqref{tan_7}.  The definition of the events $A_m$, $m\in\mathbb{N}_0$, implies that
\begin{equation}\{\mathbf{d}(\Theta^{M,r}_{k,\theta},\mathcal{M}\cap U)<\delta, \Theta^{M,r}_{k,\theta}\notin V_{R,\delta}(x_0), A_{k-1}\}\subseteq\Omega, \;k\in\mathbb{N},\;\textrm{are disjoint events.}\end{equation}
Therefore, it holds that
\begin{equation}\label{tan_9}\begin{aligned} & \sum_{k=1}^n \mathbb{P}\left[\mathbf{d}(\Theta^{M,r}_{k,\theta},\mathcal{M}\cap U)<\delta, \Theta^{M,r}_{k,\theta}\notin V_{R,\delta}(x_0), A_{k-1}\right] \\ & =\mathbb{P}\left[\coprod_{k=1}^n\{\mathbf{d}(\Theta^{M,r}_{k,\theta},\mathcal{M}\cap U)<\delta, \Theta^{M,r}_{k,\theta}\notin V_{R,\delta}(x_0), A_{k-1}\}\right].\end{aligned}\end{equation}
Lemma~\ref{lem_tub} proves that
\begin{equation}\label{tan_10}\mathbb{P}\left[\coprod_{k=1}^n\{\mathbf{d}(\Theta^{M,r}_{k,\theta},\mathcal{M}\cap U)<\delta, \Theta^{M,r}_{k,\theta}\notin V_{R,\delta}(x_0), A_{k-1}\}\right]\leq \mathbb{P}\left[\max_{1\leq k\leq n}\abs{\Theta^{M,r}_{k,\theta}-x_0}\mathbf{1}_{A_{k-1}}>R-\delta\right].\end{equation}
Since $\Theta^{M,k}_{0,\theta}\in V_{\nicefrac{R}{2},\delta}(x_0)$, the triangle inequality prove for every  $k\in\{1,2,\ldots,n\}$ that
\begin{equation}\begin{aligned}\abs{\Theta^{M,r}_{k,\theta}-x_0} & \leq  \abs{\Theta^{M,r}_{k,\theta}-\Theta^{M,k}_{0,\theta}}+\abs{\Theta^{M,k}_{0,\theta}-p(\Theta^{M,k}_{0,\theta})}+\abs{p(\Theta^{M,r}_{0,\theta})-x_0} \\ & \leq \abs{\Theta^{M,r}_{k,\theta}-\theta}+\delta+\frac{R}{2}.\end{aligned}\end{equation}
Therefore, for every $k\in\{1,\ldots,n\}$, on the event $\big\{\big|\Theta^{M,r}_{k,\theta}-x_0\big|>R-\delta\big\}$ it holds that
\begin{equation}\frac{R}{2}-2\delta< \big|\Theta^{M,r}_{k,\theta}-\Theta^{M,r}_{0,\theta}\big|.\end{equation}
This implies that
\begin{equation}\label{tan_11} \left\{\max_{1\leq k\leq n}\abs{\Theta^{M,r}_{k,\theta}-x_0}\mathbf{1}_{A_{k-1}}>R-\delta\right\}\subseteq \left\{\max_{1\leq k\leq n}\abs{\Theta^{M,r}_{k,\theta}-\Theta^{M,r}_{0,\theta}}\mathbf{1}_{A_{k-1}}>\frac{R}{2}-2\delta\right\}.\end{equation}
In combination, \eqref{tan_9}, \eqref{tan_10}, and \eqref{tan_11} prove that
\begin{equation}\label{tan_12} \sum_{k=1}^n \mathbb{P}\left[\mathbf{d}(\Theta^{M,r}_{k,\theta},\mathcal{M}\cap U)<\delta, \Theta^{M,r}_{k,\theta}\notin V_{R,\delta}(x_0), A_{k-1}\right]\leq \mathbb{P}\left[\max_{1\leq k\leq n}\abs{\Theta^{M,r}_{k,\theta}-\Theta^{M,r}_{0,\theta}}\mathbf{1}_{A_{k-1}}>\frac{R}{2}-2\delta\right].\end{equation}
It follows from Proposition~\ref{ng_tang}, \eqref{tan_12}, and Chebyshev's inequality that there exists $c\in(0,\infty)$ which satisfies that
\begin{equation}\label{tan_14}\sum_{k=1}^n \mathbb{P}\left[\mathbf{d}(\Theta^{M,r}_{k,\theta},\mathcal{M}\cap U)<\delta, \Theta^{M,r}_{k,\theta}\notin V_{R,\delta}(x_0), A_{k-1}\right]\leq \frac{cr\left(1+M^{-\frac{1}{2}}n^{1-\rho}\right)}{\left(\frac{R}{2}-2\delta\right)_+}. \end{equation}
Returning to \eqref{tan_7}, it follows from \eqref{tan_14} that there exists $c\in(0,\infty)$ which satisfies that
\begin{equation}\mathbb{P}[A_n]\geq \prod_{k=1}^n\left(1-\frac{c}{Mk^{2\rho}}\right)_+-cM^{-1}n^{1-\rho}-\frac{cr\left(1+M^{-\frac{1}{2}}n^{1-\rho}\right)}{\left(\frac{R}{2}-2\delta\right)_+},\end{equation}
where we have used the fact that, since $\rho\in(\nicefrac{2}{3},1)$, there exists $c\in(0,\infty)$ which satisfies that
\begin{equation}\sum_{k=1}^nk^{-\rho}\leq cn^{1-\rho}.\end{equation}
This completes the proof of Proposition~\ref{ng_good_set}. \end{proof}

We will now use Proposition~\ref{ng_converge} and Proposition~\ref{ng_good_set} to estimate the probability that SGD of mini-batch size $M\in\mathbb{N}$ converges to within distance $\varepsilon\in(0,1]$ of the manifold of local minima at time $n\in\mathbb{N}$.  In the theorem, we assume that the initial condition $\Theta^{M,r}_0$ is continuous uniformly distributed on a bounded open subset $A\subseteq\mathbb{R}^d$ which satisfies that $\mathcal{M}\cap U\cap A\neq\emptyset$.

\begin{thm}\label{ng_one_path} Let $d\in\mathbb{N}$, $\mathfrak{d}\in\{ 0, 1, \ldots, d - 1 \}$, $\rho\in(\nicefrac{2}{3},1)$, let $\abs{\cdot}\colon\mathbb{R}^d\rightarrow\mathbb{R}$ be the standard norm on $\mathbb{R}^d$, let $U\subseteq\mathbb{R}^d$ be an open set, let $A\subseteq\mathbb{R}^d$ be a bounded open set, let $\lambda\colon\mathcal{B}(\mathbb{R}^d)\rightarrow[0,\infty]$ be the Lebesgue-Borel measure, let $(\Omega,\mathcal{F},\mathbb{P})$ be a probability space, let $(S,\mathcal{S})$ be a measurable space, let $F=(F(\theta,x))_{(\theta,x)\in\mathbb{R}^d\times S}\colon  \mathbb{R}^d\times S\rightarrow\mathbb{R}$ be a measurable function, let $X_{n,m}\colon  \Omega\rightarrow S$, $n,m\in\mathbb{N}$, be i.i.d.\ random variables which satisfy for every $\theta\in\mathbb{R}^d$ that $\mathbb{E}\big[ |F(\theta,X_{1,1})|^2\big]<\infty$, let $f\colon\mathbb{R}^d\rightarrow\mathbb{R}$ be the function which satisfies for every $\theta\in\mathbb{R}^d$ that $f(\theta)=\mathbb{E}\big[F(\theta,X_{1,1})\big]$, let $\mathcal{M}\subseteq\mathbb{R}^d$ satisfy that

\begin{equation}\mathcal{M}=\big\{\theta\in\mathbb{R}^d\colon [f(\theta)=\inf\nolimits_{\vartheta\in\mathbb{R}^d} f(\vartheta)]\big\},\end{equation}
let $\mathbf{d}(\cdot,\mathcal{M}\cap U):\mathbb{R}^d\rightarrow\mathbb{R}$ be the function which satisfies for every $x\in\mathbb{R}^d$ that
\begin{equation} \mathbf{d}(x,\mathcal{M}\cap U)=\inf \left\{\abs{x-y}\colon y\in(\mathcal{M}\cap U)\right\},\end{equation}
let $(\cdot)_+\colon\mathbb{R}\rightarrow\mathbb{R}$ be the function which satisfies for every $x\in\mathbb{R}$ that
\begin{equation} (x)_+=\max(x,0),\end{equation}
assume for every $x\in S$ that $\mathbb{R}^d\ni\theta\mapsto F(\theta,x)\in\mathbb{R}$ is a locally Lipschitz continuous function, assume that $f|_U\colon U\rightarrow\mathbb{R}$ is a three times continuously differentiable function, assume for every non-empty compact set $\mathfrak{C}\subseteq U$ that $\sup\nolimits_{\theta\in \mathfrak{C}}\mathbb{E}\big[|F(\theta,X_{1,1})|^2+|(\nabla_\theta F)(\theta,X_{1,1})|^2\big]<\infty$, assume that $\mathcal{M}\cap U$ is a $\mathfrak{d}$-dimensional $\C^1$-submanifold of $\mathbb{R}^d$, assume that $\mathcal{M}\cap U\cap A\neq\emptyset$, assume for every $\theta\in(\mathcal{M}\cap U)$ that $\rank((\Hess f)(\theta))=d-\mathfrak{d}$, for every $M\in\mathbb{N}$, $r\in(0,\infty)$ let $\Theta^{M,r}_0\colon\Omega\rightarrow\mathbb{R}^d$ be continuous uniformly distributed on $A$, assume for every $M\in\mathbb{N}$, $r\in(0,\infty)$ that $\Theta^{M,r}_0$ and $\big(X_{n,m}\big)_{n,m\in\mathbb{N}}$ are independent, for every $M\in\mathbb{N}$, $r\in(0,\infty)$ let $\Theta^{M,r}_n\colon\Omega\rightarrow\mathbb{R}^d$, $n\in\mathbb{N}$, be random variables which satisfy that
\begin{equation}\Theta^{M,r}_n=\Theta^{M,r}_{n-}-\frac{r}{n^\rho M}\!\left[\sum_{m=1}^M(\nabla_\theta F)(\Theta^{M,r}_{n-1},X_{n,m})\right].\end{equation}
Then for every $x_0\in(\mathcal{M}\cap U\cap A)$  there exist $R_0,\delta_0,\mathfrak{r},c\in(0,\infty)$ such that for every $R\in(0,R_0]$, $\delta\in(0,\delta_0]$, $r\in(0,\mathfrak{r}]$, $n,M\in\mathbb{N}$, $\varepsilon\in(0,1]$ it holds that
\begin{equation}\label{gss_0}\begin{aligned} &  \mathbb{P}\Big( \mathbf{d}\big(\Theta^{M,r}_n,\mathcal{M}\cap U\big)\geq\varepsilon\Big)\leq \\ &  \frac{\lambda\big(A\backslash V_{\nicefrac{R}{2},\delta}(x_0)\big)}{\lambda(A)} +c\varepsilon^{-2}n^{-\rho}+1- \prod_{k=1}^n\left(1-\frac{c}{Mk^{2\rho}}\right)_++cM^{-1}n^{1-\rho}+\frac{cr\left(1+M^{-\frac{1}{2}}n^{1-\rho}\right)}{\left(\frac{R}{2}-2\delta\right)_+}.\end{aligned} \end{equation}
\end{thm}

\begin{proof}[Proof of Theorem~\ref{ng_one_path}]  Let $x_0\in(\mathcal{M}\cap U\cap A)$.  Since $U\subseteq\mathbb{R}^d$ is open, fix $V\in\Proj(x_0)$ (cf. Definition~\ref{define_proj}) which satisfies that $V\subseteq U$.  Fix $R_0,\delta_0\in(0,\infty)$ that satisfy the conclusion of Proposition~\ref{cts_prop_tub} for this set $V$.  Fix $ \mathfrak{r}\in(0,\infty)$ that satisfies the conclusions of Lemma~\ref{lem_projection} and Proposition~\ref{ng_good_set}.  Let $R\in(0,R_0]$, $\delta\in(0,\delta_0]$, $r\in(0,\mathfrak{r}]$, $M\in\mathbb{N}$.  As in Proposition~\ref{ng_converge}, let $\nabla_\theta F^{M,n}\colon\mathbb{R}^d\times\Omega\rightarrow\mathbb{R}^d$, $n\in\mathbb{N}$, be the functions which satisfy for every $(\theta,\omega)\in\mathbb{R}^d\times\Omega$ that
\begin{equation}\nabla_\theta F^{M,n}(\theta)=\nabla_\theta F^{M,n}(\theta,\omega)=\frac{1}{M}\sum_{m=1}^M(\nabla_\theta F)(\theta,X_{n,m}(\omega)).\end{equation}
For every $\theta\in\mathbb{R}^d$ let $\Theta^{M,r}_{0,\theta}\colon\Omega\rightarrow\mathbb{R}^d$ satisfy for every $\omega\in\Omega$ that $\Theta^{M,r}_{0,\theta}(\omega)=\theta$ and for every $n\in\mathbb{N}$ let $\Theta^{M,r}_{n,\theta}\colon\Omega\rightarrow\mathbb{R}^d$ satisfy that
\begin{equation}\Theta^{M,r}_{n,\theta}=\Theta^{M,r}_{n-1,\theta}-\frac{r}{n^\rho}\nabla_\theta F^{M,n}(\Theta^{M,r}_{n-1,\theta}).\end{equation}
Let $\Theta^{M,r}_0\colon\Omega\rightarrow\mathbb{R}^d$ be a random variable which is continuous uniformly distributed on $A$, assume that $\Theta^{M,r}_0$ and $(X_{n,m})_{n,m\in\mathbb{N}}$ are independent, and for every $n\in\mathbb{N}$ let $\Theta^{M,r}_n\colon\Omega\rightarrow\mathbb{R}^d$ satisfy that $\Theta^{M,r}_n=\Theta^{M,r}_{n,\Theta^{M,r}_0}$.  
Let $n\in\mathbb{N}$, $\varepsilon\in(0,1]$.  It holds that
\begin{equation}\label{gss_1}\begin{aligned} \mathbb{P}\Big(  \mathbf{d}\big(\Theta^{M,r}_n,\mathcal{M}\cap U\big)\geq \varepsilon\Big) & =   \mathbb{P}\Big(  \mathbf{d}\big(\Theta^{M,r}_n,\mathcal{M}\cap U\big)\geq\varepsilon, \Theta^{M,r}_0\in V_{\nicefrac{R}{2},\delta}(x_0)\Big) \\ &\quad + \mathbb{P}\Big(  \mathbf{d}\big(\Theta^{M,r}_n,\mathcal{M}\cap U\big)\geq\varepsilon, \Theta^{M,r}_0\notin V_{\nicefrac{R}{2},\delta}(x_0)\Big).\end{aligned}\end{equation}
For the second term on the righthand side of \eqref{gss_0}, it follows from the continuous uniform distribution of $\Theta^{M,r}_0$ on $A$ that
\begin{equation}\label{gss_2} \mathbb{P}\Big(   \mathbf{d}\big(\Theta^{M,r}_n,\mathcal{M}\cap U\big) \geq\varepsilon, \Theta^{M,r}_0\notin V_{\nicefrac{R}{2},\delta}(x_0)\Big)\leq \frac{\lambda\big(A\backslash V_{\nicefrac{R}{2},\delta}(x_0)\big)}{\lambda(A)}.\end{equation}
We will now estimate the first term on the righthand side of \eqref{gss_1}.  For every $m\in\mathbb{N}_0$, $\theta\in\mathbb{R}^d$ let $A_{m,\theta}\subseteq\Omega$ be the event which satisfies that that
\begin{equation}A_{m,\theta}=\Big\{\forall\; k\in\{0,\ldots,m\}\;\Theta^{M,r}_{k,\theta}\in V_{R,\delta}(x_0)\; \Big\},\end{equation}
and for every $m\in\mathbb{N}_0$ let $A_m\in\mathcal{F}$ satisfy that
\begin{equation}A_m=\Big\{\forall\; k\in\{0,\ldots,m\}\;\Theta^{M,r}_k\in V_{R,\delta}(x_0)\Big\}.\end{equation}
It holds that
\begin{equation}\label{gss_3}\begin{aligned} & \mathbb{P}\Big(  \mathbf{d}\big(\Theta^{M,r}_n,\mathcal{M}\cap U\big)\geq\varepsilon, \Theta^{M,r}_0\in V_{\nicefrac{R}{2},\delta}(x_0) \Big) \\ & =  \mathbb{P}\Big(  \mathbf{d}\big(\Theta^{M,r}_n,\mathcal{M}\cap U\big) \geq\varepsilon, \Theta^{M,r}_0\in V_{\nicefrac{R}{2},\delta}(x_0), A_{n-1}\Big) \\ & \quad + \mathbb{P}\Big(  \mathbf{d}\big(\Theta^{M,r}_n,\mathcal{M}\cap U\big) \geq\varepsilon, \Theta^{M,r}_0\in V_{\nicefrac{R}{2},\delta}(x_0), \Omega\backslash A_{n-1}\Big).\end{aligned}\end{equation}
For the second term on the righthand side of \eqref{gss_3}, it follows from Proposition~\ref{ng_good_set} that there exists $c\in(0,\infty)$ such that
\begin{equation}\label{gss_4} \begin{aligned} &  \mathbb{P}\Big(  \mathbf{d}\big(\Theta^{M,r}_n,\mathcal{M}\cap U\big) \geq\varepsilon, \Theta^{M,r}_0\notin V_{\nicefrac{R}{2},\delta}(x_0), \Omega\backslash A_{n-1}\Big) \\ & \leq 1- \prod_{k=1}^n\left(1-\frac{c}{Mk^{2\rho}}\right)_++cM^{-1}n^{1-\rho}+\frac{cr\left(1+M^{-\frac{1}{2}}n^{1-\rho}\right)}{\left(\frac{R}{2}-2\delta\right)_+},\end{aligned}\end{equation}
where we have used the fact that $\rho\in(\nicefrac{2}{3},1)$ implies that there exists $c\in(0,\infty)$ that satisfies for every $n\in\{2,3,\ldots\}$ that $n^{1-\rho}\leq c(n-1)^{1-\rho}$.  For the first term on the righthand side of \eqref{gss_3}, since the random variables $\Theta^{M,r}_0$ and $\big(X_{n,m}\big)_{n,m\in\mathbb{N}}$ are independent, it holds that
\begin{equation}\begin{aligned}\label{gss_5} & \mathbb{P}\Big(  \mathbf{d}\big(\Theta^{M,r}_n,\mathcal{M}\cap U\big) \geq\varepsilon, \Theta^{M,r}_0\in V_{\nicefrac{R}{2},\delta}(x_0), A_{n-1}\Big) \\ & \leq \frac{\lambda\big(V_{\nicefrac{R}{2},\delta}(x_0)\cap A\big)}{\lambda(A)}\sup_{\theta\in V_{\nicefrac{R}{2},\delta}(x_0)}\mathbb{P}\Big( \mathbf{d}\big(\Theta^{M,r}_{n,\theta},\mathcal{M}\cap U\big) \geq\varepsilon, A_{n-1,\theta}\Big).\end{aligned}\end{equation}
Proposition~\ref{ng_converge} and Chebyshev's inequality prove that there exists $c\in(0,\infty)$ such that for every $\theta\in V_{\nicefrac{R}{2},\delta}(x_0)$ it holds that
\begin{equation}\label{gss_6}\mathbb{P}\Big( \mathbf{d}\big(\Theta^{M,r}_{n,\theta},\mathcal{M}\cap U\big) \geq\varepsilon, A_{n-1,\theta}\Big)\leq  \varepsilon^{-2}\mathbb{E}\Big[\Big(\mathbf{d}\big(\Theta^{M,r}_{n,\theta},\mathcal{M}\cap U\big)\wedge 1\Big)^2\mathbf{1}_{A_{n-1,\theta}}\Big]\leq c \varepsilon^{-2}n^{-\rho}.\end{equation}
In combination \eqref{gss_5} and \eqref{gss_6} prove that there exists $c\in(0,\infty)$ such that
\begin{equation}\label{gss_7} \mathbb{P}\Big(  \mathbf{d}\big(\Theta^{M,r}_n,\mathcal{M}\cap U\big) \geq\varepsilon, \Theta^{M,r}_0\in V_{\nicefrac{R}{2},\delta}(x_0), A_{n-1}\Big)\leq c\varepsilon^{-2}n^{-\rho}.\end{equation}
Returning to \eqref{gss_3}, it follows from \eqref{gss_4} and \eqref{gss_7} that there exists $c\in(0,\infty)$ such that
\begin{equation}\label{gss_8}\begin{aligned} & \mathbb{P}\Big(  \mathbf{d}\big(\Theta^{M,r}_n,\mathcal{M}\cap U\big)\geq\varepsilon, \Theta^{M,r}_0\in V_{\nicefrac{R}{2},\delta}(x_0) \Big) \\ & \leq c\varepsilon^{-2}n^{-\rho}+1- \prod_{k=1}^n\left(1-\frac{c}{Mk^{2\rho}}\right)_++cM^{-1}n^{1-\rho}+\frac{cr\left(1+M^{-\frac{1}{2}}n^{1-\rho}\right)}{\left(\frac{R}{2}-2\delta\right)_+}.\end{aligned}\end{equation}
Returning finally to \eqref{gss_1}, it follows from \eqref{gss_2} and \eqref{gss_8} that there exists $c\in(0,\infty)$ such that
\begin{equation}\label{gss_9}\begin{aligned} &  \mathbb{P}\Big(  \mathbf{d}\big(\Theta^{M,r}_n,\mathcal{M}\cap U\big)\geq\varepsilon\Big) \leq \\ &  \frac{\lambda\big(A\backslash V_{\nicefrac{R}{2},\delta}(x_0)\big)}{\lambda(A)} +c\varepsilon^{-2}n^{-\rho}+1- \prod_{k=1}^n\left(1-\frac{c}{Mk^{2\rho}}\right)_++cM^{-1}n^{1-\rho}+\frac{cr\left(1+M^{-\frac{1}{2}}n^{1-\rho}\right)}{\left(\frac{R}{2}-2\delta\right)_+},\end{aligned} \end{equation}
which completes the proof of Theorem~\ref{ng_one_path}.  \end{proof}

The next corollary estimates the probability that $K\in\mathbb{N}$ independent samples of SGD with mini-batch size $M\in\mathbb{N}$ fail to to converge to within distance $\varepsilon\in(0,1]$ of the manifold of local minima $\mathcal{M}\cap U$ at time $n\in\mathbb{N}$.  The proof is a straightforward consequence of Theorem~\ref{ng_one_path} and the independence of the random variables.

 \begin{cor}\label{cor_ng_one_path}  Let $d\in\mathbb{N}$, $\mathfrak{d}\in\{ 0, 1, \ldots, d - 1 \}$, $\rho\in(\nicefrac{2}{3},1)$, let $\abs{\cdot}\colon\mathbb{R}^d\rightarrow\mathbb{R}$ be the standard norm on $\mathbb{R}^d$, let $U\subseteq\mathbb{R}^d$ be an open set, let $A\subseteq\mathbb{R}^d$ be a bounded open set, let $(\Omega,\mathcal{F},\mathbb{P})$ be a probability space, let $(S,\mathcal{S})$ be a measurable space, let $F=(F(\theta,x))_{(\theta,x)\in\mathbb{R}^d\times S}\colon  \mathbb{R}^d\times S\rightarrow\mathbb{R}$ be a measurable function, let $X_{n,m}\colon  \Omega\rightarrow S$, $n,m\in\mathbb{N}$, be i.i.d.\ random variables which satisfy for every $\theta\in\mathbb{R}^d$ that $\mathbb{E}\big[ |F(\theta,X_{1,1})|^2\big]<\infty$, let $f\colon\mathbb{R}^d\rightarrow\mathbb{R}$ be the function which satisfies for every $\theta\in\mathbb{R}^d$ that $f(\theta)=\mathbb{E}\big[F(\theta,X_{1,1})\big]$, let $\mathcal{M}\subseteq\mathbb{R}^d$ satisfy that
\begin{equation}\mathcal{M}=\big\{\theta\in\mathbb{R}^d\colon [f(\theta)=\inf\nolimits_{\vartheta\in\mathbb{R}^d} f(\vartheta)]\big\},\end{equation}
let $\mathbf{d}(\cdot,\mathcal{M}\cap U):\mathbb{R}^d\rightarrow\mathbb{R}$ be the function which satisfies for every $x\in\mathbb{R}^d$ that
\begin{equation} \mathbf{d}(x,\mathcal{M}\cap U)=\inf \left\{\abs{x-y}\colon y\in(\mathcal{M}\cap U)\right\},\end{equation}
let $(\cdot)_+\colon\mathbb{R}\rightarrow\mathbb{R}$ be the function which satisfies for every $x\in\mathbb{R}$ that
\begin{equation} (x)_+=\max(x,0),\end{equation}
assume for every $x\in S$ that $\mathbb{R}^d\ni\theta\mapsto F(\theta,x)\in\mathbb{R}$ is a locally Lipschitz continuous function, assume that $f|_U\colon U\rightarrow\mathbb{R}$ is a three times continuously differentiable function, assume for every non-empty compact set $\mathfrak{C}\subseteq U$ that $\sup\nolimits_{\theta\in \mathfrak{C}}\mathbb{E}\big[|F(\theta,X_{1,1})|^2+|(\nabla_\theta F)(\theta,X_{1,1})|^2\big]<\infty$, assume that $\mathcal{M}\cap U$ is a $\mathfrak{d}$-dimensional $\C^1$-submanifold of $\mathbb{R}^d$, assume that $\mathcal{M}\cap U\cap A\neq\emptyset$, assume for every $\theta\in(\mathcal{M}\cap U)$ that $\rank((\Hess f)(\theta))=d-\mathfrak{d}$, for every $n\in\mathbb{N}_0$, $M\in\mathbb{N}$, $r\in(0,\infty)$ let $\Theta^{k,M,r}_{n}\colon\Omega\rightarrow\mathbb{R}^d$, $k\in\mathbb{N}$, be i.i.d.\ random variables, assume for every $M\in\mathbb{N}$, $r\in(0,\infty)$ that $\Theta^{1,M,r}_{0}$ is continuous uniformly distributed on $A$, assume for every $M\in\mathbb{N}$, $r\in(0,\infty)$ that $\Theta^{M,r}_0$ and $\big(X_{n,m}\big)_{n,m\in\mathbb{N}}$ are independent, and assume for every $n,M\in\mathbb{N}$, $r\in(0,\infty)$ that
\begin{equation}\Theta^{1,M,r}_{n}=\Theta^{1,M,r}_{n-1}-\frac{r}{n^\rho M}\!\left[\sum_{m=1}^M(\nabla_\theta F)(\Theta^{1,M,r}_{n-1},X_{n,m})\right].\end{equation}
Then for every $x_0\in(\mathcal{M}\cap U\cap A)$  there exist $R_0,\delta_0,\mathfrak{r},c\in(0,\infty)$ such that for every $R\in(0,R_0]$, $\delta\in(0,\delta_0]$, $r\in(0,\mathfrak{r}]$, $n,M,K\in\mathbb{N}$, $\varepsilon\in(0,1]$ it holds that
\begin{equation}\begin{aligned}\label{mp_0} & \mathbb{P}\Big( \min_{k\in\{1,2,\ldots,K\}}\mathbf{d}(\Theta^{k,M,r}_n,\mathcal{M}\cap U)\geq\varepsilon\Big)\leq   \\ & \left(\frac{\lambda\big(A\backslash V_{\nicefrac{R}{2},\delta}(x_0)\big)}{\lambda(A)} +c\varepsilon^{-2}n^{-\rho}+1- \prod_{k=1}^n\left(1-\frac{c}{Mk^{2\rho}}\right)_++cM^{-1}n^{1-\rho}+\frac{cr\left(1+M^{-\frac{1}{2}}n^{1-\rho}\right)}{\left(\frac{R}{2}-2\delta\right)_+}\right)^K.\end{aligned}\end{equation}
\end{cor}

\begin{proof}[Proof of Corollary~\ref{cor_ng_one_path}] Let $x_0\in(\mathcal{M}\cap U\cap A)$.  Since $U\subseteq\mathbb{R}^d$ is open, fix $V\in\Proj(x_0)$ (cf. Definition~\ref{define_proj}) which satisfies that $V\subseteq U$.  Fix $R_0,\delta_0\in(0,\infty)$ which satisfy the conclusion of Proposition~\ref{cts_prop_tub} for this set $V$.  Fix $ \mathfrak{r}\in(0,\infty)$ which satisfy the conclusions of Lemma~\ref{lem_projection} and Proposition~\ref{ng_good_set}.  Let $R\in(0,R_0]$, $\delta\in(0,\delta_0]$, $r\in(0,\mathfrak{r}]$, $n,M,K\in\mathbb{N}$.  Since the $\Theta^{k,M,r}_n$, $k\in\mathbb{N}$, are i.i.d.\ it holds that
\begin{equation}\label{mp_1}\begin{aligned} \mathbb{P}\Big( \min_{k\in\{1,2,\ldots,K\}}\mathbf{d}(\Theta^{k,M,r}_n,\mathcal{M}\cap U)\geq\varepsilon\Big) & = \prod_{k=1}^K\mathbb{P}\Big( \mathbf{d}(\Theta^{k,M,r}_n,\mathcal{M}\cap U)\geq\varepsilon\Big) \\ & =  \mathbb{P}\Big( \mathbf{d}(\Theta^{1,M,r}_n,\mathcal{M}\cap U)\geq\varepsilon\Big)^K.\end{aligned}\end{equation}
Theorem~\ref{ng_one_path} and \eqref{mp_1} prove estimate \eqref{mp_0}, which completes the proof of Corollary~\ref{cor_ng_one_path}.  \end{proof}

The following corollary translates the convergence of $\Theta^{k,M,r}_n$, $k\in\{1,2,\ldots,K\}$, to the local manifold of minima $\mathcal{M}\cap U$ into a statement concerning the minimization of the objective function.  The proof is a consequence of Corollary~\ref{cor_ng_one_path} and the local regularity of the objective function.

\begin{cor}\label{obj_ng_one_path} Let $d\in\mathbb{N}$, $\mathfrak{d}\in\{ 0, 1, \ldots, d - 1 \}$, $\rho\in(\nicefrac{2}{3},1)$, let $\abs{\cdot}\colon\mathbb{R}^d\rightarrow\mathbb{R}$ be the standard norm on $\mathbb{R}^d$, let $U\subseteq\mathbb{R}^d$ be an open set, let $A\subseteq\mathbb{R}^d$ be a bounded open set, let $(\Omega,\mathcal{F},\mathbb{P})$ be a probability space, let $(S,\mathcal{S})$ be a measurable space, let $F=(F(\theta,x))_{(\theta,x)\in\mathbb{R}^d\times S}\colon  \mathbb{R}^d\times S\rightarrow\mathbb{R}$ be a measurable function, let $X_{n,m}\colon  \Omega\rightarrow S$, $n,m\in\mathbb{N}$, be i.i.d.\ random variables which satisfy for every $\theta\in\mathbb{R}^d$ that $\mathbb{E}\big[ |F(\theta,X_{1,1})|^2\big]<\infty$, let $f\colon\mathbb{R}^d\rightarrow\mathbb{R}$ be the function which satisfies for every $\theta\in\mathbb{R}^d$ that $f(\theta)=\mathbb{E}\big[F(\theta,X_{1,1})\big]$, let $\mathcal{M}\subseteq\mathbb{R}^d$ satisfy that
\begin{equation}\mathcal{M}=\big\{\theta\in\mathbb{R}^d\colon [f(\theta)=\inf\nolimits_{\vartheta\in\mathbb{R}^d} f(\vartheta)]\big\},\end{equation}
let $(\cdot)_+\colon\mathbb{R}\rightarrow\mathbb{R}$ be the function which satisfies for every $x\in\mathbb{R}$ that
\begin{equation} (x)_+=\max(x,0),\end{equation}
assume for every $x\in S$ that $\mathbb{R}^d\ni\theta\mapsto F(\theta,x)\in\mathbb{R}$ is a locally Lipschitz continuous function, assume that $f|_U\colon U\rightarrow\mathbb{R}$ is a three times continuously differentiable function, assume for every non-empty compact set $\mathfrak{C}\subseteq U$ that $\sup\nolimits_{\theta\in \mathfrak{C}}\mathbb{E}\big[|F(\theta,X_{1,1})|^2+|(\nabla_\theta F)(\theta,X_{1,1})|^2\big]<\infty$, assume that $\mathcal{M}\cap U$ is a $\mathfrak{d}$-dimensional $\C^1$-submanifold of $\mathbb{R}^d$, assume that $\mathcal{M}\cap U\cap A\neq\emptyset$, assume for every $\theta\in(\mathcal{M}\cap U)$ that $\rank((\Hess f)(\theta))=d-\mathfrak{d}$, for every $n\in\mathbb{N}_0$, $M\in\mathbb{N}$, $r\in(0,\infty)$ let $\Theta^{k,M,r}_{n}\colon\Omega\rightarrow\mathbb{R}^d$, $k\in\mathbb{N}$, be i.i.d.\ random variables, assume for every $M\in\mathbb{N}$, $r\in(0,\infty)$ that $\Theta^{1,M,r}_{0}$ is continuous uniformly distributed on $A$, assume for every $M\in\mathbb{N}$, $r\in(0,\infty)$ that $\Theta^{M,r}_0$ and $\big(X_{n,m}\big)_{n,m\in\mathbb{N}}$ are independent, and assume for every $n,M\in\mathbb{N}$, $r\in(0,\infty)$ that
\begin{equation}\Theta^{1,M,r}_{n}=\Theta^{1,M,r}_{n-1}-\frac{r}{n^\rho M}\!\left[\sum_{m=1}^M(\nabla_\theta F)(\Theta^{1,M,r}_{n-1},X_{n,m})\right].\end{equation}
Then for every $x_0\in(\mathcal{M}\cap U\cap A)$  there exist $R_0,\delta_0,\mathfrak{r},c\in(0,\infty)$ such that for every $R\in(0,R_0]$, $\delta\in(0,\delta_0]$, $r\in(0,\mathfrak{r}]$, $n,M,K\in\mathbb{N}$, $\varepsilon\in(0,1]$ it holds that
\begin{equation}\begin{aligned}& \mathbb{P}\Big( \Big[\big[\min_{k\in\{1,2,\ldots,K\}}f(\Theta^{k,M,r}_n)\big]-\inf\nolimits_{\theta\in\mathbb{R}^d}f(\theta)\Big]\geq\varepsilon\Big)\leq   \\ & \left(\frac{\lambda\big(A\backslash V_{\nicefrac{R}{2},\delta}(x_0)\big)}{\lambda(A)} +c\varepsilon^{-2}n^{-\rho}+1- \prod_{k=1}^n\left(1-\frac{c}{Mk^{2\rho}}\right)_++cM^{-1}n^{1-\rho}+\frac{cr\left(1+M^{-\frac{1}{2}}n^{1-\rho}\right)}{\left(\frac{R}{2}-2\delta\right)_+}\right)^K.\end{aligned}\end{equation}
\end{cor}

\begin{proof}[Proof of Corollary~\ref{obj_ng_one_path}] The proof is an immediate consequence of Corollary~\ref{cor_ng_one_path} and the local regularity of the objective function.\end{proof}

Under the assumptions and notations of Corollary~\ref{obj_ng_one_path}, since a random variable $\varTheta^{K,M,r}_n\colon\Omega\rightarrow\mathbb{R}^d$ satisfy that
\begin{equation}f(\varTheta^{K,M,r}_n)=\Big[\min_{k\in\{1,2,\ldots,K\}}f(\Theta^{k,M,r}_n)\Big],\end{equation}
is either computationally inefficient or computationally impossible to obtain, we will prove that such a minimizer can be efficiently computed using mini-batch averages.  In the following lemma, we prove that there exists a measurable selection that minimizes a mini-batch approximation.

\begin{lem}\label{lem_measurable_selection}  Let $d\in\mathbb{N}$, let $(\Omega,\mathcal{F},\mathbb{P})$ be a probability space, let $(S,\mathcal{S})$ be a measurable space, let $F\colon\mathbb{R}^d\times S\rightarrow\mathbb{R}$ be a measurable function, let $X_k\colon\Omega\rightarrow S$, $k\in\mathbb{N}$, be i.i.d.\ random variables, and let $\Theta^k\colon\Omega\rightarrow \mathbb{R}^d$, $k\in\mathbb{N}$, be i.i.d. random variables.  Then for every $K,\mathfrak{M}\in\mathbb{N}$ there exists a random variable $\varTheta^{K,\mathfrak{M}}\colon\Omega\rightarrow\mathbb{R}^d$ such that
\begin{equation}\label{measurable_selection_1} \sum_{m=1}^\mathfrak{M}F(\varTheta^{K,\mathfrak{M}},X_m)=\Big[\min_{k\in\{1,2, \ldots,K\}}\Big(\sum_{m=1}^\mathfrak{M}F(\Theta^k, X_m)\Big)\Big].\end{equation}
\end{lem}

\begin{proof}[Proof of Lemma~\ref{lem_measurable_selection}]  Let $K,\mathfrak{M}\in\mathbb{N}$.  Let $\mathfrak{K}\colon\Omega\rightarrow\{1,2,\ldots,\mathfrak{M}\}$ satisfy for every $\omega\in\Omega$ that
\begin{equation}\label{measurable_selection_2} \mathfrak{K}(\omega)=\min\Big\{k\in\{1,2,\ldots,K\}\colon \sum_{m=1}^\mathfrak{M}F(\Theta^k(\omega),X_m)=\Big[\min_{j\in\{1,2,\ldots,\mathfrak{M}\}}\Big(\sum_{m=1}^\mathfrak{M}F(\Theta^j, X_m)\Big)\Big]\Big\}.\end{equation}
Let $\varTheta^{K,\mathfrak{M}}\colon\Omega\rightarrow\mathbb{R}^d$ satisfy for every $\omega\in\Omega$ that
\begin{equation}\label{measurable_selection_3} \varTheta^{K,\mathfrak{M}}(\omega)=\Theta^{\mathfrak{K}(\omega)}(\omega).\end{equation}
It follow from \eqref{measurable_selection_2} and \eqref{measurable_selection_3} that $\Theta^{K,\mathfrak{M}}$ is measurable and satisfies \eqref{measurable_selection_1}, which completes the proof of Lemma~\ref{lem_measurable_selection}.  \end{proof}

In the following theorem, we prove that the minimum appearing on the lefthand side of \eqref{mp_0} can be efficiently computed using mini-batch averages of the type appearing in Lemma~\ref{lem_measurable_selection}.  

 \begin{thm}\label{thm_intro_ng_one_path}  Let $d\in\mathbb{N}$, $\mathfrak{d}\in\{ 0, 1, \ldots, d - 1 \}$, $\rho\in(\nicefrac{2}{3},1)$, let $U\subseteq\mathbb{R}^d$ be an open set, let $A\subseteq\mathbb{R}^d$ be a bounded open set, let $(\Omega,\mathcal{F},\mathbb{P})$ be a probability space, let $(S,\mathcal{S})$ be a measurable space, let $F=(F(\theta,x))_{(\theta,x)\in\mathbb{R}^d\times S}\colon  \mathbb{R}^d\times S\rightarrow\mathbb{R}$ be a measurable function, let $X_{n,m}\colon  \Omega\rightarrow S$, $n,m\in\mathbb{N}$, be i.i.d.\ random variables which satisfy for every $\theta\in\mathbb{R}^d$ that $\mathbb{E}\big[ |F(\theta,X_{1,1})|^2\big]<\infty$, let $f\colon\mathbb{R}^d\rightarrow\mathbb{R}$ be the function which satisfies for every $\theta\in\mathbb{R}^d$ that $f(\theta)=\mathbb{E}\big[F(\theta,X_{1,1})\big]$, let $\mathcal{M}\subseteq\mathbb{R}^d$ satisfy that
\begin{equation}\mathcal{M}=\big\{\theta\in\mathbb{R}^d\colon [f(\theta)=\inf\nolimits_{\vartheta\in\mathbb{R}^d} f(\vartheta)]\big\},\end{equation}
let $(\cdot)_+\colon\mathbb{R}\rightarrow\mathbb{R}$ be the function which satisfies for every $x\in\mathbb{R}$ that
\begin{equation} (x)_+=\max(x,0),\end{equation}
assume for every $x\in S$ that $\mathbb{R}^d\ni\theta\mapsto F(\theta,x)\in\mathbb{R}$ is a locally Lipschitz continuous function, assume that $f|_U\colon U\rightarrow\mathbb{R}$ is a three times continuously differentiable function, assume for every non-empty compact set $\mathfrak{C}\subseteq U$ that $\sup\nolimits_{\theta\in \mathfrak{C}}\mathbb{E}\big[|F(\theta,X_{1,1})|^2+|(\nabla_\theta F)(\theta,X_{1,1})|^2\big]<\infty$, assume that $\mathcal{M}\cap U$ is a $\mathfrak{d}$-dimensional $\C^1$-submanifold of $\mathbb{R}^d$, assume that $\mathcal{M}\cap U\cap A\neq\emptyset$, assume for every $\theta\in(\mathcal{M}\cap U)$ that $\rank((\Hess f)(\theta))=d-\mathfrak{d}$, for every $n\in\mathbb{N}_0$, $M\in\mathbb{N}$, $r\in(0,\infty)$ let $\Theta^{k,M,r}_{n}\colon\Omega\rightarrow\mathbb{R}^d$, $k\in\mathbb{N}$, be i.i.d.\ random variables, assume for every $n,M\in\mathbb{N}$, $r\in(0,\infty)$ that $(\Theta^{k,M,r}_{n-1})_{k\in\{2,3,\ldots\}}$ and $(X_{n,k})_{k\in\mathbb{N}}$ are independent, assume for every $M\in\mathbb{N}$, $r\in(0,\infty)$ that $\Theta^{1,M,r}_{0}$ is continuous uniformly distributed on $A$, assume for every $M\in\mathbb{N}$, $r\in(0,\infty)$ that $\Theta^{M,r}_0$ and $\big(X_{n,m}\big)_{n,m\in\mathbb{N}}$ are independent, assume for every $n,M\in\mathbb{N}$, $r\in(0,\infty)$ that
\begin{equation}\Theta^{1,M,r}_{n}=\Theta^{1,M,r}_{n-1}-\frac{r}{n^\rho M}\!\left[\sum_{m=1}^M(\nabla_\theta F)(\Theta^{1,M,r}_{n-1},X_{n,m})\right],\end{equation}
and for every $n,M,\mathfrak{M},K\in\mathbb{N}$, $r\in(0,\infty)$ let $\varTheta^{K,M,\mathfrak{M},r}_n\colon\Omega\rightarrow\mathbb{R}^d$ be a random variable which satisfies that
\begin{equation}\sum_{m=1}^\mathfrak{M}F(\varTheta^{K,M,\mathfrak{M},r}_n, X_{n+1,m})=\Big[\min_{k\in\{1,2,\ldots,K\}}\Big(\sum_{m=1}^\mathfrak{M}F(\Theta^{k,M,r}_n, X_{n+1,m})\Big)\Big].\end{equation}
Then for every $x_0\in(\mathcal{M}\cap U\cap A)$  there exist $R_0,\delta_0,\mathfrak{r},c\in(0,\infty)$ such that for every $R\in(0,R_0]$, $\delta\in(0,\delta_0]$, $r\in(0,\mathfrak{r}]$, $n,M,\mathfrak{M},K\in\mathbb{N}$, $\varepsilon\in(0,1]$ it holds that
\begin{equation}\begin{aligned} & \mathbb{P}\Big( \Big[f(\varTheta^{K,M,\mathfrak{M},r}_n)-\inf\nolimits_{\theta\in\mathbb{R}^d}f(\theta)\Big]\geq\varepsilon\Big) \leq \frac{cK}{\varepsilon^2\mathfrak{M}} \\ & + \left(\frac{\lambda\big(A\backslash V_{\nicefrac{R}{2},\delta}(x_0)\big)}{\lambda(A)} +c\varepsilon^{-2}n^{-\rho}+1- \prod_{k=1}^n\left(1-\frac{c}{Mk^{2\rho}}\right)_++cM^{-1}n^{1-\rho}+\frac{cr\left(1+M^{-\frac{1}{2}}n^{1-\rho}\right)}{\left(\frac{R}{2}-2\delta\right)_+}\right)^K.\end{aligned}\end{equation}
\end{thm}

\begin{proof}[Proof of Theorem~\ref{thm_intro_ng_one_path}] Let $x_0\in(\mathcal{M}\cap U\cap A)$.  Since $U\subseteq\mathbb{R}^d$ is open, fix  $V\in\Proj(x_0)$ (cf. Definition~\ref{define_proj}) which satisfies that $V\subseteq U$.  Fix $R_0,\delta_0\in(0,\infty)$ which satisfy the conclusion of Proposition~\ref{cts_prop_tub} for this set $V$.  Fix $ \mathfrak{r}\in(0,\infty)$ which satisfy the conclusions of Lemma~\ref{lem_projection} and Proposition~\ref{ng_good_set}.  Let $R\in(0,R_0]$, $\delta\in(0,\delta_0]$, $r\in(0,\mathfrak{r}]$, $n,M,\mathfrak{M},K\in\mathbb{N}$.  For every $i\in\{1,2,\ldots,K\}$ let $B'_i\subseteq\Omega$ satisfy that
\begin{equation}\label{ting_0} B'_i=\Big\{\omega\in\Omega\colon f(\Theta^{i,M,r}_n(\omega))=\Big[\min_{k\in\{1,2,\ldots,K\}}f(\Theta^{k,M,r}_n(\omega))\Big]\Big\},\end{equation}
and let $B_1\subseteq\Omega$ satisfy that $B_1=B'_1$ and for every $i\in\{2,3,\ldots,K\}$ let $B_i\subseteq\Omega$ satisfy that $B_i=B'_i\backslash \cup_{m=1}^{i-1}B_m$.  Since the events $B_i$, $i\in\{1,2,\ldots,K\}$, are disjoint, it holds that
\begin{equation}\label{ting_1}\begin{aligned}  & \mathbb{P}\Big( \Big[f(\varTheta^{K,M,\mathfrak{M},r}_n)-\inf\nolimits_{\theta\in\mathbb{R}^d}f(\theta)\Big]\geq\varepsilon\Big) \\ & =  \sum_{i=1}^K\mathbb{P}\Big( \Big[f(\varTheta^{K,M,\mathfrak{M},r}_n)-\inf\nolimits_{\theta\in\mathbb{R}^d}f(\theta)\Big]\geq\varepsilon, B_i\Big) \\ & =   \sum_{i=1}^K\mathbb{P}\Big( \big[f(\varTheta^{K,M,\mathfrak{M},r}_n)-f(\Theta^{i,M,r}_n)+f(\Theta^{i,M,r}_n)-\inf_{\theta\in\mathbb{R}^d}f(\theta)\big]\geq\varepsilon, B_i\Big) \\ & \leq \mathbb{P}\Big( \Big[\big[\min_{k\in\{1,2,\ldots,K\}}f(\Theta^{k,M,r}_n)\big]-\inf\nolimits_{\theta\in\mathbb{R}^d}f(\theta)\Big]\geq\frac{\varepsilon}{2}\Big)+\sum_{i=1}^K\mathbb{P}\Big( \big[f(\varTheta^{K,M,\mathfrak{M},r}_n)-f(\Theta^{i,M,r}_n)\geq\frac{\varepsilon}{2}, B_i\Big).\end{aligned}\end{equation}
For the first term on the righthand side of \eqref{ting_1}, Corollary~\ref{obj_ng_one_path} proves that there exists $c\in(0,\infty)$ which satisfies that
\begin{equation}\label{ting_2} \begin{aligned} & \mathbb{P}\Big(\Big[\big[\min_{k\in\{1,2,\ldots,K\}}f(\Theta^{k,M,r}_n)\big]-\inf\nolimits_{\theta\in\mathbb{R}^d}f(\theta)\Big]\geq\frac{\varepsilon}{2}\Big) \\ &\leq  \left(\frac{\lambda\big(A\backslash V_{\nicefrac{R}{2},\delta}(x_0)\big)}{\lambda(A)} +c\varepsilon^{-2}n^{-\rho}+1- \prod_{k=1}^n\left(1-\frac{c}{Mk^{2\rho}}\right)_++cM^{-1}n^{1-\rho}+\frac{cr\left(1+M^{-\frac{1}{2}}n^{1-\rho}\right)}{\left(\frac{R}{2}-2\delta\right)_+}\right)^K. \end{aligned}\end{equation}
We will now estimate the second term on the righthand side of \eqref{ting_2}.  Let $\tilde{B}_j\subseteq\Omega$, $j\in\{1,2,\ldots,K\}$, be disjoint events which satisfy that $\Omega=\coprod_{j\in\{1,2,\ldots,K\}}\tilde{B}_j$ and that
\begin{equation}\tilde{B}_j\subseteq \Big\{\omega\in \Omega\colon \sum_{m=1}^\mathfrak{M}F(\varTheta^{K,M,\mathcal{M},r}_n(\omega),X_{n+1,m}(\omega))=\sum_{m=1}^\mathfrak{M}F(\Theta^{j,M,r}_n(\omega),X_{n+1,m}(\omega))\Big\}.\end{equation}
Since the events $\tilde{B}_j$, $j\in\{1,2,\ldots,K\}$, are disjoint, the final term of \eqref{ting_1} satisfies that
\begin{equation}\label{ting_3}\sum_{i=1}^K\mathbb{P}\Big( f(\varTheta^{K,M,\mathfrak{M},r}_n)-f(\Theta^{i,M,r}_n)\geq\frac{\varepsilon}{2}, B_i\Big)=\sum_{i,j=1}^K\mathbb{P}\Big( f(\Theta^{j,M,r}_n)-f(\Theta^{i,M,r}_n)\geq\frac{\varepsilon}{2}, B_i,\tilde{B}_j\Big).\end{equation}
Let $F^{\mathfrak{M},n}\colon\mathbb{R}^d\times\Omega\rightarrow\mathbb{R}$ be the function which satisfies for every $\theta\in\mathbb{R}^d$, $\omega\in\Omega$ that
\begin{equation}F^{\mathfrak{M},n}(\theta,\omega)=\frac{1}{\mathfrak{M}}\sum_{m=1}^\mathfrak{M}F(\theta, X_{n+1,m}(\omega)).\end{equation}
For every $i,j\in\{1,2,\ldots,K\}$, since it holds for every $\omega\in B_i\cap\tilde{B}_j$ that
\begin{equation}F^{\mathfrak{M},n}(\Theta^{j,M,r}_n(\omega),\omega)-F^{\mathfrak{M},n}(\Theta^{i,M,r}_n(\omega),\omega)\leq 0,\end{equation}
it holds for every $i,j\in\{1,2,\ldots,K\}$ that
\begin{equation}\label{ting_4} \begin{aligned} & \mathbb{P}\Big( f(\Theta^{j,M,r}_n)-f(\Theta^{i,M,r}_n)\geq\frac{\varepsilon}{2}, B_i, \tilde{B}_j\Big) \\ & \leq \mathbb{P}\Big( f(\Theta^{j,M,r}_n(\omega))-F^{\mathfrak{M},n}(\Theta^{j,M,r}_n(\omega),\omega)+F^{\mathfrak{M},n}(\Theta^{i,M,r}_n(\omega),\omega)-f(\Theta^{i,M,r}_n(\omega))\geq\frac{\varepsilon}{2}, B_i,\tilde{B}_j\Big) \\ & \leq \mathbb{P}\Big( \Big|f(\Theta^{j,M,r}_n(\omega))-F^{\mathfrak{M},n}(\Theta^{j,M,r}_n(\omega),\omega)\Big|\geq\frac{\varepsilon}{4}, B_i,\tilde{B}_j\Big) \\ & \quad +\mathbb{P}\Big( \Big|f(\Theta^{i,M,r}_n(\omega))-F^{\mathfrak{M},n}(\Theta^{i,M,r}_n(\omega),\omega)\Big|\geq\frac{\varepsilon}{4}, B_i,\tilde{B}_j\Big).\end{aligned}\end{equation}
It follows from \eqref{ting_3} and \eqref{ting_4} that
\begin{equation}\label{ting_5}\begin{aligned} & \sum_{i=1}^K\mathbb{P}\Big( f(\varTheta^{K,M,\mathfrak{M},r}_n)-f(\Theta^{i,M,r}_n)\geq\frac{\varepsilon}{2}, B_i\Big) \\ & \leq \sum_{j=1}^K\mathbb{P}\Big( \Big|f(\Theta^{j,M,r}_n(\omega))-F^{\mathfrak{M},n}(\Theta^{j,M,r}_n(\omega),\omega)\Big|\geq\frac{\varepsilon}{4}, \tilde{B}_j\Big) \\ & \quad +  \sum_{i=1}^K\mathbb{P}\Big( \Big|f(\Theta^{i,M,r}_n(\omega))-F^{\mathfrak{M},n}(\Theta^{i,M,r}_n(\omega),\omega)\Big|\geq\frac{\varepsilon}{4}, B_i\Big). \end{aligned}\end{equation}
For the first term on the righthand side of \eqref{ting_5}, it holds that
\begin{equation}\begin{aligned}& \sum_{j=1}^K\mathbb{P}\Big( \Big|f(\Theta^{j,M,r}_n(\omega))-F^{\mathfrak{M},n}(\Theta^{j,M,r}_n(\omega),\omega)\Big|\geq\frac{\varepsilon}{4}, \tilde{B}_j\Big) \\ & \leq \sum_{j=1}^K\mathbb{P}\Big( \Big|f(\Theta^{j,M,r}_n(\omega))-F^{\mathfrak{M},n}(\Theta^{j,M,r}_n(\omega),\omega)\Big|\geq\frac{\varepsilon}{4}\Big).\end{aligned}\end{equation}
Since the random variables $(\Theta^{k,M,r}_n)_{k\in\mathbb{N}}$ and $(X_{n+1,k})_{k\in\mathbb{N}}$ are independent, since the $(\Theta^{k,M,r}_n)_{k\in\mathbb{N}}$ are identically distributed, and since the distribution of $\Theta^{1,M,r}_n$ has bounded support on $\mathbb{R}^d$, for the distribution $\mu_n$ of $\Theta^{1,M,r}_n$ on $\mathbb{R}^d$, Lemma~\ref{variance_lem}, Chebyshev's inequality, and the definition of $F^{\mathfrak{M},n}$ prove that that there exists $c\in(0,\infty)$ which satisfies for every $j\in\{1,\ldots,K\}$ that
\begin{equation}\begin{aligned}\mathbb{P}\Big( \Big|f(\Theta^{j,M,r}_n(\omega))-F^{\mathfrak{M},n}(\Theta^{j,M,r}_n(\omega),\omega)\Big|\geq\frac{\varepsilon}{4}\Big) & =  \int_{\mathbb{R}^d}\mathbb{P}\Big( \Big|f(\theta)-\frac{1}{\mathfrak{M}}\sum_{m=1}^\mathfrak{M}F(\theta, X_{n+1,m})\Big|\geq\frac{\varepsilon}{4}\Big)\mu_n(\dd\theta) \\ & \leq  \frac{c}{\varepsilon^2\mathfrak{M}}.\end{aligned}\end{equation}
Therefore, it holds that
\begin{equation}\label{ting_6} \sum_{j=1}^K\mathbb{P}\Big( \Big|f(\Theta^{j,M,r}_n(\omega))-F^{\mathfrak{M},n}(\Theta^{j,M,r}_n(\omega),\omega)\Big|\geq\frac{\varepsilon}{4}, \tilde{B}_j\Big)\leq \frac{cK}{\varepsilon^2\mathfrak{M}}.\end{equation}
For the second term on the righthand side of \eqref{ting_5}, it is sufficient to apply the same argument, which proves that there exists $c\in(0,\infty)$ which satisfies that
\begin{equation}\sum_{i=1}^K\mathbb{P}\Big( \Big|f(\Theta^{i,M,r}_n(\omega))-F^{\mathfrak{M},n}(\Theta^{i,M,r}_n(\omega),\omega)\Big|\geq\frac{\varepsilon}{4}, B_i\Big)\leq \frac{cK}{\varepsilon^2\mathfrak{M}}.\end{equation}
Returning to \eqref{ting_3}, it follows from \eqref{ting_5} and \eqref{ting_6} that there exists $c\in(0,\infty)$ which satisfies that
\begin{equation}\label{ting_8}\sum_{i=1}^K\mathbb{P}\Big( f(\varTheta^{K,M,\mathfrak{M},r}_n)-f(\Theta^{i,M,r}_n)\geq\frac{\varepsilon}{2}, B_i\Big)\leq \frac{cK}{\varepsilon^2\mathfrak{M}}.\end{equation}
Returning finally to \eqref{ting_1}, it follows from \eqref{ting_2} and \eqref{ting_8} that there exists $c\in(0,\infty)$ which satisfies that
\begin{equation}\begin{aligned} & \mathbb{P}\Big( \Big[f(\varTheta^{K,M,\mathfrak{M},r}_n)-\inf\nolimits_{\theta\in\mathbb{R}^d}f(\theta)\Big]\geq\varepsilon\Big) \leq \frac{cK}{\varepsilon^2\mathfrak{M}} \\ & +\left(\frac{\lambda\big(A\backslash V_{\nicefrac{R}{2},\delta}(x_0)\big)}{\lambda(A)} +c\varepsilon^{-2}n^{-\rho}+1- \prod_{k=1}^n\left(1-\frac{c}{Mk^{2\rho}}\right)_++cM^{-1}n^{1-\rho}+\frac{cr\left(1+M^{-\frac{1}{2}}n^{1-\rho}\right)}{\left(\frac{R}{2}-2\delta\right)_+}\right)^K,\end{aligned}\end{equation}
which completes the proof of Theorem~\ref{thm_intro_ng_one_path}. \end{proof}

In the final corollary of this section, we will compute the computational efficiency of the algorithm proposed in Theorem~\ref{thm_intro_ng_one_path}.  The constant implicitly depends on the computational cost of computing $F$ and $\nabla_\theta F$ and initializing the random variable $X_{1,1}$, but it does not depend upon the running time $n\in\mathbb{N}$, the sampling size $K\in\mathbb{N}$, or the mini-batch sizes $M,\mathfrak{M}\in\mathbb{N}$.

 \begin{cor}\label{cor_computation}  Let $d\in\mathbb{N}$, $\mathfrak{d}\in\{ 0, 1, \ldots, d - 1 \}$, $\rho\in(\nicefrac{2}{3},1)$, let $U\subseteq\mathbb{R}^d$ be an open set, let $A\subseteq\mathbb{R}^d$ be a bounded open set, let $(\Omega,\mathcal{F},\mathbb{P})$ be a probability space, let $(S,\mathcal{S})$ be a measurable space, let $F=(F(\theta,x))_{(\theta,x)\in\mathbb{R}^d\times S}\colon  \mathbb{R}^d\times S\rightarrow\mathbb{R}$ be a measurable function, let $X_{n,m}\colon  \Omega\rightarrow S$, $n,m\in\mathbb{N}$, be i.i.d.\ random variables which satisfy for every $\theta\in\mathbb{R}^d$ that $\mathbb{E}\big[ |F(\theta,X_{1,1})|^2\big]<\infty$, let $f\colon\mathbb{R}^d\rightarrow\mathbb{R}$ be the function which satisfies for every $\theta\in\mathbb{R}^d$ that $f(\theta)=\mathbb{E}\big[F(\theta,X_{1,1})\big]$, let $\mathcal{M}\subseteq\mathbb{R}^d$ satisfy that
\begin{equation}\mathcal{M}=\big\{\theta\in\mathbb{R}^d\colon [f(\theta)=\inf\nolimits_{\vartheta\in\mathbb{R}^d} f(\vartheta)]\big\},\end{equation}
assume for every $x\in S$ that $\mathbb{R}^d\ni\theta\mapsto F(\theta,x)\in\mathbb{R}$ is a locally Lipschitz continuous function, assume that $f|_U\colon U\rightarrow\mathbb{R}$ is a three times continuously differentiable function, assume for every non-empty compact set $\mathfrak{C}\subseteq U$ that $\sup\nolimits_{\theta\in \mathfrak{C}}\mathbb{E}\big[|F(\theta,X_{1,1})|^2+|(\nabla_\theta F)(\theta,X_{1,1})|^2\big]<\infty$, assume that $\mathcal{M}\cap U$ is a $\mathfrak{d}$-dimensional $\C^1$-submanifold of $\mathbb{R}^d$, assume that $\mathcal{M}\cap U\cap A\neq\emptyset$, assume for every $\theta\in(\mathcal{M}\cap U)$ that $\rank((\Hess f)(\theta))=d-\mathfrak{d}$, for every $n\in\mathbb{N}_0$, $M\in\mathbb{N}$, $r\in(0,\infty)$ let $\Theta^{k,M,r}_{n}\colon\Omega\rightarrow\mathbb{R}^d$, $k\in\mathbb{N}$, be i.i.d.\ random variables, assume for every $n,M\in\mathbb{N}$, $r\in(0,\infty)$ that $(\Theta^{k,M,r}_{n-1})_{k\in\{2,3,\ldots\}}$ and $(X_{n,k})_{k\in\mathbb{N}}$ are independent, assume for every $M\in\mathbb{N}$, $r\in(0,\infty)$ that $\Theta^{1,M,r}_{0}$ is continuous uniformly distributed on $A$, assume for every $M\in\mathbb{N}$, $r\in(0,\infty)$ that $\Theta^{M,r}_0$ and $\big(X_{n,m}\big)_{n,m\in\mathbb{N}}$ are independent, assume for every $n,M\in\mathbb{N}$, $r\in(0,\infty)$ that
\begin{equation}\Theta^{1,M,r}_{n}=\Theta^{1,M,r}_{n-1}-\frac{r}{n^\rho M}\!\left[\sum_{m=1}^M(\nabla_\theta F)(\Theta^{1,M,r}_{n-1},X_{n,m})\right],\end{equation}
and for every $n,M,\mathfrak{M},K\in\mathbb{N}$, $r\in(0,\infty)$ let $\varTheta^{K,M,\mathfrak{M},r}_n\colon\Omega\rightarrow\mathbb{R}^d$ be a random variable which satisfies that
\begin{equation}\sum_{m=1}^\mathfrak{M}F(\varTheta^{K,M,\mathfrak{M},r}_n, X_{n+1,m})=\Big[\min_{k\in\{1,2,\ldots,K\}}\Big(\sum_{m=1}^\mathfrak{M}F(\Theta^{k,M,r}_n, X_{n+1,m})\Big)\Big].\end{equation}
Then for every $x_0\in(\mathcal{M}\cap U\cap A)$ there exist $R_0,\delta_0,\mathfrak{r}\in(0,\infty)$ such that for every $R\in(0,R_0]$, $\delta\in(0,\delta_0]$, $r\in(0,\mathfrak{r}]$ there exist $c_i\in(0,\infty)$, $i\in\{1,2,3,4\}$, such that for every $\varepsilon,\eta\in(0,1]$, for $n(\varepsilon), M(\varepsilon), K(\eta), \mathfrak{M}(\varepsilon,\eta)\in\mathbb{N}$ which satisfy that
\begin{equation} n(\varepsilon)=c_1\varepsilon^{-\nicefrac{2}{\rho}},\;\; M(\varepsilon)=c_2\varepsilon^{-\nicefrac{4}{\rho}+4},\;\; \mathfrak{M}(\varepsilon,\eta)=c_3\varepsilon^{-2}\eta^{-1}\abs{\log(\eta)},\;\;\textrm{and}\;\;K=c_4\abs{\log(\eta)},\end{equation}
it holds that
\begin{equation}\mathbb{P}\Big( \Big[f(\varTheta^{K(\eta),M(\varepsilon),\mathfrak{M}(\varepsilon,\eta),r}_{n(\varepsilon)})-\inf\nolimits_{\theta\in\mathbb{R}^d}f(\theta)\Big]\geq\varepsilon\Big) \leq \eta.\end{equation}
\end{cor}

\begin{proof}[Proof of Corollary~\ref{cor_computation}]  Let $x_0\in(\mathcal{M}\cap U\cap A)$.  Let $R_0,\delta_0,\mathfrak{r}\in(0,\infty)$ satisfy the conclusion of Theorem~\ref{thm_intro_ng_one_path}.  Theorem~\ref{thm_intro_ng_one_path} proves that there exists $\overline{c}\in(0,\infty)$ such that for every $R\in(0,R_0]$, $\delta\in(0,\delta_0]$, $r\in(0,\mathfrak{r}]$, $n,M,\mathfrak{M},K\in\mathbb{N}$, $\varepsilon\in(0,1]$ it holds that
\begin{equation}\label{thc_0}\begin{aligned} & \mathbb{P}\Big( \Big[f(\varTheta^{K,M,\mathfrak{M},r}_n)-\inf\nolimits_{\theta\in\mathbb{R}^d}f(\theta)\Big]\geq\varepsilon\Big) \leq \frac{\overline{c}K}{\varepsilon^2\mathfrak{M}} \\ & +\left(\frac{\lambda\big(A\backslash V_{\nicefrac{R}{2},\delta}(x_0)\big)}{\lambda(A)} +\overline{c}\varepsilon^{-2}n^{-\rho}+1- \prod_{k=1}^n\left(1-\frac{\overline{c}}{Mk^{2\rho}}\right)_++\overline{c}M^{-1}n^{1-\rho}+\frac{\overline{c}r\left(1+M^{-\frac{1}{2}}n^{1-\rho}\right)}{\left(\frac{R}{2}-2\delta\right)_+}\right)^K.\end{aligned}\end{equation}
Fix $\overline{R}\in(0,R_0]$, $\overline{\delta}\in(0,\delta_0]$ which satisfy that
\begin{equation}\label{thc_00} \frac{\overline{R}}{2}-2\overline{\delta}>0.\end{equation}
Since $\mathcal{M}\cap U\cap A\neq\emptyset$, it holds that
\begin{equation}\label{thc_1}\frac{\lambda\big(A\backslash V_{\nicefrac{R}{2},\delta}(x_0)\big)}{\lambda(A)}\in(0,1).\end{equation}
For every $M\in\mathbb{N}$ which satisfies that $M\geq 2\overline{c}$, since $\rho\in(\nicefrac{2}{3},1)$ there exists $c\in(0,\infty)$ which satisfies that
\begin{equation}\log\Big(\prod_{k=1}^n\left(1-\frac{\overline{c}}{Mk^{2\rho}}\right)_+\Big)\geq -\frac{c}{M}\sum_{k=1}^Mk^{-2\rho}\geq -\frac{c}{M},\end{equation}
and therefore for every $M\geq 2\overline{c}$ there exists $c\in(0,\infty)$ which satisfies that
\begin{equation}\label{thc_2}-\prod_{k=1}^n\left(1-\frac{\overline{c}}{Mk^{2\rho}}\right)_+\leq -\exp\Big(-\frac{c}{M}\Big).\end{equation}
It follows from \eqref{thc_00} that there exists $c\in(0,\infty)$ which satisfies that
\begin{equation}\label{thc_3}\overline{c}M^{-1}n^{1-\rho}+\frac{\overline{c}rM^{-\frac{1}{2}}n^{1-\rho}}{\left(\frac{\overline{R}}{2}-2\overline{\delta}\right)_+}\leq cM^{-\frac{1}{2}}n^{1-\rho}.\end{equation}
Returning to \eqref{thc_0}, it follows from \eqref{thc_2} and \eqref{thc_3} that there exists $c\in(0,\infty)$ which satisfies that
\begin{equation}\label{thc_4}\begin{aligned} & \mathbb{P}\Big( \Big[f(\varTheta^{K,M,\mathfrak{M},r}_n)-\inf\nolimits_{\theta\in\mathbb{R}^d}f(\theta)\Big]\geq\varepsilon\Big) \leq \frac{\overline{c}K}{\varepsilon^2\mathfrak{M}} \\ & +\left(\frac{\lambda\big(A\backslash V_{\nicefrac{\overline{R}}{2},\overline{\delta}}(x_0)\big)}{\lambda(A)} +\overline{c}\varepsilon^{-2}n^{-\rho}+1-\exp\Big(-\frac{c}{M}\Big)+cM^{-\frac{1}{2}}n^{1-\rho}+\frac{\overline{c}r}{(\frac{\overline{R}}{2}-2\overline{\delta})_+}\right)^K.\end{aligned}\end{equation}
Let $\eta\in(0,1]$.  It follows from \eqref{thc_00}, \eqref{thc_1} and an explicit computation that there exist $c_i\in(0,\infty)$, $i\in\{1,2,3,4\}$, and $\mathfrak{r}_1\in(0,\mathfrak{r}]$ such that for $n(\varepsilon), M(\varepsilon), \mathfrak{M}(\varepsilon,\eta), K(\eta)\in\mathbb{N}$ which satisfy that
\begin{equation}\label{thc_5} n(\varepsilon)=c_1\varepsilon^{-\nicefrac{2}{\rho}},\; M(\varepsilon)=c_2\varepsilon^{-\nicefrac{4}{\rho}+4},\; \mathfrak{M}(\varepsilon,\eta)=c_3\varepsilon^{-2}\eta^{-1}\abs{\log(\eta)},\;\textrm{and}\;K=c_4\abs{\log(\eta)},\end{equation}
it holds that
\begin{equation}\frac{\overline{c}K}{\varepsilon^2\mathfrak{M}(\epsilon,\eta)}\leq\frac{\eta}{2},\end{equation}
and for every $r\in(0,\mathfrak{r}_1]$ that
\begin{equation}\begin{aligned} & \left(\frac{\lambda\big(A\backslash V_{\nicefrac{\overline{R}}{2},\overline{\delta}}(x_0)\big)}{\lambda(A)} +\overline{c}\varepsilon^{-2}n(\varepsilon)^{-\rho}+1-\exp\Big(-\frac{c}{M(\varepsilon)}\Big)+cM(\varepsilon)^{-\frac{1}{2}}n(\varepsilon)^{1-\rho}+\frac{\overline{c}r}{(\frac{\overline{R}}{2}-2\overline{\delta})_+}\right)^{K(\eta)} \\ & \leq\frac{\eta}{2}.\end{aligned}\end{equation}
Returning to \eqref{thc_4}, it follows for every $r\in(0,\mathfrak{r}_1]$ that
\begin{equation}\mathbb{P}\Big( \Big[f(\varTheta^{K(\eta),M(\varepsilon),\mathfrak{M}(\varepsilon,\eta),r}_{n(\varepsilon)})-\inf\nolimits_{\theta\in\mathbb{R}^d}f(\theta)\Big]\geq\varepsilon\Big) \leq \eta,\end{equation}
which completes the proof of Corollary~\ref{cor_computation}.  \end{proof}

\section{Stochastic gradient descent - The compact case}\label{stoch_discrete}

In this section, we will analyze the converge of SGD to the manifold of local minima under the additional assumption that the manifold of local minima is compact.   The essential difference in this case is that SGD cannot leave a basin of attraction along directions tangential to the manifold.  We first observe the convergence of SGD in directions normal to the manifold.

The following proposition is an immediate consequence of Proposition~\ref{ng_converge} and the compactness of $\mathcal{M}\cap U$, where the essential difference in the compact case is that $R\in(0,\infty)$ can be chosen arbitrarily large.  In particular, by compactness, for every $x_0\in(\mathcal{M}\cap U)$ there exists $R_0\in(0,\infty)$ such that for every $R_1,R_2\in[R_0,\infty)$, $\delta\in(0,\infty)$ it holds that $V_{R_1,\delta}(x_0)=V_{R_2,\delta}(x_0)$.  Furthermore, it follows from Remark~\ref{rho_remark} that the results apply to $\rho\in(0,1)$.

\begin{prop}\label{ng_converge_comp}  Let $d\in\mathbb{N}$, $\mathfrak{d}\in\{ 0, 1, \ldots, d - 1 \}$, $\rho\in(0,1)$, let $\abs{\cdot}\colon\mathbb{R}^d\rightarrow\mathbb{R}$ be the standard norm on $\mathbb{R}^d$, let $U\subseteq\mathbb{R}^d$ be an open set, let $(\Omega,\mathcal{F},\mathbb{P})$ be a probability space, let $(S,\mathcal{S})$ be a measurable space, let $F=(F(\theta,x))_{(\theta,x)\in\mathbb{R}^d\times S}\colon  \mathbb{R}^d\times S\rightarrow\mathbb{R}$ be a measurable function, let $X_{n,m}\colon  \Omega\rightarrow S$, $n,m\in\mathbb{N}$, be i.i.d.\ random variables which satisfy for every $\theta\in\mathbb{R}^d$ that $\mathbb{E}\big[ |F(\theta,X_{1,1})|^2\big]<\infty$, let $f\colon\mathbb{R}^d\rightarrow\mathbb{R}$ be the function which satisfies for every $\theta\in\mathbb{R}^d$ that $f(\theta)=\mathbb{E}\big[F(\theta,X_{1,1})\big]$, let $\mathcal{M}\subseteq\mathbb{R}^d$ satisfy that
\begin{equation}\mathcal{M}=\big\{\theta\in\mathbb{R}^d\colon [f(\theta)=\inf\nolimits_{\vartheta\in\mathbb{R}^d} f(\vartheta)]\big\},\end{equation}
let $\mathbf{d}(\cdot,\mathcal{M}\cap U):\mathbb{R}^d\rightarrow\mathbb{R}$ be the function which satisfies for every $x\in\mathbb{R}^d$ that
\begin{equation} \mathbf{d}(x,\mathcal{M}\cap U)=\inf \left\{\abs{x-y}\colon y\in(\mathcal{M}\cap U)\right\},\end{equation}
assume for every $x\in S$ that $\mathbb{R}^d\ni\theta\mapsto F(\theta,x)\in\mathbb{R}$ is a locally Lipschitz continuous function, assume that $f|_U\colon U\rightarrow\mathbb{R}$ is a three times continuously differentiable function, assume for every non-empty compact set $\mathfrak{C}\subseteq U$ that $\sup\nolimits_{\theta\in \mathfrak{C}}\mathbb{E}\big[|F(\theta,X_{1,1})|^2+|(\nabla_\theta F)(\theta,X_{1,1})|^2\big]<\infty$, assume that $\mathcal{M}\cap U$ is a non-empty compact $\mathfrak{d}$-dimensional $\C^1$-submanifold of $\mathbb{R}^d$, assume for every $\theta\in(\mathcal{M}\cap U)$ that $\rank((\Hess f)(\theta))=d-\mathfrak{d}$, for every $M\in\mathbb{N}$, $r\in(0,\infty)$, $\theta\in\mathbb{R}^d$ let $\Theta^{M,r}_{0,\theta}\in\mathbb{R}^d\colon\Omega\rightarrow\mathbb{R}^d$ satisfy for every $\omega\in\Omega$ that $\theta^{M,r}_{0,\theta}(\omega)=\theta$, for every $n,M\in\mathbb{N}$, $r\in(0,\infty)$, $\theta\in\mathbb{R}^d$ let $\Theta^{M,r}_{n,\theta}\colon\Omega\rightarrow\mathbb{R}^d$ satisfy that
\begin{equation}\Theta^{M,r}_{n,\theta}=\Theta^{M,r}_{n-1,\theta}-\frac{r}{n^\rho M}\!\left[\sum_{m=1}^M(\nabla_\theta F)(\Theta^{M,r}_{n-1,\theta},X_{n,m})\right],\end{equation}
and for every $n,M\in\mathbb{N}$, $r,R,\delta\in(0,\infty)$, $\theta\in\mathbb{R}^d$, $x_0\in(\mathcal{M}\cap U)$ let $A_n(M,r,R,\delta,\theta,x_0)\in\mathcal{F}$ satisfy that
\begin{equation}A_n(M,r,R,\delta,\theta,x_0)=\Big\{\forall\;m\in\{0,\ldots,n\}\;\Theta^{M,r}_{m,\theta}\in V_{R,\delta}(x_0)\Big\}.\end{equation}
Then for every $x_0\in(\mathcal{M}\cap U)$ there exist $\delta_0,\mathfrak{r},c\in(0,\infty)$ such that for every $R\in(0,\infty)$, $\delta\in(0,\delta_0]$, $r\in(0,\mathfrak{r}]$, $n,M\in\mathbb{N}$, $\theta\in V_{R,\delta}(x_0)$ (cf. Definition~\ref{dtbn}) it holds that 
\begin{equation}\left(\mathbb{E}\left[\left(\mathbf{d}(\Theta^{M,r}_{n,\theta},\mathcal{M}\cap U)\wedge 1\right)^2\mathbf{1}_{A_{n-1}}\right]\right)^\frac{1}{2}\leq cn^{-\frac{\rho}{2}}.\end{equation}
\end{prop}

\begin{proof}[Proof of Proposition~\ref{ng_converge_comp}] The proof is an immediate consequence of Proposition~\ref{ng_converge} and the compactness of $\mathcal{M}\cap U$. \end{proof}

We will now obtain a lower bound in probability for the events $A_m$, $m\in\mathbb{N}$.  It follows from Proposition~\ref{ng_good_set} and the compactness of $\mathcal{M}\cap U$ that for every $x_0\in(\mathcal{M}\cap U)$ there exist $\delta_0, \mathfrak{r},c\in(0,\infty)$ such that the conclusion of Proposition~\ref{ng_good_set} is satisfied for every $\delta\in(0,\delta_0]$, $ r\in(0, \mathfrak{r}]$, and $R\in(0,\infty)$ for this constant $c\in(0,\infty)$.  That is, since for every $R_1,R_2\in(0,\infty)$ sufficiently large we have $V_{R_1,\delta}(x_0)=V_{R_2,\delta}(x_0)$, it holds that the constant can be chosen independently of $R\in(0,\infty)$.

The proof of the following proposition is then an immediate consequence of Proposition~\ref{ng_good_set}, after using the fact that the constant $c\in(0,\infty)$ is independent of $R\in(0,\infty)$ and passing to the limit $R\rightarrow\infty$.  The improvement in the estimate, when compared to Proposition~\ref{ng_good_set}, is a result of the fact that SGD cannot leave the basin of attraction along the directions tangential to the manifold.

\begin{prop}\label{ng_good_set_comp}  Let $d\in\mathbb{N}$, $\mathfrak{d}\in\{ 0, 1, \ldots, d - 1 \}$, $\rho\in(0,1)$, let $\abs{\cdot}\colon\mathbb{R}^d\rightarrow\mathbb{R}$ be the standard norm on $\mathbb{R}^d$, let $U\subseteq\mathbb{R}^d$ be an open set, let $(\Omega,\mathcal{F},\mathbb{P})$ be a probability space, let $(S,\mathcal{S})$ be a measurable space, let $F=(F(\theta,x))_{(\theta,x)\in\mathbb{R}^d\times S}\colon  \mathbb{R}^d\times S\rightarrow\mathbb{R}$ be a measurable function, let $X_{n,m}\colon  \Omega\rightarrow S$, $n,m\in\mathbb{N}$, be i.i.d.\ random variables which satisfy for every $\theta\in\mathbb{R}^d$ that $\mathbb{E}\big[ |F(\theta,X_{1,1})|^2\big]<\infty$, let $f\colon\mathbb{R}^d\rightarrow\mathbb{R}$ be the function which satisfies for every $\theta\in\mathbb{R}^d$ that $f(\theta)=\mathbb{E}\big[F(\theta,X_{1,1})\big]$, let $\mathcal{M}\subseteq\mathbb{R}^d$ satisfy that
\begin{equation}\mathcal{M}=\big\{\theta\in\mathbb{R}^d\colon [f(\theta)=\inf\nolimits_{\vartheta\in\mathbb{R}^d} f(\vartheta)]\big\},\end{equation}
let $\mathbf{d}(\cdot,\mathcal{M}\cap U):\mathbb{R}^d\rightarrow\mathbb{R}$ be the function which satisfies for every $x\in\mathbb{R}^d$ that
\begin{equation} \mathbf{d}(x,\mathcal{M}\cap U)=\inf \left\{\abs{x-y}\colon y\in(\mathcal{M}\cap U)\right\},\end{equation}
let $(\cdot)_+\colon\mathbb{R}\rightarrow\mathbb{R}$ be the function which satisfies for every $x\in\mathbb{R}$ that
\begin{equation} (x)_+=\max(x,0),\end{equation}
assume for every $x\in S$ that $\mathbb{R}^d\ni\theta\mapsto F(\theta,x)\in\mathbb{R}$ is a locally Lipschitz continuous function, assume that $f|_U\colon U\rightarrow\mathbb{R}$ is a three times continuously differentiable function, assume for every non-empty compact set $\mathfrak{C}\subseteq U$ that $\sup\nolimits_{\theta\in \mathfrak{C}}\mathbb{E}\big[|F(\theta,X_{1,1})|^2+|(\nabla_\theta F)(\theta,X_{1,1})|^2\big]<\infty$, assume that $\mathcal{M}\cap U$ is a non-empty compact $\mathfrak{d}$-dimensional $\C^1$-submanifold of $\mathbb{R}^d$, assume for every $\theta\in(\mathcal{M}\cap U)$ that $\rank((\Hess f)(\theta))=d-\mathfrak{d}$, for every $M\in\mathbb{N}$, $r\in(0,\infty)$, $\theta\in\mathbb{R}^d$ let $\Theta^{M,r}_{0,\theta}\in\mathbb{R}^d\colon\Omega\rightarrow\mathbb{R}^d$ satisfy for every $\omega\in\Omega$ that $\theta^{M,r}_{0,\theta}(\omega)=\theta$, for every $n,M\in\mathbb{N}$, $r\in(0,\infty)$, $\theta\in\mathbb{R}^d$ let $\Theta^{M,r}_{n,\theta}\colon\Omega\rightarrow\mathbb{R}^d$ satisfy that
\begin{equation}\Theta^{M,r}_{n,\theta}=\Theta^{M,r}_{n-1,\theta}-\frac{r}{n^\rho M}\!\left[\sum_{m=1}^M(\nabla_\theta F)(\Theta^{M,r}_{n-1,\theta},X_{n,m})\right],\end{equation}
and for every $n,M\in\mathbb{N}$, $r,R,\delta\in(0,\infty)$, $\theta\in\mathbb{R}^d$, $x_0\in(\mathcal{M}\cap U)$ let $A_n(M,r,R,\delta,\theta,x_0)\in\mathcal{F}$ satisfy that
\begin{equation}A_n(M,r,R,\delta,\theta,x_0)=\Big\{\forall\;m\in\{0,\ldots,n\}\;\Theta^{M,r}_{m,\theta}\in V_{R,\delta}(x_0)\Big\}.\end{equation}
Then for every $x_0\in(\mathcal{M}\cap U)$ there exist $\delta_0,\mathfrak{r},c\in(0,\infty)$ such that for every $R\in(0,\infty)$, $\delta\in(0,\delta_0]$, $r\in(0,\mathfrak{r}]$, $n,M\in\mathbb{N}$, $\theta\in V_{\nicefrac{R}{2},\delta}(x_0)$ (cf. Definition~\ref{dtbn}) it holds that

\begin{equation}\mathbb{P}[A_n]\geq \prod_{k=1}^n\left(1-\frac{c}{Mk^{2\rho}}\right)_+-cM^{-1}n^{1-\rho}.\end{equation}
\end{prop}

\begin{proof}[Proof of Proposition~\ref{ng_good_set_comp}]The proof is an immediate consequence of Proposition~\ref{ng_good_set} and the compactness of $\mathcal{M}\cap U$.\end{proof}

The following theorem proves the convergence of SGD with initial data sampled from a uniform distribution on a bounded open set $A\subseteq\mathbb{R}^d$ which satisfies that $\mathcal{M}\cap U\cap A\neq\emptyset$.  The proof is an immediate consequence of Theorem~\ref{ng_one_path}, Proposition~\ref{ng_converge_comp}, and Proposition~\ref{ng_good_set_comp}.

\begin{thm}\label{ng_one_path_comp} Let $d\in\mathbb{N}$, $\mathfrak{d}\in\{ 0, 1, \ldots, d - 1 \}$, $\rho\in(0,1)$, let $\abs{\cdot}\colon\mathbb{R}^d\rightarrow\mathbb{R}$ be the standard norm on $\mathbb{R}^d$, let $U\subseteq\mathbb{R}^d$ be an open set, let $A\subseteq\mathbb{R}^d$ be a bounded open set, let $\lambda\colon\mathcal{B}(\mathbb{R}^d)\rightarrow[0,\infty]$ be the Lebesgue-Borel measure, let $(\Omega,\mathcal{F},\mathbb{P})$ be a probability space, let $(S,\mathcal{S})$ be a measurable space, let $F=(F(\theta,x))_{(\theta,x)\in\mathbb{R}^d\times S}\colon  \mathbb{R}^d\times S\rightarrow\mathbb{R}$ be a measurable function, let $X_{n,m}\colon  \Omega\rightarrow S$, $n,m\in\mathbb{N}$, be i.i.d.\ random variables which satisfy for every $\theta\in\mathbb{R}^d$ that $\mathbb{E}\big[ |F(\theta,X_{1,1})|^2\big]<\infty$, let $f\colon\mathbb{R}^d\rightarrow\mathbb{R}$ be the function which satisfies for every $\theta\in\mathbb{R}^d$ that $f(\theta)=\mathbb{E}\big[F(\theta,X_{1,1})\big]$, let $\mathcal{M}\subseteq\mathbb{R}^d$ satisfy that
\begin{equation}\mathcal{M}=\big\{\theta\in\mathbb{R}^d\colon [f(\theta)=\inf\nolimits_{\vartheta\in\mathbb{R}^d} f(\vartheta)]\big\},\end{equation}
let $\mathbf{d}(\cdot,\mathcal{M}\cap U):\mathbb{R}^d\rightarrow\mathbb{R}$ be the function which satisfies for every $x\in\mathbb{R}^d$ that
\begin{equation} \mathbf{d}(x,\mathcal{M}\cap U)=\inf \left\{\abs{x-y}\colon y\in(\mathcal{M}\cap U)\right\},\end{equation}
let $(\cdot)_+\colon\mathbb{R}\rightarrow\mathbb{R}$ be the function which satisfies for every $x\in\mathbb{R}$ that
\begin{equation} (x)_+=\max(x,0),\end{equation}
assume for every $x\in S$ that $\mathbb{R}^d\ni\theta\mapsto F(\theta,x)\in\mathbb{R}$ is a locally Lipschitz continuous function, assume that $f|_U\colon U\rightarrow\mathbb{R}$ is a three times continuously differentiable function, assume for every non-empty compact set $\mathfrak{C}\subseteq U$ that $\sup\nolimits_{\theta\in \mathfrak{C}}\mathbb{E}\big[|F(\theta,X_{1,1})|^2+|(\nabla_\theta F)(\theta,X_{1,1})|^2\big]<\infty$, assume that $\mathcal{M}\cap U$ is a compact $\mathfrak{d}$-dimensional $\C^1$-submanifold of $\mathbb{R}^d$, assume that $\mathcal{M}\cap U\cap A\neq\emptyset$, assume for every $\theta\in(\mathcal{M}\cap U)$ that $\rank((\Hess f)(\theta))=d-\mathfrak{d}$, for every $M\in\mathbb{N}$, $r\in(0,\infty)$ let $\Theta^{M,r}_0\colon\Omega\rightarrow\mathbb{R}^d$ be continuous uniformly distributed on $A$, assume for every $M\in\mathbb{N}$, $r\in(0,\infty)$ that $\Theta^{M,r}_0$ and $\big(X_{n,m}\big)_{n,m\in\mathbb{N}}$ are independent, and for every $M\in\mathbb{N}$, $r\in(0,\infty)$ let $\Theta^{M,r}_{n,\theta}\colon\Omega\rightarrow\mathbb{R}^d$, $n\in\mathbb{N}$, be random variables which satisfy that
\begin{equation}\Theta^{M,r}_n=\Theta^{M,r}_{n-1}-\frac{r}{n^\rho M}\!\left[\sum_{m=1}^M(\nabla_\theta F)(\Theta^{M,r}_{n-1},X_{n,m})\right].\end{equation}
Then for every $x_0\in(\mathcal{M}\cap U\cap A)$  there exist $\delta_0,\mathfrak{r},c\in(0,\infty)$ such that for every $R\in(0,\infty)$, $\delta\in(0,\delta_0]$, $r\in(0,\mathfrak{r}]$, $n,M\in\mathbb{N}$, $\varepsilon\in(0,1]$ it holds that
\begin{equation}\begin{aligned} & \mathbb{P}\Big(  \mathbf{d}\big(\Theta^{M,r}_n,\mathcal{M}\cap U\big)\geq\varepsilon\Big)  \\ & \leq \frac{\lambda\big(A\backslash V_{\nicefrac{R}{2},\delta}(x_0)\big)}{\lambda(A)} +c\varepsilon^{-2}n^{-\rho}+1- \prod_{k=1}^n\left(1-\frac{c}{Mk^{2\rho}}\right)_++cM^{-1}n^{1-\rho}.\end{aligned}\end{equation}
\end{thm}

\begin{proof}[Proof of Theorem~\ref{ng_one_path_comp}] The proof is an immediate consequence of Theorem~\ref{ng_one_path}, Proposition~\ref{ng_converge_comp}, and Proposition~\ref{ng_good_set_comp}.\end{proof}

The following theorem estimates probability that $K\in\mathbb{N}$ independent solutions of SGD with initial data sampled from a uniform distribution on a compact set $A\subseteq\mathbb{R}^d$ which satisfies that $\mathcal{M}\cap U\cap A$ is non-empty fail to converge to within distance $\varepsilon\in(0,1]$ to the local manifold of minima at time $n\in\mathbb{N}$.  The convergence is measured by minimizing a mini-batch average of the objective function.  The proof is a consequence of Theorem~\ref{ng_one_path_comp} and the arguments leading from Theorem~\ref{ng_one_path} to Theorem~\ref{thm_intro_ng_one_path}.

 \begin{thm}\label{thm_intro_ng_one_path_comp}  Let $d\in\mathbb{N}$, $\mathfrak{d}\in\{ 0, 1, \ldots, d - 1 \}$, $\rho\in(0,1)$, let $\abs{\cdot}\colon\mathbb{R}^d\rightarrow\mathbb{R}$ be the standard norm on $\mathbb{R}^d$, let $U\subseteq\mathbb{R}^d$ be an open set, let $A\subseteq\mathbb{R}^d$ be a bounded open set, let $(\Omega,\mathcal{F},\mathbb{P})$ be a probability space, let $(S,\mathcal{S})$ be a measurable space, let $F=(F(\theta,x))_{(\theta,x)\in\mathbb{R}^d\times S}\colon  \mathbb{R}^d\times S\rightarrow\mathbb{R}$ be a measurable function, let $X_{n,m}\colon  \Omega\rightarrow S$, $n,m\in\mathbb{N}$, be i.i.d.\ random variables which satisfy for every $\theta\in\mathbb{R}^d$ that $\mathbb{E}\big[ |F(\theta,X_{1,1})|^2\big]<\infty$, let $f\colon\mathbb{R}^d\rightarrow\mathbb{R}$ be the function which satisfies for every $\theta\in\mathbb{R}^d$ that $f(\theta)=\mathbb{E}\big[F(\theta,X_{1,1})\big]$, let $\mathcal{M}\subseteq\mathbb{R}^d$ satisfy that
\begin{equation}\mathcal{M}=\big\{\theta\in\mathbb{R}^d\colon [f(\theta)=\inf\nolimits_{\vartheta\in\mathbb{R}^d} f(\vartheta)]\big\},\end{equation}
let $(\cdot)_+\colon\mathbb{R}\rightarrow\mathbb{R}$ be the function which satisfies for every $x\in\mathbb{R}^d$ that
\begin{equation}(x)_+=\max(0,x),\end{equation}
assume for every $x\in S$ that $\mathbb{R}^d\ni\theta\mapsto F(\theta,x)\in\mathbb{R}$ is a locally Lipschitz continuous function, assume that $f|_U\colon U\rightarrow\mathbb{R}$ is a three times continuously differentiable function, assume for every non-empty compact set $\mathfrak{C}\subseteq U$ that $\sup\nolimits_{\theta\in \mathfrak{C}}\mathbb{E}\big[|F(\theta,X_{1,1})|^2+|(\nabla_\theta F)(\theta,X_{1,1})|^2\big]<\infty$, assume that $\mathcal{M}\cap U$ is a compact $\mathfrak{d}$-dimensional $\C^1$-submanifold of $\mathbb{R}^d$, assume that $\mathcal{M}\cap U\cap A\neq\emptyset$, assume for every $\theta\in(\mathcal{M}\cap U)$ that $\rank((\Hess f)(\theta))=d-\mathfrak{d}$, for every $n\in\mathbb{N}_0$, $M\in\mathbb{N}$, $r\in(0,\infty)$ let $\Theta^{k,M,r}_{n}\colon\Omega\rightarrow\mathbb{R}^d$, $k\in\mathbb{N}$, be i.i.d.\ random variables, assume for every $n,M\in\mathbb{N}$, $r\in(0,\infty)$ that $(\Theta^{k,M,r}_{n-1})_{k\in\{2,3,\ldots\}}$ and $(X_{n,k})_{k\in\mathbb{N}}$ are independent, assume for every $M\in\mathbb{N}$, $r\in(0,\infty)$ that $\Theta^{1,M,r}_{0}$ is continuous uniformly distributed on $A$, assume for every $M\in\mathbb{N}$, $r\in(0,\infty)$ that $\Theta^{M,r}_0$ and $\big(X_{n,m}\big)_{n,m\in\mathbb{N}}$ are independent, assume for every $n,M\in\mathbb{N}$, $r\in(0,\infty)$ that
\begin{equation}\Theta^{1,M,r}_{n}=\Theta^{1,M,r}_{n-1}-\frac{r}{n^\rho M}\!\left[\sum_{m=1}^M(\nabla_\theta F)(\Theta^{1,M,r}_{n-1},X_{n,m})\right],\end{equation}
and for every $n,M,\mathfrak{M},K\in\mathbb{N}$, $r\in(0,\infty)$ let $\varTheta^{K,M,\mathfrak{M},r}_n\colon\Omega\rightarrow\mathbb{R}^d$ be a random variable which satisfies that
\begin{equation}\sum_{m=1}^\mathfrak{M}F(\varTheta^{K,M,\mathfrak{M},r}_n, X_{n+1,m})=\Big[\min_{k\in\{1,2,\ldots,K\}}\Big(\sum_{m=1}^\mathfrak{M}F(\Theta^{k,M,r}_n, X_{n+1,m})\Big)\Big].\end{equation}
Then for every $x_0\in(\mathcal{M}\cap U\cap A)$  there exist $\delta_0,\mathfrak{r},c\in(0,\infty)$ such that for every $R\in(0,\infty)$, $\delta\in(0,\delta_0]$, $r\in(0,\mathfrak{r}]$, $n,M,K\in\mathbb{N}$, $\varepsilon\in(0,1]$ it holds that
\begin{equation}\begin{aligned} & \mathbb{P}\Big( \Big[f(\varTheta^{K,M,\mathfrak{M},r}_n)-\inf\nolimits_{\theta\in\mathbb{R}^d}f(\theta)\Big]\geq\varepsilon\Big) \\ &  \leq \frac{cK}{\varepsilon^2\mathfrak{M}} + \left(\frac{\lambda\big(A\backslash V_{\nicefrac{R}{2},\delta}(x_0)\big)}{\lambda(A)} +c\varepsilon^{-2}n^{-\rho}+1- \prod_{k=1}^n\left(1-\frac{c}{Mk^{2\rho}}\right)_++cM^{-1}n^{1-\rho}\right)^K.\end{aligned}\end{equation}
\end{thm}

\begin{proof}[Proof of Theorem~\ref{thm_intro_ng_one_path_comp}] The proof is an immediate consequence of Theorem~\ref{ng_one_path_comp}, Theorem~\ref{ng_one_path}, and Theorem~\ref{thm_intro_ng_one_path}.\end{proof}

In the final proposition of this section, we prove that the computation efficiency of the SGD algorithm proposed in Theorem~\ref{thm_intro_ng_one_path_comp} is improved by the compactness of $\mathcal{M}\cap U$.  The improvement is due to the fact that the mini-batch size $M\in\mathbb{N}$ can be chosen smaller in the compact case, since the mini-batch size no longer needs to account for the possibility that SGD leaves a basin of attraction along directions tangential to the local manifold of minima.

 \begin{cor}\label{cor_computation_comp}  Let $d\in\mathbb{N}$, $\mathfrak{d}\in\{ 0, 1, \ldots, d - 1 \}$, $\rho\in(0,1)$, let $U\subseteq\mathbb{R}^d$ be an open set, let $A\subseteq\mathbb{R}^d$ be a bounded open set, let $(\Omega,\mathcal{F},\mathbb{P})$ be a probability space, let $(S,\mathcal{S})$ be a measurable space, let $F=(F(\theta,x))_{(\theta,x)\in\mathbb{R}^d\times S}\colon  \mathbb{R}^d\times S\rightarrow\mathbb{R}$ be a measurable function, let $X_{n,m}\colon  \Omega\rightarrow S$, $n,m\in\mathbb{N}$, be i.i.d.\ random variables which satisfy for every $\theta\in\mathbb{R}^d$ that $\mathbb{E}\big[ |F(\theta,X_{1,1})|^2\big]<\infty$, let $f\colon\mathbb{R}^d\rightarrow\mathbb{R}$ be the function which satisfies for every $\theta\in\mathbb{R}^d$ that $f(\theta)=\mathbb{E}\big[F(\theta,X_{1,1})\big]$, let $\mathcal{M}\subseteq\mathbb{R}^d$ satisfy that
\begin{equation}\mathcal{M}=\big\{\theta\in\mathbb{R}^d\colon [f(\theta)=\inf\nolimits_{\vartheta\in\mathbb{R}^d} f(\vartheta)]\big\},\end{equation}
assume for every $x\in S$ that $\mathbb{R}^d\ni\theta\mapsto F(\theta,x)\in\mathbb{R}$ is a locally Lipschitz continuous function, assume that $f|_U\colon U\rightarrow\mathbb{R}$ is a three times continuously differentiable function, assume for every non-empty compact set $\mathfrak{C}\subseteq U$ that $\sup\nolimits_{\theta\in \mathfrak{C}}\mathbb{E}\big[|F(\theta,X_{1,1})|^2+|(\nabla_\theta F)(\theta,X_{1,1})|^2\big]<\infty$, assume that $\mathcal{M}\cap U$ is a compact $\mathfrak{d}$-dimensional $\C^1$-submanifold of $\mathbb{R}^d$, assume that $\mathcal{M}\cap U\cap A\neq\emptyset$, assume for every $\theta\in(\mathcal{M}\cap U)$ that $\rank((\Hess f)(\theta))=d-\mathfrak{d}$, for every $n\in\mathbb{N}_0$, $M\in\mathbb{N}$, $r\in(0,\infty)$ let $\Theta^{k,M,r}_{n}\colon\Omega\rightarrow\mathbb{R}^d$, $k\in\mathbb{N}$, be i.i.d.\ random variables, assume for every $n,M\in\mathbb{N}$, $r\in(0,\infty)$ that $(\Theta^{k,M,r}_{n-1})_{k\in\{2,3,\ldots\}}$ and $(X_{n,k})_{k\in\mathbb{N}}$ are independent, assume for every $M\in\mathbb{N}$, $r\in(0,\infty)$ that $\Theta^{1,M,r}_{0}$ is continuous uniformly distributed on $A$, assume for every $M\in\mathbb{N}$, $r\in(0,\infty)$ that $\Theta^{M,r}_0$ and $\big(X_{n,m}\big)_{n,m\in\mathbb{N}}$ are independent, assume for every $n,M\in\mathbb{N}$, $r\in(0,\infty)$ that
\begin{equation}\Theta^{1,M,r}_{n}=\Theta^{1,M,r}_{n-1}-\frac{r}{n^\rho M}\!\left[\sum_{m=1}^M(\nabla_\theta F)(\Theta^{1,M,r}_{n-1},X_{n,m})\right],\end{equation}
and for every $n,M,\mathfrak{M},K\in\mathbb{N}$, $r\in(0,\infty)$ let $\varTheta^{K,M,\mathfrak{M},r}_n\colon\Omega\rightarrow\mathbb{R}^d$ be a random variable which satisfies that
\begin{equation}\sum_{m=1}^\mathfrak{M}F(\varTheta^{K,M,\mathfrak{M},r}_n, X_{n+1,m})=\Big[\min_{k\in\{1,2,\ldots,K\}}\Big(\sum_{m=1}^\mathfrak{M}F(\Theta^{k,M,r}_n, X_{n+1,m})\Big)\Big].\end{equation}
Then for every $x_0\in(\mathcal{M}\cap U\cap A)$ there exist $R_0,\delta_0,\mathfrak{r}\in(0,\infty)$ such that for every $R\in(0,R_0]$, $\delta\in(0,\delta_0]$, $r\in(0,\mathfrak{r}]$ there exist $c_i\in(0,\infty)$, $i\in\{1,2,3,4\}$, such that for every $\varepsilon,\eta\in(0,1]$, for $n(\varepsilon), M(\varepsilon), K(\eta), \mathfrak{M}(\varepsilon,\eta)\in\mathbb{N}$ which satisfy that
\begin{equation} n(\varepsilon)=c_1\varepsilon^{-\nicefrac{2}{\rho}},\;\; M(\varepsilon)=c_2\varepsilon^{-\nicefrac{2}{\rho}+2},\;\; \mathfrak{M}(\varepsilon,\eta)=c_3\varepsilon^{-2}\eta^{-1}\abs{\log(\eta)},\;\;\textrm{and}\;\;K=c_4\abs{\log(\eta)},\end{equation}
it holds that
\begin{equation}\mathbb{P}\Big( \Big[f(\varTheta^{K(\eta),M(\varepsilon),\mathfrak{M}(\varepsilon,\eta),r}_{n(\varepsilon)})-\inf\nolimits_{\theta\in\mathbb{R}^d}f(\theta)\Big]\geq\varepsilon\Big) \leq \eta.\end{equation}
\end{cor}

\begin{proof}[Proof of Corollary~\ref{cor_computation_comp}]  The proof is an immediate consequence of Theorem~\ref{thm_intro_ng_one_path_comp} and the proof of Corollary~\ref{cor_computation}.\end{proof}

\section{Applications}\label{sec_applications}

In this section, we prove that the conditions of Theorem~\ref{intro_ng_one_path} are satisfied for some (simple) objective functions $f\colon  \mathbb{R}^d\rightarrow\mathbb{R}$ of the type \eqref{intro_aan} that arise in the training of neural networks.  We will consider the case of a four-parameter affine-linear network with a linear activation function and the case of a two-parameter network with the ReLU activation function.  We will prove that the set of global minima are respectively a codimension $2$ submanifold of the parameter space, and a codimension $1$ submanifold.  This implies, in particular, that the global minima are not locally unique, and that the established convergence results, such as those proven in \cite{DMG15,JKNW18}, do not apply.

\subsection{A four-parameter network with a linear activation function}\label{examples}

In this section, we show that the conditions of Theorem~\ref{intro_ng_one_path} are satisfied by a four-parameter affine-linear network with a linear activation function.

\begin{prop}\label{pair_assumption}  Let $\varphi\in L^2([0,1])$ be finite, let $(\Omega,\mathcal{F},\mathbb{P})$ be a probability space, let $X_{n,m}\colon\Omega\rightarrow[0,1]$, $n,m\in\mathbb{N}$, be i.i.d. random variables that are continuous uniformly distributed on $[0,1]$, let $f\colon\mathbb{R}^4\rightarrow\mathbb{R}$ be the function which satisfies for every $\theta=(\theta_1,\theta_2,\theta_3,\theta_4)\in\mathbb{R}^4$ that 
\begin{equation}f(\theta)=\int_0^1\abs{\theta_3\theta_1x+\theta_3\theta_2+\theta_4-\varphi(x)}^2\dx,\end{equation}
and let $F\colon\mathbb{R}^4\times[0,1]\rightarrow\mathbb{R}$ be the function that satisfies for every $\theta\in\mathbb{R}^4$, $x\in[0,1]$ that
\begin{equation} F(\theta,x)=\abs{\theta_3\theta_1x+\theta_3\theta_2+\theta_4-\varphi(x)}^2.\end{equation}
Then the functions $f$, $F$ and the random variables $X_{n,m}$, $n,m\in\mathbb{N}$, satisfy the conditions of Theorem~\ref{intro_ng_one_path}. \end{prop}

\begin{proof}[Proof of Proposition~\ref{pair_assumption}]  Let $\varphi\in L^2([0,1])$ be finite.  The finiteness of $\varphi$ proves that, for every $x\in[0,1]$, we have $F(\cdot,x)\in \C^{0,1}_\textrm{loc}(\mathbb{R}^4)$.  It follows by the uniform distribution of the $X_{n,m}$, $n,m\in\mathbb{N}$, on $[0,1]$ that $f(\cdot)=\mathbb{E}[F(\cdot,X_{1,1})]$, and it follows from the $L^2$-integrability of $\varphi$ that for every compact subset $\mathfrak{C}\subseteq\mathbb{R}^4$ it holds that
\begin{equation}\sup_{\theta\in \mathfrak{C}}\mathbb{E}\left[\abs{F(\theta,X_{1,1})}^2+\abs{\nabla_\theta F(\theta,X_{1,1})}^2\right]< \infty.\end{equation}
It follows by the definition of $f$ and $\varphi\in L^2([0,1])$ that $f\in \C^3_{\textrm{loc}}(\mathbb{R}^4)$.  It remains to characterize the set of minima of $f$.  We first observe that when minimizing $f$, it is sufficient to minimize the potential over the set $\{\theta_3\neq 0\}$.  To see this, suppose that $\theta=(\theta_1,\theta_2,0,\theta_4)$.  Then for $\tilde{\theta}=(0,0,1,\theta_4)$ it holds that
\begin{equation}f(\theta)=\int_0^1\abs{\theta_4-\varphi(x)}^2\dx=f(\tilde{\theta}).\end{equation}
Therefore, it holds that
\begin{equation}\label{lin_constant_min}\inf_{\theta\in\mathbb{R}^4}f(\theta)=\inf_{\theta\in\{\theta_3\neq 0\}}f(\theta).\end{equation}
Let $\theta\in\mathbb{R}^4\cap\{\theta_3\neq 0\}$ be fixed but arbitrary.  An explicit computation proves the critical points of $f$ satisfy that
\begin{equation}\label{lin_grad}\nabla f(\theta)=2\int_0^1\left(\theta_3\theta_1x+\theta_3\theta_2+\theta_4-\varphi(x)\right)\left(\begin{aligned} & \theta_3x \\ & \theta_3 \\ & \theta_1x+\theta_2 \\ & 1 \end{aligned}\right)\dx=0.\end{equation}
For $r_k\in\mathbb{R}$, $k\in\{0,1\}$, which satisfy that
\begin{equation} r_k =\int_0^1x^k\varphi(x)\dx,\end{equation}
it follows that $\theta\in\mathbb{R}^4$ satisfies equation \eqref{lin_grad} if and only if it holds that
\begin{equation}\label{lin_system}\left\{\begin{aligned} & \frac{1}{3}\theta_1\theta_3^2+\frac{1}{2}\theta_2\theta_3^2+\frac{1}{2}\theta_3\theta_4- r_1\theta_3=0, \\ & \frac{1}{2}\theta_1\theta_3^2+\theta_2\theta_3^2+\theta_3\theta_4- r_0\theta_3=0, \\ & \frac{1}{3}\theta^2_1\theta_3+\frac{1}{2}\theta_1\theta_2\theta_3+\frac{1}{2}\theta_1\theta_4- r_1\theta_1+\frac{1}{2}\theta_1\theta_2\theta_3+\theta^2_2\theta_3+\theta_2\theta_4- r_0\theta_2=0, \\ & \frac{1}{2}\theta_1\theta_3+\theta_2\theta_3+\theta_4- r_0=0.\end{aligned}\right.\end{equation}
For $\theta\in\mathbb{R}^4$ which satisfies that $\theta_3\neq 0$, an explicit computation proves that $\theta$ satisfies system \eqref{lin_system} if and only if it holds that
\begin{equation}\label{lin_nonzero}\theta_1\theta_3=-6( r_0-2 r_1)\;\;\textrm{and}\;\;\theta_4=-\theta_2\theta_3+4 r_0-6 r_1.\end{equation}
For $U\subseteq\mathbb{R}^4$ which satisfies that
\begin{equation} U=\{\theta\in\mathbb{R}^4\colon\theta_3\neq 0\},\end{equation}
for $\mathcal{M}\subseteq\mathbb{R}^4$ which satisfies that
\begin{equation}\mathcal{M}=\{\theta\in\mathbb{R}^4\colon f(\theta)=\inf\nolimits_{\vartheta\in\mathbb{R}^4}f(\vartheta)\},\end{equation}
we claim that
\begin{equation}\label{lin_nonzero_2}\mathcal{M}\cap U=\{\;\theta\in\mathbb{R}^4\colon \theta\;\textrm{satisfies \eqref{lin_nonzero} and}\;\theta_3\neq 0\;\}.\end{equation}
Let $\theta\in\mathbb{R}^4$ satisfy \eqref{lin_nonzero} and $\theta_3\neq 0$.  Proceeding by contradiction, suppose that there exists $\theta_0=(\theta_{1,0}, \theta_{2,0}, \theta_{3,0}, \theta_{4,0})$ which satisfies $\theta_{3,0}\neq 0$ such that
\begin{equation}f(\theta_0)<f(\theta).\end{equation}
Since an explicit computation proves for every $(\theta_1,\theta_4)\in\mathbb{R}^2$ that
\begin{equation}\lim_{\abs{(\theta_1,\theta_4)}\rightarrow\infty}f(\theta_1,\theta_{2,0},\theta_{3,0},\theta_4)=\infty,\end{equation}
the identical considerations leading to \eqref{lin_nonzero} prove that
\begin{equation}(\theta_1,\theta_4)\in\mathbb{R}^2\mapsto f(\theta_1,\theta_{2,0},\theta_{3,0},\theta_4),\end{equation}
is uniquely minimized, owing to $\theta_{3,0}\neq 0$, by $(\theta_1,\theta_4)\in\mathbb{R}^2$ which satisfies that
\begin{equation}\label{lin_nonzero_1}\theta_1=-\frac{6( r_0-2 r_1)}{\theta_{3,0}}\;\;\textrm{and}\;\;\theta_4=-\theta_{2,0}\theta_{3,0}+4 r_0+6 r_1.\end{equation}
We conclude that $\tilde{\theta}_0\in\mathbb{R}^4$ satisfies that
\begin{equation}\tilde{\theta}_0=(-\frac{6( r_0-2 r_1)}{\theta_{3,0}},\theta_{2,0}, \theta_{3,0}, -\theta_{2,0}\theta_{3,0}+4 r_0+6 r_1),\end{equation}
satisfies \eqref{lin_nonzero} and $\tilde{\theta}_{3,0}\neq 0$.  Therefore, it holds that
\begin{equation}f(\tilde{\theta}_0)< f(\theta_0),\end{equation}
which contradicts the fact that $\nabla f=0$ on the connected set of $\theta\in\mathbb{R}^4$ which satisfies \eqref{lin_nonzero} and $\theta_3\neq 0$.  This proves \eqref{lin_nonzero_2}.  It is immediate from \eqref{lin_nonzero} that $\mathcal{M}\cap U$ is a non-empty, $2$-dimensional, $\C^1$-submanifold of $\mathbb{R}^4$.  It remains only to prove the nondegeneracy assumption.  for every $\theta\in(\mathcal{M}\cap U)$, after computing the Hessian\footnote{Due to the symmetry of the Hessian, we only write the upper diagonal.}, it holds that
\begin{equation}\begin{aligned} \big(\Hess f\big)(\theta) & =  2\int_0^1\left(\begin{aligned} \theta_3^2x^2 && \theta_3^2x && \theta_1\theta_3x^2+\theta_2\theta_3x && \theta_3x \\ && \theta_3^2 && \theta_1\theta_3x+\theta_2\theta_3 && \theta_3 \\ && && (\theta_1x+\theta_2)^2 && \theta_1x+\theta_2 \\ && && && 1\end{aligned}\right)\dx \\ & =   \left(\begin{aligned} \frac{2}{3}\theta_3^2 && \theta_3^2 && \frac{2}{3}\theta_1\theta_3+\theta_2\theta_3 && \theta_3 \\ && 2\theta_3^2 && \theta_1\theta_3+2\theta_2\theta_3 && 2\theta_3 \\ && && \frac{2}{3}\theta_1^2+2\theta_1\theta_2+2\theta_2^2 && \theta_1+2\theta_2 \\ && && && 2 \end{aligned} \right),\end{aligned}\end{equation}
where this equality relies upon the fact that, due to \eqref{lin_grad} and $\theta_3\neq 0$ on $\mathcal{M}\cap U$, we have that
\begin{equation}\int_0^1(\theta_3\theta_1x+\theta_3\theta_2+\theta_4-\varphi(x))\dx=\int_0^1(\theta_3\theta_1x+\theta_3\theta_2+\theta_4-\varphi(x))x\dx=0.\end{equation}
A column-reduction, which relies on the fact that for every $\theta\in(\mathcal{M}\cap U)$ we have $\theta_3\neq 0$, proves for every $\theta\in(\mathcal{M}\cap U)$ that
\begin{equation}\textrm{rank}((\Hess f)(\theta))=2=\codim(\mathcal{M}\cap U).\end{equation}
This completes the proof of Proposition~\ref{pair_assumption}. \end{proof}

\vspace{.1cm}

\subsection{A two parameter network with the ReLU activation function}~\label{two_parameter}

In this section, we show that the conditions of Theorem~\ref{intro_ng_one_path} are satisfied by a two-parameter affine-linear network with the ReLU activation function.

\begin{prop}\label{prop_ReLU}  Let $(\Omega,\mathcal{F},\mathbb{P})$ be a probability space, let $X_{n,m}\colon\Omega\rightarrow[0,1]$, $n,m\in\mathbb{N}$, be i.i.d. random variables that are continuous uniformly distributed on $[0,1]$, let $f\colon\mathbb{R}^2\rightarrow\mathbb{R}$ be the function which satisfies for every $\theta=(\theta_1,\theta_2)\in\mathbb{R}^2$ that 
\begin{equation}f(\theta)=\int_0^1\abs{\theta_2 \max(\theta_1x,0)-\sin(x)}^2\dx,\end{equation}
and let $F\colon\mathbb{R}^2\times[0,1]\rightarrow\mathbb{R}$ be the function which satisfies for every $\theta\in\mathbb{R}^2$, $x\in[0,1]$ that
\begin{equation} F(\theta,x)=\abs{\theta_2 \max(\theta_1x,0)-\sin(x)}^2.\end{equation}
Then the functions $f$, $F$ and the random variables $X_{n,m}$, $n,m\in\mathbb{N}$, satisfy the conditions of Theorem~\ref{intro_ng_one_path}. \end{prop}

\begin{proof}[Proof of Proposition~\ref{prop_ReLU}] It is immediate that $F(\cdot,x)\in\C^{0,1}_{\textrm{loc}}(\mathbb{R}^2)$.  Since the $X_{n,m}$, $n,m\in\mathbb{N}$ are uniformly distributed on $[0,1]$, for every $\theta\in\mathbb{R}^2$ it holds that
\begin{equation}f(\theta)=\mathbb{E}[F(\theta,X_{1,1})],\end{equation}
and, furthermore, a straightforward computation proves for every compact set $\mathfrak{C}\subseteq\mathbb{R}^2$ that
\begin{equation}\sup_{\theta\in \mathfrak{C}}\mathbb{E}\left[\abs{F(\theta,X_{1,1})}^2+\abs{\nabla_\theta F(\theta,X_{1,1})}^2\right]<\infty.\end{equation}
It remains only to characterize the minima of the objective function, and to verify the nondegeneracy condition.  An explicit computation proves that, when minimizing $f$, it is sufficient to restrict to the set $\{\theta_1>0, \theta_2>0\}$.  Let $U\subseteq\mathbb{R}^2$ satisfy that
\begin{equation} U=\{\theta\in\mathbb{R}^2\colon \theta_1>0, \theta_2>0\}.\end{equation}
We observe for every $\theta\in U$ that
\begin{equation}\label{two_pot_2}f(\theta)=\int_0^1\abs{\theta_1\theta_2 x-\sin(x)}^2\dx,\end{equation}
and for every $\theta\in U$ that
\begin{equation}\nabla f(\theta)=2\int_0^1(\theta_1\theta_2x-\sin(x))\left(\begin{aligned} & \theta_2x \\ &\theta_1 x\end{aligned}\right)\dx.\end{equation}
Therefore, for $\theta\in U$ it holds that $\nabla f(\theta)=0$ if and only if it holds that
\begin{equation}\label{two_pot_20}\theta_1\theta_2=3\int_0^1x\sin(x)\dx=3(\sin(1)-\cos(1)).\end{equation}
Let $\mathcal{M}\subseteq\mathbb{R}^2$ satisfy that
\begin{equation}\mathcal{M}=\{\theta\in\mathbb{R}^2\colon f(\theta)=\inf\nolimits_{\vartheta\in\mathbb{R}^4}f(\vartheta)\}.\end{equation}
We claim that
\begin{equation}\label{lin_nonzero_200}\mathcal{M}\cap U=\{\;\theta\in\mathbb{R}^2\colon \theta\;\textrm{satisfies \eqref{two_pot_20}},\; \theta_1>0,\; \textrm{and}\; \theta_2>0\}.\end{equation}
Suppose that $\theta\in U$ satisfies \eqref{two_pot_20}.  By contradiction suppose that there exists $\theta_0=(\theta_{1,0},\theta_{2,0})\in\{\theta_1>0,\theta_2>0\}$ such that
\begin{equation}f(\theta_0)<f(\theta).\end{equation}
Since $\theta_{1,0}>0$ an explicit computation proves that
\begin{equation}\label{two_pot_21}\lim_{\theta_2\rightarrow\infty}f(\theta_{1,0},\theta_2)=+\infty\;\;\textrm{and}\;\;f(\theta_{1,0},0)>f(\theta_0).\end{equation}
The arguments leading from \eqref{two_pot_2} to \eqref{two_pot_20} prove that \eqref{two_pot_21} is uniquely minimized when
\begin{equation}\theta_2=\frac{3}{\theta_{1,0}}(\sin(1)-\cos(1)).\end{equation}
Therefore, for $\tilde{\theta}_0\in\mathbb{R}^2$ which satisfies that
\begin{equation}\tilde{\theta}_0=(\theta_{1,0}, \frac{3}{\theta_{1,0}}(\sin(1)-\cos(1))),\end{equation}
we have that $\tilde{\theta}_0\in U$, that $\tilde{\theta}_0$ satisfies \eqref{two_pot_20}, and that
\begin{equation}f(\tilde{\theta}_0)\leq f(\theta_0)<f(\theta).\end{equation}
This contradicts the fact that $\nabla f=0$ on the connected set of $\theta\in U$ that satisfy \eqref{two_pot_20}.  This proves \eqref{lin_nonzero_200}.  Since it is clear that $\mathcal{M}\cap U$ is a non-empty, $1$-dimensional, $\C^1$-submanifold of $\mathbb{R}^2$, it remains only to establish the nondegeneracy assumption.
For every $\theta\in(\mathcal{M}\cap U)$ it holds that
\begin{equation}\begin{aligned} \big(\Hess f\big)(\theta) & =2\left(\begin{aligned} & \frac{1}{3}\theta_2^2 & \frac{2}{3}\theta_1\theta_2-(\sin(1)-\cos(1)) \\ &  & \frac{1}{3}\theta_1^2 \end{aligned}\right) \\ & =2\left(\begin{aligned} & \frac{1}{3}\theta_2^2 & \sin(1)-\cos(1) \\ & & \frac{3(\sin(1)-\cos(1))^2}{\theta_2^2} \end{aligned}\right).\end{aligned}\end{equation}
A column reduction and $\theta_2\neq 0$ prove for every $\theta\in(\mathcal{M}\cap U)$ that
\begin{equation}\textrm{rank}((\Hess f)(\theta))=1=\codim(\mathcal{M}\cap U).\end{equation}
This completes the proof of Proposition~\ref{prop_ReLU}. \end{proof}

\section*{Acknowledgements}

The first author acknowledges financial support from the National Science Foundation Mathematical Sciences Postdoctoral Research Fellowship under Grant Number 1502731.

The second author acknowledges financial support by the DFG through the CRC 1283 ``Taming uncertainty and profiting from randomness and low regularity in analysis, stochastics and their applications.''

\bibliography{Exit}
\bibliographystyle{plain}

\end{document}